\documentclass{article}
\usepackage{amsmath}
\usepackage{amsfonts}
\usepackage{amsthm}
\usepackage{amssymb}
\usepackage{enumerate}
\usepackage{ytableau}

\newtheorem{thm}{Theorem}[section]
\newtheorem{prop}[thm]{Proposition}
\newtheorem{lem}[thm]{Lemma}
\newtheorem{cor}[thm]{Corollary}

\theoremstyle{definition}
\newtheorem{defn}[thm]{Definition}
\newtheorem{rmk}[thm]{Remark}
\newtheorem{ex}[thm]{Example}

\begin{document}

\title{Generalized Higher Specht Polynomials and Homogeneous Representations of Symmetric Groups}

\author{Shaul Zemel}

\maketitle


\section*{Introduction}

Consider, inside the ring $\mathbb{Q}[\mathbf{x}_{n}]$ of polynomials in $n$ variables (over $\mathbb{Q}$), the part $\mathbb{Q}[\mathbf{x}_{n}]_{d}$ that is homogeneous of degree $d$, as a representation of the symmetric group $S_{n}$. When asking how many times the irreducible representation associated with a partition $\lambda \vdash n$ (also known as the Specht module $\mathcal{S}^{\lambda}$) shows up in it, one finds two answers, both of which following quite quickly from classical considerations appearing in \cite{[St2]}.

One answer is the number of semi-standard Young tableaux of shape $\lambda$ and entry sum $d$, as seen in Proposition \ref{Kostka} below. For the other formula, one counts pairs consisting of cocharge tableaux (or equivalently standard Young tableaux) of shape $\lambda$ and multi-sets whose elements sum to $d$ with an appropriate containment condition---see Proposition \ref{Adlambda} below for the precise formulation. One goal of this paper is to ``lift'' both formulae into direct sums of explicit copies of each Specht module $\mathcal{S}^{\lambda}$, whose sum over $\lambda \vdash n$ yields the full representation in question.

\medskip

We recall from \cite{[B]} that the ring $R_{n}$ of co-invariants of the action of $S_{n}$ on $\mathbb{Q}[\mathbf{x}_{n}]$ is isomorphic, as a graded ring, to the cohomology ring of the flag variety (from this fact stems the theory of Schubert polynomials, as well as other vast and beautiful theories in Algebraic Geometry and Algebraic Combinatorics). As a representation of $S_{n}$, it is isomorphic to the regular one. Decomposing it into irreducible components is carried out using the \emph{higher Specht polynomials}, constructed in \cite{[ATY]}, generalizing the definition of the classical Specht polynomials from \cite{[Pe]} and others.

The paper \cite{[HRS]} considered more general quotients $R_{n,k}$, for $1 \leq k \leq n$, and show that they are related to the action of $S_{n}$ on ordered partitions of the set $\mathbb{N}_{n}$ of integers between 1 and $n$ into $k$ sets. The paper \cite{[GR]} constructs higher Specht polynomial bases for these representations, and in the prequel \cite{[Z1]} to the current paper we decomposed $R_{n,k}$ into sub-representations associated with the orbits of $S_{n}$ in that action, in a way that respects these higher Specht polynomial bases. These orbits are determined by the sizes of the sets, and if the composition of these sizes is associated with a subset $I\subseteq\mathbb{N}_{n-1}$, of size $k-1$, then to the resulting representation $R_{n,I}$ we also constructed a homogeneous counterpart $R_{n,I}^{\mathrm{hom}}$.

\medskip

The quotients $R_{n,k}$ are, in fact, the coarsest in a larger family of quotients $R_{n,k,s}$ for $0 \leq s \leq k$, which may in fact be defined for any $0 \leq s\leq\min\{n,k\}$, without the condition $k \leq n$ (see Remark \ref{kngen} below). These also admit higher Specht polynomial bases as in \cite{[GR]} (in fact, the construction for $R_{n,k}$ in that reference is done by induction on $s$), and the construction from \cite{[PR]} of $R_{n,k}$ as the cohomology ring of some algebraic variety extends also to $R_{n,k,s}$ for every such $n$, $k$, and $s$ (as does the connection with the Delta conjecture presented in \cite{[HRW]}).

In this paper we extend the construction from \cite{[Z1]} to the more general rings $R_{n,k,s}$, and in particular to the quotient $R_{n,k,0}$ of $\mathbb{Q}[\mathbf{x}_{n}]$ by the $k$th powers $x_{i}^{k}$ of all the variables. This is done by allowing some of the sets in the ordered partitions to be empty (Lemma \ref{propOPnks} compares this point of view with the one from \cite{[HRS]}), and then the index set $I$ from \cite{[Z1]} can be a multi-set, and the composition may be a weak composition. We show how for $s=0$ the direct sum of the $R_{n,I}$'s in this setting coincides with that of the homogeneous representations $R_{n,I}^{\mathrm{hom}}$, using which we establish one of the liftings (this is achieved in Theorem \ref{Rnksdecom} and Corollary \ref{multCI} below).

\medskip

The most direct construction of the higher Specht polynomials is based on a standard Young tableau $T$ and a cocharge tableau $C$ of the same shape, as presented in, e.g., Definition \ref{sets} below (these also known, in a slightly different normalization, as quasi-Yamanouchi tableaux---see, e.g., \cite{[AS]} and \cite{[BCDS]}, though for us it is important to allow 0 to show up in such tableaux). However, one observes that the only property of $C$ that is required for the construction to work is it being semi-standard. This allows for constructing a \emph{generalized higher Specht polynomial} $F_{M,T}$ for every such $T$ and every semi-standard Young tableau $M$ of the same shape as $T$, and these again span irreducible representations. We show in Theorem \ref{FMTdecom} below how these can be used to yield the second lifting from above.

The construction of the higher Specht polynomials appears in the literature in various normalizations. We showed in \cite{[Z1]} that the one we work with here is adapted for compatibility when passing from $n$ to $n+1$. This is the same for the generalized ones, as Proposition \ref{forstab} below shows that given $T$ and $M$ as above, there are explicit tableaux $\iota T$ and $\hat{\iota}M$, now with $n+1$ boxes and similar properties, such that substituting $x_{n+1}=0$ in the generalized higher Specht polynomial $F_{\hat{\iota}M,\iota T}$ yields $F_{M,T}$ back again. This allows us to construct stable versions of these polynomials and their representations in infinitely many variables, as we will do in the sequel \cite{[Z2]} to this paper.

\medskip

We now recall that \cite{[HRS]} defines star and bar insertions, which take ordered partitions of $n$ to ordered partitions of $n+1$. The effect of these operations, at the last location, on the representations $R_{n,I}$ and $R_{n,I}^{\mathrm{hom}}$, where $I$ is a set, were investigated in \cite{[Z1]}. Here we extend this analysis to the multi-set case for the homogeneous representations, where the bar operation yields the same induction formula as in \cite{[Z1]}, while the star operation becomes trickier due to the possibly empty last set (see Proposition \ref{mapsmulti} and Remark \ref{nostarbar} below).

We also recall the theory of representation stability from \cite{[CF]} and \cite{[Fa]}, with the relations to FI-modules in \cite{[CEF]}, to other categories in \cite{[SS]}, and to the notion of central stability from \cite{[Pu]}. This notion considers, for each $n$, a representation $U_{n}$ of $S_{n}$, such that $U_{n}$ embeds into $U_{n+1}$ over $\mathbb{Q}[S_{n}]$, and for large enough $n$ it generates it over $\mathbb{Q}[S_{n+1}]$. It follows that for every multi-set $I$, the representations $\{R_{n,I}^{\mathrm{hom}}\}_{n}$ have this property with explicit embeddings (these are based on the maps taking $F_{M,T}$ to $F_{\hat{\iota}M,\iota T}$ as above), using which one can define limits, which will be considered in more detail in \cite{[Z2]} (see Remark \ref{repstab} below). In fact, the representations $\{R_{n,I}\}_{n}$ are also stable for every such $I$, but they produce well-defined limit only when $I$ is a set, as Remark \ref{onlysetsRnI} below explains.

Similar operations are also defined on the full representations $\{\mathbb{Q}[\mathbf{x}_{n}]_{d}\}_{n}$, as well as on their sub-representations that are supported on polynomials of a given content, and they have properties of the same kind, as Theorem \ref{opersVM}, Corollary \ref{ExtVMlim}, and Remark \ref{repstab} below show.

\medskip

This paper is divided into 3 sections. Section \ref{DMSRep} reviews the notions involving weak compositions, multi-sets, ordered partitions with possibly empty sets, and the quotients $R_{n,k,s}$, and presents the two formulae for the multiplicity of $\mathcal{S}^{\lambda}$ in $\mathbb{Q}[\mathbf{x}_{n}]_{d}$. In Section \ref{GenReps} we recall the higher Specht polynomials, define the generalized ones and the associated representations, introduce the representations $R_{n,I}$ and $R_{n,I}^{\mathrm{hom}}$ for multi-sets $I$, and establish the two liftings. Finally, Section \ref{OpersReps} considers the maps between these representations, and investigate the resulting properties.

\medskip

I am grateful to B. Sagan, B. Rhoades, M. Gillespie, and S. van Willigenburg, for their interest, for several encouraging conversations on this subject, and for helpful comments. I am indebted, in particular, to D. Grinberg for a detailed reading of previous versions, for introducing me to several references, and for numerous suggestions which drastically improved the presentation of this paper.

\section{Descents, Multi-Sets, and Representations \label{DMSRep}}

In \cite{[Z1]} we considered several constructions involving compositions of $n$ and subsets of the set $\mathbb{N}_{n-1}$ of integers between 1 and $n-1$, where the subsets may be viewed as increasing sequences. Here we will require more general notions, for which we recall that a \emph{weak composition} of $n$ is a (finite) sequence $\alpha$ of non-negative integers, the sum of which is $n$. We shall write the sequence as $\{\alpha_{h}\}_{h=0}^{k-1}$, where the integer $k$ is the \emph{length} $\ell(\alpha)$ of $\alpha$ (which means that we view weak compositions which differ only by the number of zeros at the end as distinct), and we use the notation $\alpha\vDash_{w}n$ for this situation.

For extending the bijection $\operatorname{comp}_{n}$ between compositions and subsets (or increasing sequences), as showing up in, e.g., Lemma 1.5 of \cite{[Z1]}, we need to consider \emph{multi-sets}, whose elements can be from the larger set $\{0\}\cup\mathbb{N}_{n}$ (including 0 and $n$), and each element can show up with some non-negative (finite) multiplicity. These can be considered as \emph{non-decreasing} sequences of non-negative integers that are bounded by $n$. The \emph{size} of such a multi-set is, as usual, the sum of the multiplicities of its elements.

The extension of Lemma 1.5 to this setting is as follows.
\begin{lem}
Fix $n\geq1$, and consider a multi-set $J$ as above, of size $k-1$, which we write in a non-decreasing as $\{j_{h}\}_{h=1}^{k-1}$, and extend with $j_{0}=0$ and $j_{k}=n$. Then the sequence of integers $\{j_{h+1}-j_{h}\}_{h=0}^{k-1}$ is a weak composition $\operatorname{comp}_{n}J\vDash_{w}n$ of length $k$, and given any $\alpha\vDash_{w}n$ with $\ell(\alpha)=k$, setting $j_{0}=0$ and then $j_{h+1}=j_{h}+\alpha_{h}$ by induction on $h$ produces, after removing $j_{0}$ and $j_{k}=n$, the unique multi-set $J$ with $\alpha=\operatorname{comp}_{n}J$. \label{mswcomp}
\end{lem}
It is clear that the stars-and-bars argument proving Lemma 1.5 of \cite{[Z1]} (as given in, e.g., pages 17--18 of \cite{[St1]}) also establishes Lemma \ref{mswcomp} in this more general setting. If $\alpha\vDash_{w}n$ then the multi-set $J$ with $\alpha=\operatorname{comp}_{n}J$ will be written as $\operatorname{comp}_{n}^{-1}\alpha$.
\begin{rmk}
It is clear that if the multi-set $J$ in Lemma \ref{mswcomp} is a subset of $\mathbb{N}_{n-1}$ then the weak composition $\operatorname{comp}_{n}J$ is the actual composition $\operatorname{comp}_{n}J \vDash n$ from Lemma 1.5 of \cite{[Z1]}, and conversely that for a true composition $\alpha \vDash n$, the multi-set $\operatorname{comp}_{n}^{-1}\alpha$ is a the subset $\operatorname{comp}_{n}^{-1}\alpha\subseteq\mathbb{N}_{n-1}$. Hence our notation extends the classical one. We also note that in \cite{[Z1]} we implicitly assumed that $k \leq n$, since when $k-1$ is the size of a subset of $\mathbb{N}_{n-1}$, this inequality must hold. We do not impose that inequality in general. \label{comptow}
\end{rmk}

In \cite{[Z1]} we used compositions to describe the content of (semi-standard) tableaux, and the generalized cocharge tableaux there had to satisfy the condition that if $k-1$ is the maximal entry showing up in one, then it must contain each of the integers $0 \leq h<k$ at least once. To each weak composition $\alpha\vDash_{w}n$ we attach the \emph{content represented by $\alpha$} to again be the multi-set $\mu$, with entries contained in the set of the integers $0 \leq h<k$, and in which the multiplicity of $h$ in $\mu$ is $\alpha_{h}$. We again have that $\mu$ is of size $n$ and determines $\alpha$, and every multi-set of non-negative integers of size $n$ (and in particular any content of a tableau whose shape is of size $n$ and whose entries are non-negative integers) is the content of some weak composition $\alpha\vDash_{w}n$ (see Definition \ref{SSYTct} below).

\begin{ex}
Let $n=4$, and consider $J$ to be the multi-set $\{1,1,4\}$, for which we have $k=4$. Then $\operatorname{comp}_{4}J=1030\vDash_{w}4$, and the content represented by it is 0222, or $02^{3}$, written as a non-decreasing sequence. Taking the multi-set $\{0,0,1,1,3\}$, now with $k=6$, we get the weak composition $001021\vDash_{w}4$, and the corresponding content is 2445 (or $24^{2}5$), ending in $6-1=5$. Note that the former content did not end in $4-1=3$, because $J$ contained $n=4$ and hence $\operatorname{comp}_{4}J$ ended in a zero. \label{wcompex}
\end{ex}

\medskip

Recall that in \cite{[Z1]} the subsets $J\subseteq\mathbb{N}_{n-1}$, and mainly $I=\{n-i\;|\;i \in J\}$ later, were required to contain various sets obtained from permutations in the symmetric group $S_{n}$, or standard Young tableaux, or cocharge tableaux, and some expressions were given in terms of the complements. In this paper we will carry out similar constructions, but in which $J$ (or $I$) is a multi-set as above. This is to be interpreted as follows.
\begin{defn}
Consider a multi-set $J$, of size $k-1$, whose entries are elements of $\mathbb{N}_{n}\cup\{0\}$. We then define $\hat{J}$ to be the standard subset of $\mathbb{N}_{n-1}$ consisting of those elements there that show up in $J$ with a positive multiplicity (ignoring those of 0 and $n$), and set $\hat{k}:=|\hat{J}|+1$. Given a subset $D\subseteq\mathbb{N}_{n-1}$, we will use an abuse of notation and terminology and say that $D$ is contained in $J$ (or $D \subseteq J$) if $D\subseteq\hat{J}$. In this case define the complement $J \setminus D$ to be the multi-set obtained by considering $J$ and subtracting 1 from the multiplicity of any element of $D$. In particular this is how we define the multi-set complement $J\setminus\hat{J}$. \label{multisets}
\end{defn}
It is clear that if the multi-set $J$ in Definition \ref{multisets} is an actual subset of $\mathbb{N}_{n-1}$ (so that $\hat{J}=J$), then the notions from that definition, including the complement, are the usual ones (and $J\setminus\hat{J}$ is just empty).

As an example to how this works, we extend the destandardization from Lemma 1.8 of \cite{[Z1]}, which is based on Definition A1.5.5 of \cite{[St2]} or the end of Section 2.1 of \cite{[vL]}, to a more general setting. For this we recall a few definitions from \cite{[Z1]} that will also be used in this paper. Recall that $\lambda \vdash n$ denoted that $\lambda$ is a partition of $n$ (with the same length $\ell(\lambda)$ as a composition), and that $\operatorname{SYT}(\lambda)$ stands for the set of standard Young tableaux of shape $\lambda$, with $\operatorname{sh}(T)$ denoting the shape of a tableau $T$ (as a Ferrers diagram or as a partition, which we identify with one another).

\begin{defn}
Let $\lambda \vdash n$ be a partition, and $d\geq0$.
\begin{enumerate}[$(i)$]
\item Set $\operatorname{SSYT}(\lambda)$ to be the set of semi-standard Young tableaux of shape $\lambda$.
\item The entry sum of $M\in\operatorname{SSYT}(\lambda)$ is denoted by $\Sigma(M)$.
\item The set of $M\in\operatorname{SSYT}(\lambda)$ with $\Sigma(M)=d$ will be denoted by $\operatorname{SSYT}_{d}(\lambda)$.
\item The \emph{content} of $M$ is the multi-set of entries showing up in $M$.
\item For any content $\mu$, the set of $M\in\operatorname{SSYT}(\lambda)$ whose content is $\mu$ is denoted by $\operatorname{SSYT}_{\mu}(\lambda)$.
\end{enumerate} \label{defSSYT}
\end{defn}
It is clear from Definition \ref{defSSYT} that $\operatorname{SSYT}_{d}(\lambda)$ is the disjoint union, over all contents $\mu$ of sum $d$, of the sets $\operatorname{SSYT}_{\mu}(\lambda)$.

We also recall from, e.g., Definition 1.2 of \cite{[Z1]} that reference, that $\operatorname{Dsi}(S)$ stand for the set of \emph{descending indices}, also known as \emph{descents}, of the element $S\in\operatorname{SYT}(\lambda)$, which are the indices $i$ for which $i+1$ shows up in a row below that of $i$. We will use the notation $v_{S}(i)$ from \cite{[CZ]} and \cite{[Z1]} for the box containing $i$ in $S$, with the row $R_{S}(i)$ and the column $C_{S}(i)$ (so that $i\in\operatorname{Dsi}(S)$ means $R_{S}(i+1)>R_{S}(i)$), and prove the following generalization of Lemma 1.8 of \cite{[Z1]}.
\begin{lem}
If $\mu$ is the content associated with the weak composition $\operatorname{comp}_{n}J$ for a multi-set $J$ as in Lemma \ref{mswcomp}, then the standardization is a bijection between the set $\operatorname{SSYT}_{\mu}(\lambda)$ from Definition \ref{defSSYT} and the set $\{S\in\operatorname{SYT}(\lambda)\;|\;\operatorname{Dsi}(S) \subseteq J\}$, where the containment is in the sense of Definition \ref{multisets}, for every $\lambda \vdash n$. \label{ctJmulti}
\end{lem}

\begin{proof}
As in the proof of Lemma 1.8 of \cite{[Z1]}, if $S$ is the standardization of $M\in\operatorname{SSYT}_{\mu}(\lambda)$, then the boxes containing $h$ in $M$ are those containing the numbers $j_{h}<p \leq j_{h+1}$ in $S$, where this set can be empty in case $j_{h+1}=j_{h}$ and the $h$th entry of $\operatorname{comp}_{n}J\vDash_{w}n$ vanishes. The argument from there shows that the fact that $M$ is semi-standard is equivalent, by the way standardization works, to the assertion that wherever $i\in\operatorname{Dsi}(S)$, the entry of $M$ at the box $v_{S}(i+1)$ is strictly larger than the one inside $v_{S}(i)$. The inverse of the map is clear, by simply ordering $\mu$ in non-decreasing order, and replacing each entry of $S$ by the one from $\mu$ in the corresponding location (this clearly yields a tableau $M$ of content $\mu$, and as we saw, since $\mu$ is associated with $\operatorname{comp}_{n}J$ and $J$ contains $\operatorname{Dsi}(S)$, this tableau $M$ is semi-standard). This proves the lemma.
\end{proof}
As in Definition 1.9 of \cite{[Z1]}, we may denote the destandardization inverting the map from Lemma \ref{ctJmulti} by $\operatorname{ct}_{J}$, and thus the standardization itself is $\operatorname{ct}_{J}^{-1}$.

\begin{ex}
Both multi-sets from Example \ref{wcompex} contain $\{1\}$, so that we take $\lambda=31\vdash4$ and a tableau $S$ with the latter set as $\operatorname{Dsi}(S)$, and get \[S=\begin{ytableau} 1 & 3 & 4 \\ 2 \end{ytableau},\mathrm{\ \ with\ \ }\operatorname{ct}_{J}(S)=\begin{ytableau} 0 & 2 & 2 \\ 2 \end{ytableau},\mathrm{\ \ and\ \ }\begin{ytableau} 2 & 4 & 5 \\ 4 \end{ytableau}\] is the semi-standard Young tableau that is associated with $S$ via Lemma \ref{ctJmulti} with the other multi-set. Note that the last 0 in $\operatorname{comp}_{n}J$, corresponding to the fact that $n=4 \in J$, has no effect on $\operatorname{ct}_{J}(S)$ (indeed, removing it would have produced the same element of $\operatorname{SSYT}_{6}(\lambda)$), but any other zero is visible in the fact that some entries do not appear in these tableaux, despite larger entries showing up. \label{ctJex}
\end{ex}
The existence of zeros, namely the skipping over some integers as discussed in Example \ref{ctJex}, shows how the Definition \ref{defSSYT} and Lemma \ref{ctJmulti} generalize Definitions 1.7 and 1.9 and Lemma 1.8 of \cite{[Z1]}.

\medskip

In addition to the set $\operatorname{Dsi}(S)$ of descents of a tableau $S\in\operatorname{SYT}(\lambda)$, and its set $\operatorname{Asi}(S)$ of \emph{ascending indices} or \emph{ascents} (which is the complement of $\operatorname{Dsi}(S)$ in $\mathbb{N}_{n-1}$ in case $\lambda=\operatorname{sh}(S) \vdash n$), we gather some definitions from \cite{[Z1]}, and extend them as follows.
\begin{defn}
Fix $\lambda \vdash n$, $S\in\operatorname{SYT}(\lambda)$, and $M\in\operatorname{SSYT}(\lambda)$.
\begin{enumerate}[$(i)$]
\item We set $\operatorname{Dsi}^{c}(S):=\{n-i\;|\;i\in\operatorname{Dsi}(S)\}$ and $\operatorname{Asi}^{c}(S):=\{n-i\;|\;i\in\operatorname{Asi}(S)\}$.
\item Consider $M$ as $\operatorname{ct}_{J}(T)$ for the unique multi-set $J$ and tableau $T\in\operatorname{SYT}(\lambda)$ via Lemma \ref{ctJmulti}, so that the content of $M$ is associated with $\operatorname{comp}_{n}J$ from Lemma \ref{mswcomp}. Then the multi-set $\{n-i\;|\;n>i \in J\}$ will be denoted by $\operatorname{Dsp}^{c}(M)$.
\item If $J$ from part $(ii)$ is the set $\operatorname{Dsi}(T)$, then $M$ is a \emph{cocharge tableau}. The set of the tableaux obtained in this way for our $\lambda$ is denoted by $\operatorname{CCT}(\lambda)$.
\item If $M=C\in\operatorname{CCT}(\lambda)$ is expressed as $\operatorname{ct}_{J}(T)$ with $J=\operatorname{Dsi}(T)$, then we write $C=\operatorname{ct}(T)$, and $T=\operatorname{ct}_{J}^{-1}(C)$ is written as $\operatorname{ct}^{-1}(C)$.
\item Let $I$ be a multi-set containing $\operatorname{Dsi}^{c}(S)$ from part $(i)$. Then we denote the complement $I\setminus\operatorname{Dsi}^{c}(S)$, interpreted as in Definition \ref{multisets}, by $\operatorname{Asi}^{c}_{I}(S)$. Similarly, if $M=C\in\operatorname{CCT}(\lambda)$ and $I$ contains $\operatorname{Dsp}^{c}(C)$ then the complement is denoted by $\operatorname{Asp}^{c}_{I}(C)$.
\end{enumerate} \label{sets}
\end{defn}
It is clear that the sets $\operatorname{Dsi}^{c}(S)$ and $\operatorname{Asi}^{c}(S)$ from Definition \ref{sets} are complements inside the set $\mathbb{N}_{n-1}$, and that if $M=\operatorname{ct}_{J}(T)$ then the multi-set $\operatorname{Dsp}^{c}(M)$ contains $\operatorname{Dsi}^{c}(T)$, with equality holding if and only if $M\in\operatorname{CCT}(\lambda)$. The fact that for $\operatorname{Dsp}^{c}(M)$ we removed the multiplicity of $n$ from $J$ (thus ignoring the possible multiplicity of 0 in the former set) is related to the fact that in Example \ref{ctJex}, the parameter $k$ associated with the multi-set $J$ was 3, but the maximal entry in $\operatorname{ct}_{J}(S)$ was strictly smaller (as $n=4 \in J$), as well as to Lemma \ref{Dspc} below. Lemma 1.15 of \cite{[Z1]} characterizes the elements of $\operatorname{CCT}(\lambda)$ among the general semi-standard Young tableaux as those tableaux $C$ satisfying the condition that if $h>0$ shows up in $C$ and $v$ is the leftmost box containing $h$ in $C$, then some box in $C$ contains $h-1$ and sits in a row above that of $v$. They are thus, up to a small normalization issue, the same as the quasi-Yamanouchi tableaux from \cite{[AS]}, \cite{[BCDS]}, and others.

Using Definition \ref{sets}, we generalize Lemma 1.17 of \cite{[Z1]}
\begin{lem}
Consider $\lambda \vdash n$ and $M\in\operatorname{SSYT}(\lambda)$.
\begin{enumerate}[$(i)$]
\item Write the multi-set $\operatorname{Dsp}^{c}(M)$ as an increasing sequence $\{i_{g}\}_{g=1}^{k-1}$ for the appropriate $k$, and add $i_{0}=0$ and $i_{k}=n$ as usual. Then for every index $0 \leq g \leq k$ there are $i_{g}$ entries in $M$ that equal at least $k-g$.
\item The sum $\sum_{i\in\operatorname{Dsp}^{c}(M)}i$ equals $\Sigma(M)$.
\item we can determine $M$ by the semi-standard filling of all its rows except the first one by positive integers, plus indicating the multiplicities of all the positive integers in the first row, plus the value of $n$.
\end{enumerate} \label{Dspc}
\end{lem}

\begin{proof}
Write $M$ as $\operatorname{ct}_{J}(S)$ for $S\in\operatorname{SYT}(\lambda)$ and a multi-set $J$ and via Lemma \ref{ctJmulti}, so that $J=\{j_{h}\}_{h=1}^{k-1}$ (plus perhaps some entries that equal $n$), with $j_{h}=n-i_{k-h}$ (and $j_{0}=0$ and $j_{k}=n$ as well). Then we saw in the proof there (or of Lemma 1.8 of \cite{[Z1]}) that every entry $h$ shows up $j_{h+1}-j_{h}$ times, and thus the number of entries that are equal to or larger than $k-g$ is the sum of these differences over $k-g \leq h<k$. But this telescopic sum reduces to $j_{k}-j_{k-g}=n-j_{k-g}=i_{g}$, part $(i)$ follows.

For part $(ii)$ we argue as in the proof of Corollary 1.10 of \cite{[Z1]}. Recalling that $M=\operatorname{ct}_{J}(S)$, the proof of Lemma \ref{ctJmulti} shows that for in $M$, the box $v_{S}(p)$ for any $p\in\mathbb{N}_{n}$ contains the number of elements of $J$ that are strictly smaller than $p$ (which is why the elements of $J$ that equal $n$ play no role, and we can ignore then and fix the index $k$ without them). We are interested in the sum $d=\Sigma(M)$ of these expressions over $p\in\mathbb{N}_{n}$.

But when we carry out the summation over $p$, we get the number of pairs $(p,i)\in\mathbb{N}_{n} \times J$ with $i>p$. Then the number of pairs containing every $i \in J$ are with $i<p \leq n$, namely $n-i$ such values, which for $i=j_{h}$ was seen to be $i_{k-h}$. As the sum of these elements is thus $\sum_{i\in\operatorname{Dsp}^{c}(M)}i$, this establishes part $(ii)$.

Part $(iii)$ is a consequence of the fact that the first row is determined, for semi-standard tableaux, by $n$ and the number of appearances of all the non-zero entries in that row. This proves the lemma.
\end{proof}

\begin{ex}
For the tableau $\operatorname{ct}_{J}(S)$ from Example \ref{ctJex}, the $\operatorname{Dsp}^{c}$-multi-set is $\{3,3\}$, with $k=3$. We thus have $i_{1}=i_{2}=3$ and $k=3$, and indeed, in the content $02^{3}$ of this tableau, $i_{0}=0$ entries are at least $3-0=3$, there are 3 entries that are at least $3-1=2$ with the same 3 being at least $3-2=1$ (all as part $(i)$ in Lemma \ref{Dspc} predicts), and with $i_{3}=n=4$ there are indeed 4 entries that are at least 0. The sum of the elements of the multi-set is $3+3=6$, just like the entry sum $0+2+2+2=6$, as expected from part $(ii)$ in Lemma \ref{Dspc}. The $\operatorname{Dsp}^{c}$-multi-set of the second tableau there is $\{1,3,3,4,4\}$, now with $k=6$, and if we see how many entries in the content $24^{2}5$ equals at least $6-g$ for $0 \leq g\leq6$, we indeed get the numbers 0, 1, 3, 3, 4, 4, and 4 (since $i_{6}=n=4$ as well). The two sums here are $1+3+3+4+4=15$ and $2+4+4+5=15$, exemplifying again both parts of that lemma. \label{exDspc}
\end{ex}
The description from part $(i)$ in Lemma \ref{Dspc} is also valid for $g=k$, since indeed there are $i_{k}=n$ non-negative entries, compared to the bound $k-g=0$. Moreover, as each entry of $M$ counts elements of $J$ that are strictly smaller from some integer $p \leq n$, and there are $k-1$ elements of $J$ that are smaller that $n$ (with multiplicities), meaning that no entry of $M$ equals at least $k-0$, and the assertion holds for $g=0$ as well. These limit cases were also presented in Example \ref{exDspc}.

\medskip

For fixing some notation, we recall Definition 2.2 of \cite{[Z1]}.
\begin{defn}
Take two integers $n\geq1$ and $d\geq0$.
\begin{enumerate}[$(i)$]
\item The algebra $\mathbb{Q}[x_{1},\ldots,x_{n}]$ will be denoted by $\mathbb{Q}[\mathbf{x}_{n}]$, and similarly $\mathbb{Z}[\mathbf{x}_{n}] $ stands for $\mathbb{Z}[x_{1},\ldots,x_{n}]$.
\item We write $\mathbb{Q}[\mathbf{x}_{n}]_{d}$ and $\mathbb{Z}[\mathbf{x}_{n}]_{d}$ for the parts of $\mathbb{Q}[\mathbf{x}_{n}]$ and $\mathbb{Z}[\mathbf{x}_{n}]$ respectively that are homogeneous of degree $d$.
\item A monomial in $\mathbb{Q}[\mathbf{x}_{n}]$ is said to have \emph{content} $\mu$ if $\mu$ is the multi-set, or sequence, of non-negative integers showing up as the exponents of variables in it. Adding and removing zeros from contents do not affect the content, so in particular we can consider it as a partition of the degree of the monomial.
\item If $\mu \vdash d$ and $\ell(\mu) \leq n$ then $\mathbb{Q}[\mathbf{x}_{n}]_{\mu}$ stands for the space generated by the monomials of content $\mu$ in $\mathbb{Q}[\mathbf{x}_{n}]$, with in $\mathbb{Z}[\mathbf{x}_{n}]_{\mu}$ being the intersection with $\mathbb{Z}[\mathbf{x}_{n}]$.
\end{enumerate} \label{Qxnd}
\end{defn}

We recall that to every $\lambda \vdash n$ there is an associated irreducible representation $\mathcal{S}^{\lambda}$ of $S_{n}$, which is called the \emph{Specht module} corresponding to $\lambda$, and these are, up to isomorphism, all the irreducible representations of $S_{n}$. Since they are all defined over $\mathbb{Q}$, one may ask given any finite-dimensional representation of $S_{n}$ over a field of characteristic 0, what is the multiplicity of any component $\mathcal{S}^{\lambda}$ in that representation.

A natural representation to ask this question about is $\mathbb{Q}[\mathbf{x}_{n}]_{d}$ for some $d\geq0$. There are two expressions yielding that answer, both of which are obtained using the tools from \cite{[St2]} as follows. We begin with the simpler one.
\begin{prop}
The multiplicity of $\mathcal{S}^{\lambda}$ in $\mathbb{Q}[\mathbf{x}_{n}]_{d}$ is equal to $|\operatorname{SSYT}_{d}(\lambda)|$. \label{Kostka}
\end{prop}

\begin{proof}
The space $\mathbb{Q}[\mathbf{x}_{n}]_{d}$ is the direct sum of the subspaces $\mathbb{Q}[\mathbf{x}_{n}]_{\mu}$ from Definition \ref{Qxnd}, where $\mu$ runs over the partitions of $d$ having length at most $n$. For each such $\mu$, which we write as $\{\mu_{i}\}_{i=1}^{n}$ after extending $\mu$ by $n-\ell(\mu)$ vanishing entries, we set $p_{\mu}:=\prod_{i=1}^{n}x_{i}^{\mu_{i}}$, and then the monomial basis for $\mathbb{Q}[\mathbf{x}_{n}]_{\mu}$ consists of the distinct images of $p_{\mu}$ under the action of $S_{n}$.

It follows that $\mathbb{Q}[\mathbf{x}_{n}]_{\mu}$ is the representation associated with the action of $S_{n}$ on the orbit of $p_{\mu}$, which is thus the one induced from the trivial representation on the stabilizer of that monomial. Write $\alpha\vDash_{w}n$ for the weak composition such that $\mu$ is the content represented by $\alpha$ (namely $\alpha_{h}$ for $0 \leq h \leq d$ is the multiplicity of $h$ in $\mu$). Thus we can decompose $\mathbb{N}_{n}$ into $d+1$ sets $\{A_{h}\}_{h=0}^{d}$, where $i \in A_{h}$ if and only if $x_{i}$ shows up with the exponent $h$ in $p_{\mu}$, and we write $p_{\mu}=\prod_{h=0}^{d}\prod_{i \in A_{h}}x_{i}^{h}$.

But now it is clear that an element of $S_{n}$ fixes $p_{\mu}$ if and only if it takes any element of $A_{h}$ to the same set. Since $|A_{h}|=\alpha_{h}$ by definition, the stabilizer in question is therefore isomorphic to the product $\prod_{h=0}^{d}S_{\alpha_{h}}$, the representation $\mathbb{Q}[\mathbf{x}_{n}]_{\mu}$ is isomorphic to the one denoted by $M^{\alpha}$ in \cite{[Sa]} or to the one whose character is denoted by $\eta^{\alpha}$ in \cite{[St2]} (we may omit the vanishing entries from $\alpha$ and get the same representation as associated with a real decomposition $\tilde{\alpha} \vDash n$ if we like). Hence, by Young's rule (see Proposition 7.18.7 of the latter reference, or Theorem 2.11.1 of the former one), it contains each $\mathcal{S}^{\lambda}$ with a multiplicity that equals the \emph{Kostka number} $K_{\lambda,\alpha}$.

But this number is $|\operatorname{SSYT}_{\mu}(\lambda)|$ by definition. Since $\operatorname{SSYT}_{d}(\lambda)$ is the union of the sets $\operatorname{SSYT}_{\mu}(\lambda)$ over the contents $\mu$ of sum $d$, and $\mathbb{Q}[\mathbf{x}_{n}]_{d}$ is the direct sum of $\mathbb{Q}[\mathbf{x}_{n}]_{\mu}$ over the same set of contents, the desired result follows. This proves the proposition.
\end{proof}

The second formula uses coefficients in generating series.
\begin{prop}
The representation $\mathbb{Q}[\mathbf{x}_{n}]_{d}$ contains $\mathcal{S}^{\lambda}$ with a multiplicity that coincides with The number of pairs $(C,I)$ where $C\in\operatorname{CCT}(\lambda)$ and $I$ is a multi-set of positive integers that is bounded by $n$, contains $\operatorname{Dsp}^{c}(C)$ as in Definition \ref{multisets}, and satisfies $\sum_{i \in I}=d$. \label{Adlambda}
\end{prop}

\begin{proof}
The character, as a symmetric function, that is associated with $\mathcal{S}^{\lambda}$, is the Schur function $s_{\lambda}$. The representation $\mathbb{Q}[\mathbf{x}_{n}]_{d}$ is the one denoted by $\psi^{d}$ in Exercise 7.73 of \cite{[St2]} (the reference there uses $k$ for the degree). If $\sigma^{d}$ is the corresponding symmetric function (namely $\sigma^{d}=\operatorname{ch}(\psi^{d})$ in the notation there), then the exercise in question compares, for a new variable $q$, the generating series $\sum_{d=0}^{\infty}\sigma^{d}q^{d}$ with the expression $\sum_{\lambda \vdash n}s_{\lambda}(1,q,q^{2},\ldots)s_{\lambda}$, where the second multiplier is our $s_{\lambda}$ and the first one is the series obtained by substituting the powers of $q$ in it.

It follows that to get $\sigma^{d}$, one takes for every $\lambda \vdash d$ the coefficient of $q^{d}$ in the series $s_{\lambda}(1,q,q^{2},\ldots)$, and sums over $\lambda$. Thus, as representations, the multiplicity in question is coefficient in front of $q^{d}$ in the latter series.

We will evaluate this series by invoking Proposition 7.19.11 of \cite{[St2]} for a partition or diagram $\lambda$ (the result there is more general, and applies for skew-diagrams as well). The numerator there runs over $\operatorname{SYT}(\lambda)$, and any $T$ in that set contributes to that numerator $q$ raised to the power $\sum_{i\in\operatorname{Dsi}(T)}i$. We can replace $T$ by $S:=\operatorname{ev}T$, where $\operatorname{ev}$ is the evacuation process of Sch\"{u}tzenberger described, among others, in Section 3.9 of \cite{[Sa]}. Thus $\operatorname{Dsi}(T)$ becomes $\operatorname{Dsi}^{c}(S)$, and if we replace $S$ by $C:=\operatorname{ct}(S)$ then the latter set is the same as $\operatorname{Dsp}^{c}(C)$, and its sum is $\Sigma(C)$ by part $(ii)$ of Lemma \ref{Dspc} (or already Lemma 1.17 of \cite{[Z1]}).

We thus transformed the expression for the series from that proposition into $\sum_{C\in\operatorname{CCT}(\lambda)}q^{\Sigma(C)}\big/\prod_{i=1}^{n}(1-q^{i})$, and recall that we require the coefficient of $q^{d}$ there. But for each $1 \leq i \leq n$ we have $1/(1-q^{i})=\sum_{\ell_{i}=0}^{\infty}q^{i\ell_{i}}$, so that we count the collection of $C\in\operatorname{CCT}(\lambda)$ and $\{\ell_{i}\}_{i=1}^{n}$ with $\Sigma(C)+\sum_{i=1}^{n}i\ell_{i}=d$.

But note that a multi-set $I$ of positive integers that is bounded by $n$ is determined by the multiplicities with which it contains every $i\in\mathbb{N}_{n}$, and if it contains $\operatorname{Dsp}^{c}(C)$ for some fixed $C$ then this multiplicity is positive for every $i$ in that set. Writing this multiplicity as $\ell_{i}+1$ in case $i\in\operatorname{Dsp}^{c}(C)$ and just $\ell_{i}$ otherwise, we get a bijection between such multi-sets and sequences $\{\ell_{i}\}_{i=1}^{n}$, and as it is clear that the $\sum_{i \in I}i$ equals $\Sigma(C)+\sum_{i=1}^{n}i\ell_{i}$, the contributions to $q^{d}$ are based precisely on those multi-sets whose element sum equals $d$, as desired. This completes the proof of the proposition.
\end{proof}
As consequences of the constructions of in paper, we will obtain two decompositions of $\mathbb{Q}[\mathbf{x}_{n}]_{d}$ into direct sum of irreducible representations, one corresponding to the formula from Proposition \ref{Kostka}, and another to that appearing in Proposition \ref{Adlambda}.

\medskip

In \cite{[Z1]} we dealt with decomposing the quotient $R_{n,k}$ from \cite{[HRS]}, where $1 \leq k \leq n$. We recall that the latter paper considers, for $s \leq k \leq n$, the finer and more general quotient $R_{n,k,s}$ of $\mathbb{Q}[\mathbf{x}_{n}]$, obtained by dividing by the ideal generated by $x_{i}^{k}$, $1 \leq i \leq n$ and by the symmetric functions $e_{r}$, $n-s+1 \leq r \leq n$ and shows that many properties of $R_{n,k}$ extend to these finer quotients. In this section we show how to generalize the results from \cite{[Z1]} to these quotients, which for the case $s=0$ will also apply for the homogeneous representations $R_{n,I}^{\mathrm{hom}}$ considered there (and generalized in Theorem \ref{Rnksdecom} below).
\begin{rmk}
In \cite{[HRS]}, the only place where the inequality $k \leq n$ is used is in the statement that we can take $s$ all the way up to $k$ and get the quotients $R_{n,k}$ as $R_{n,k,k}$. Every result that considers a specific quotient $R_{n,k,s}$ only uses the inequalities $s \leq k$ and $s \leq n$ (or $s\leq\min\{n,k\}$ together), without any order relation between $n$ and $k$. As in Remark \ref{comptow}, this is also the setting with which we will work in this paper. In particular, for $s=0$ (the case that will be natural for the homogeneous representations later) we get the quotient in which we only divide $\mathbb{Q}[\mathbf{x}_{n}]$ by $x_{i}^{k}$, $1 \leq i \leq n$, and no symmetric functions at all, which is defined for every $n$ and $k$. \label{kngen}
\end{rmk}

The paper \cite{[HRS]} shows that $R_{n,k,s}$ is isomorphic, as an ungraded representation, to the one arising from another action of $S_{n}$ on a set. Let $\widetilde{\mathcal{OP}}_{n,k,s}$ be the set of ordered partitions of $\mathbb{N}_{n+k-s}$ into $k$ sets, such that the last $k-s$ entries lie in the last $k-s$ sets, in increasing order, on which $S_{n}$ acts as usual on the subset $\mathbb{N}_{n}$ of $\mathbb{N}_{n+k-s}$ (and indeed, the other entries do not change their locations). Then it is shown in \cite{[HRS]} that if $1 \leq s \leq k \leq n$ then $R_{n,k,s}$ is isomorphic to $\mathbb{Q}[\widetilde{\mathcal{OP}}_{n,k,s}]$, but the proof applies also when $s=0$ (this is mentioned in \cite{[GR]} as well), and without the inequality $k \leq n$ (see Remark \ref{kngen}).

We will work with an isomorph $\mathcal{OP}_{n,k,s}$ of $\widetilde{\mathcal{OP}}_{n,k,s}$, which we now construct through its decomposition, generalizing those from \cite{[Z1]}.
\begin{defn}
Consider the following objects:
\begin{enumerate}[$(i)$]
\item Let $I$ be a multi-set of integers between 0 and $n$, of size $k-1$, and set $\vec{m}=\{m_{h}\}_{h=0}^{k-1}$ to be the weak composition $\operatorname{comp}_{n}I\vDash_{w}n$ from Lemma \ref{mswcomp}. We define $\mathcal{OP}_{n,I}$ to be the set of (generalized) ordered partitions of $\mathbb{N}_{n}$ into sets of respective sizes given by $\{m_{h}\}_{h=0}^{k-1}$ (the sets with $m_{h}=0$ will be empty, but we remember them in the ordering), and let $W_{\vec{m}}=W_{m_{0},\ldots,m_{k-1}}$ be the associated representation $\mathbb{Q}[\mathcal{OP}_{n,I}]$ of $S_{n}$.
\item Recalling that a such multi-set $I$ can be written as a non-decreasing sequence, and for $\operatorname{comp}_{n}I$ we take the differences in $I\cup\{0,n\}=\{i_{h}\}_{h=0}^{k}$ as a non-decreasing sequence, we $\mathcal{OP}_{n,k,s}$, for any $0 \leq s \leq k$, to be the union of the sets $\mathcal{OP}_{n,I}$ over those $I$ for which the initial sequence $\{i_{h}\}_{h=0}^{s}$ is strictly increasing.
\end{enumerate} \label{OPnksI}
\end{defn}
It is clear that when $I$ is a subset of $\mathbb{N}_{n-1}$, the set $\mathcal{OP}_{n,I}$ from Definition \ref{OPnksI} becomes the set with the same notation from Definition 3.6 of \cite{[Z1]}. As for the latter sets, each set $\mathcal{OP}_{n,I}$ from Definition \ref{OPnksI} admits a transitive action of $S_{n}$, hence $S_{n}$ acts on the union $\mathcal{OP}_{n,k,s}$, with the orbits being the $\mathcal{OP}_{n,I}$'s. It follows that $\mathcal{OP}_{n,k,s}$ is the set of (generalized) ordered partitions of $\mathbb{N}_{n}$ into $k$ sets, in which some of the last $k-s$ sets may be empty, with the action of $S_{n}$.

The set from Definition \ref{OPnksI} have the following properties.
\begin{lem}
Take any $n$, $k$, and $0 \leq s \leq k$.
\begin{enumerate}[$(i)$]
\item The set $\mathcal{OP}_{n,k,s}$ is empty if $s>n$.
\item We have the inclusion $\mathcal{OP}_{n,k,s+1}\subseteq\mathcal{OP}_{n,k,s}$.
\item When $k \leq n$ and $s=k$, the set $\mathcal{OP}_{n,k,k}$ is the original set $\mathcal{OP}_{n,k}$.
\item Given a multi-set $I$, the number $\hat{k}$ from Definition \ref{multisets} is the number of non-negative entries in $\vec{m}=\{m_{h}\}_{h=0}^{k-1}$ from Definition \ref{OPnksI}, and we have $\hat{k}=k$ if and only if $\mathcal{OP}_{n,I}\subseteq\mathcal{OP}_{n,k,k}=\mathcal{OP}_{n,k}$.
\item If $\mathcal{OP}_{n,I}\subseteq\mathcal{OP}_{n,k,s}$ and $\hat{k}$ is as above then $s\leq\hat{k} \leq n$.
\item We have $\mathcal{OP}_{n,k,s}\cong\widetilde{\mathcal{OP}}_{n,k,s}$ as sets with an action of $S_{n}$.
\end{enumerate} \label{propOPnks}
\end{lem}

\begin{proof}
If $\mathcal{OP}_{n,k,s}$ is not empty then there is an increasing sequence $\{i_{h}\}_{h=0}^{s}$ of integers between 0 and $n$, which clearly implies $s \leq n$ and part $(i)$ follows. The fact that the condition on $I$ for $\mathcal{OP}_{n,I}$ to participate in $\mathcal{OP}_{n,k,s}$ becomes more restrictive as $s$ increases yields part $(ii)$ (including the situation where $\mathcal{OP}_{n,k,s+1}$ is empty, with $\mathcal{OP}_{n,k,s}$ possibly not). For $s=k$, we know that $\{i_{h}\}_{h=0}^{k}$ is strictly increasing, so that the elements of $I$ are all distinct and we have $i_{1}>0$ and $i_{k-1}<n$. This means that $I$ is a subset of $\mathbb{N}_{n-1}$, so that $\mathcal{OP}_{n,I}$ is as defined in \cite{[Z1]}, and the union is the asserted one from part $(iii)$ as is showed in that reference.

Next, since $\hat{I}\subseteq\mathbb{N}_{n-1}$, the proof of part $(iii)$ yields $\mathcal{OP}_{n,\hat{I}}\subseteq\mathcal{OP}_{n,\hat{k},\hat{k}}$. This yields the second assertion in part $(iv)$, and when we compare $I$ with $\hat{I}$, we only add instances of 0, $n$, or elements that already appear in $\hat{I}$. As the effect of such an addition amounts to adding vanishing entries to $\operatorname{comp}_{n}\hat{I}$ in order to get $\operatorname{comp}_{n}I$, the first assertion is established in part $(iv)$ as well. Part $(v)$ is an immediate consequence of part $(iv)$ and the fact that if $\mathcal{OP}_{n,I}\subseteq\mathcal{OP}_{n,k,s}$ then the first $s$ entries of $\operatorname{comp}_{n}I$ are positive.

Now, when we take an ordered partition that lies in $\widetilde{\mathcal{OP}}_{n,k,s}$ and omit the entries from $\mathbb{N}_{n+k-s}\setminus\mathbb{N}_{n}$, we obtain a (generalized) ordered partition of $\mathbb{N}_{n}$ into $k$ sets, but where some of last $k-s$ sets in the partition may possibly be empty. But this is thus an element of $\widetilde{\mathcal{OP}}_{n,k,s}$.

Conversely, take an $\mathcal{OP}_{n,k,s}$ and add to the set of index $k-t$ for $1 \leq t \leq k-s$ the element $n-t\in\mathbb{N}_{n+k-s}$. Since only the last $k-s$ sets in the original element were possibly empty, and in its image these sets contain one of the added elements, this produces an element of $\widetilde{\mathcal{OP}}_{n,k,s}$ (since the last $k-s$ elements of $\mathbb{N}_{n+k-s}$ lie in the prescribed sets of the ordered partition by construction).

These two constructions are clearly inverses, and they respect the action of $S_{n}$ since they only affect elements from $\mathbb{N}_{n+k-s}\setminus\mathbb{N}_{n}$. Hence they form both directions of the isomorphism from part $(vi)$. This proves the lemma.
\end{proof}
It follows from parts $(i)$ and $(vi)$ of Lemma \ref{propOPnks} that $\widetilde{\mathcal{OP}}_{n,k,s}$ must also be empty in case $s>n$, which is indeed the case since the set $\mathbb{N}_{n+k-s}$ then has less than $k$ elements, and thus cannot have ordered partitions into non-empty $k$ sets. For the effect of removing the restrictions $s \leq k$ and $s \leq n$ on the quotient $R_{n,k,s}$, see Remark \ref{bigs} below. The constructions from Definition \ref{OPnksI} and Lemma \ref{propOPnks} are related to the sets $\mathcal{W}_{n,k,s}$ from Section 8 of \cite{[PR]}, and the equivalence of the representations $\mathbb{Q}[\mathcal{OP}_{n,k,s}]$ and $\mathbb{Q}[\widetilde{\mathcal{OP}}_{n,k,s}]$, readily following from part $(vi)$ of Lemma \ref{propOPnks}, may be related to the isomorphism between the cohomology rings of the varieties $X_{n,k,s}$ and $Y_{n,k,s}$ from that reference (see also Remark \ref{cohomology} below).

\begin{ex}
For $n=4$ and $J$ as in Examples \ref{wcompex} and \ref{ctJex}, the corresponding multi-set $I$ is $\{0,3,3\}$, with $k=4$ and $\operatorname{comp}_{4}I=0301\vDash_{w}4$. One element of $\mathcal{OP}_{4,I}$ is $\big(\{\},\{1,3,4\},\{\},\{2\}\big)$, with empty first and third sets according to the prescribed sizes. We have $\hat\{I\}=\{3\}$, of size $\hat{k}-1$ for $\hat{k}=2$, and indeed our weak composition has two non-zero entries. Since the first entry vanishes (or equivalently, $0 \in I$), the only value of $s$ for which $\mathcal{OP}_{4,I}\subseteq\mathcal{OP}_{4,4,s}$ is $s=0$, and the associated element of $\widetilde{\mathcal{OP}}_{4,4,0}$ is $\big(\{5\},\{1,3,4,6\},\{7\},\{2,8\}\big)$. The multi-set associated with the other multi-set there is $\{1,3,3,4,4\}$, with $k=6$, the associated weak partition $120100\vDash_{w}4$, subset $\{1,3\}$ of $\mathbb{N}_{3}$, with $\hat{k}=3$ and three non-vanishing entries. An associated (generalized) ordered partition is $\big(\{3\},\{1,4\},\{\},\{2\},\{\},\{\}\big)$, the maximal value of $s$ is 2, and when considered as an element of $\mathcal{OP}_{4,6,1}$, the corresponding ordered partition from $\widetilde{\mathcal{OP}}_{4,6,1}$ is $\big(\{3\},\{1,4,5\},\{6\},\{2,7\},\{8\},\{9\}\big)$. \label{ordpartex}
\end{ex}

\begin{rmk}
By adding empty sets, elements of $\mathcal{OP}_{n,\ell}$ for $s\leq\ell\leq\min\{n,k\}$ (with $\ell\geq1$ in case $s=0$, since $\mathcal{OP}_{n,0}$ does not exist as we assume that $n\geq1$ throughout) can be seen as naturally embedded inside $\mathcal{OP}_{n,k,s}$. In particular, part $(ii)$ of Lemma \ref{propOPnks} shows that $\mathcal{OP}_{n,k,k-1}$ contains $\mathcal{OP}_{n,k,k}=\mathcal{OP}_{n,k}$, and the complement is obtained by adding a single empty set to all the elements of $\mathcal{OP}_{n,k-1,k-1}=\mathcal{OP}_{n,k-1}$. However, the representations attached to the parts of $\mathcal{OP}_{n,\ell}$ in Theorem 3.10 of \cite{[Z1]} will be different from those which we relate to the images of these parts in $\mathcal{OP}_{n,k,s}$ in Theorem \ref{Rnksdecom} below (except for $\mathcal{OP}_{n,k,k}=\mathcal{OP}_{n,k}$, of course), so we view these sets as distinct. In any case, if $s \leq k-2$ then $\mathcal{OP}_{n,k,s}$ contains elements that are not obtained as such images. \label{difreps}
\end{rmk}

\medskip

Lemma \ref{propOPnks} transforms the result from \cite{[HRS]} (extended via Remark \ref{kngen}) into an isomorphism between $R_{n,k,s}$ and $\mathbb{Q}[\mathcal{OP}_{n,k,s}]$, for any $0 \leq s\leq\min\{n,k\}$. In fact, the proof of this result, as well as of the Specht basis property from Theorem 4.11 of \cite{[GR]}, considers first the case $s=0$, and establishes the general case by induction using Lemma 6.9 of \cite{[HRS]}, which we now state.
\begin{lem}
For any $0 \leq s<k$ then we have the short exact sequence \[0 \to R_{n,k-1,s}\stackrel{e_{n-s}}{\longrightarrow}R_{n,k,s} \longrightarrow R_{n,k,s+1}\longrightarrow0.\] When these rings are identified, as representations of $S_{n}$, with lifts into $\mathbb{Q}[\mathbf{x}_{n}]$, this sequences splits into the equality $R_{n,k,s}=e_{n-s}R_{n,k-1,s} \oplus R_{n,k,s+1}$. \label{splexseq}
\end{lem}
Note that in the extreme case $k=1$, Lemma \ref{splexseq} identifies $R_{n,1,s}=\mathbb{Q}$ with $R_{n,1,s+1}=\mathbb{Q}$ (with the kernel being the image of $R_{n,0,s}$, which we consider to be 0), and if $s=n<k$ then it identifies $R_{n,k-1,n}=R_{n}$ with $R_{n,k,n}=R_{n}$, with the quotient $R_{n,k,n+1}$ also being 0 (and for $s>n$ all the summands vanish). Our goal is to decompose each representations $R_{n,k,s}\cong\mathbb{Q}[\mathcal{OP}_{n,k,s}]$ in a way that respects both the decomposition of each $\mathcal{OP}_{n,k,s}$ as the union of the orbits $\mathcal{OP}_{n,I}$ for the corresponding multi-sets $I$, and the short exact sequence (or direct sum expression) from Lemma \ref{splexseq}.

\section{Generalized Higher Specht Polynomials \label{GenReps}}

We will work in this paper with two types of generalizations of higher Specht polynomials. One would be extending Definitions 2.3, 3.10, and 3.23 of \cite{[Z1]}, which are based on cocharge tableaux and sets, to multi-sets, and the other one generalizes the construction from a cocharge tableau to one that is based on a general semi-standard Young tableau.

Recall from \cite{[Z1]} and others that for given a tableau $T$ whose content is $\mathbb{N}_{n}$, we have the (disjoint) subgroups $R(T)$ and $C(T)$ of $S_{n}$, with the former preserves the rows of $T$, and the latter acts only on the columns of $T$. The associated \emph{Young symmetrizer} is $\varepsilon_{T}:=\sum_{\sigma \in C(T)}\sum_{\tau \in R(T)}\operatorname{sgn}(\sigma)\sigma\tau\in\mathbb{Z}[S_{n}]$, which is a scalar multiple of an idempotent from $\mathbb{Q}[S_{n}]$. Lemma 2.1 of that reference considers the extension of $C(T)$ to a semi-direct product $\tilde{C}(T)$ in which the subgroup $R(T)\cap\tilde{C}(T)$ operates on $C(T)$ (so that each subgroup $S_{a}^{t_{a}} \subseteq C(T)$ becomes the wreath product $S_{t_{a}} \rtimes S_{a}^{t_{a}}$ of $S_{t_{a}}$ and $S_{a}$ in $\tilde{C}(T)$), and extends the sign character on $C(T)$ to a character $\widetilde{\operatorname{sgn}}:\tilde{C}(T)\to\{\pm1\}$, which is trivial on the acting group $R(T)\cap\tilde{C}(T)$.

We will work with the following notions.
\begin{defn}
Take $\lambda \vdash n$, a tableau $T$ of shape $\lambda$ and content $\mathbb{N}_{n}$, and an element $M\in\operatorname{SSYT}(\lambda)$, with some content $\mu$, of sum $d$.
\begin{enumerate}[$(i)$]
\item The monomial $p_{M,T}$ is defined to be the product $\prod_{i=1}^{n}x_{i}^{h_{i}}$, where for $i\in\mathbb{N}_{n}$ the index $h_{i}$ is the entry showing up in the box $v_{T}(i)$ of $M$.
\item When we let $R(T)$ act on $M$, or equivalently on the monomial $p_{M,T}$ inside $\mathbb{Q}[\mathbf{x}_{n}]$, we get a stabilizer, whose size we denote by $s_{M,T}$.
\item The \emph{generalized higher Specht polynomial} corresponding to $M$ and $T$ is defined to be $F_{M,T}:=\varepsilon_{T}p_{M,T}/s_{M,T}$.
\item Given any multi-set $I$, consider it as the multi-set union of the set $\hat{I}$ from Definition \ref{multisets} and its multi-set complement $I\setminus\hat{I}$. Then if $i\in\hat{I}$ then we define $r_{i}:=|\{j\in\mathbb{N}_{n-1}\setminus\hat{I}\;|\;j<i\}|$ as in Definition 3.10 of \cite{[Z1]}, and for any $i$ in the complement we set $r_{i}:=n-\hat{k}+|\{j\in\hat{I}\cup\{n\}\;|\;j>i\}|$.
\item Assume that $M=C\in\operatorname{CCT}(\lambda)$ is a cocharge tableau, and $I$ is a multi-set which contains the set $\operatorname{Dsp}^{c}(C)$ from Definition \ref{sets} in the sense of Definition \ref{multisets}, with the multi-set complement $\operatorname{Asp}^{c}_{I}(C)$. Writing that multi-set as the union of the set $\hat{I}\setminus\operatorname{Dsp}^{c}(C)$ and the multi-set $I\setminus\hat{I}$, we define $F_{C,T}^{I}:=F_{C,T}\cdot\prod_{i\in\operatorname{Asp}^{c}_{I}(C)}e_{r_{i}}$.
\item For such $M=C$ and $I$ we also define $F_{C,T}^{I,\mathrm{hom}}:=F_{C,T}\cdot\prod_{i\in\operatorname{Asp}^{c}_{I}(C)}e_{i}$.
\item We set $\vec{h}_{C}^{I}$ to be the vector $\{h_{r}\}_{r=1}^{n}$ in which $h_{r}$ with $1 \leq r \leq n$ is $\big|\{i\in\operatorname{Asp}^{c}_{I}(C)\;|\;r_{i}=r\}\big|$, and denote by $\vec{h}(C,I)$ the characteristic vector of the multi-set $\operatorname{Asp}^{c}_{I}(C)$, with the multiplicity of 0 omitted.
\end{enumerate} \label{Spechtdef}
\end{defn}
It is clear that when $I$ is a subset of $\mathbb{N}_{n-1}$, the later parts of Definition \ref{Spechtdef} become Definitions 2.3, 3.10, and 3.23 of \cite{[Z1]}, using the notation with cocharge tableaux. These parts are therefore extensions of these definitions to more general multi-sets.

\begin{rmk}
The prequel \cite{[Z1]} used two other notations for the higher Specht polynomials and their multiples by symmetric functions. First, using the bijection $\operatorname{ct}:\operatorname{SYT}(\lambda)\to\operatorname{CCT}(\lambda)$ from Definition \ref{sets}, the higher Specht polynomials $F_{C,T}^{I}$ and $F_{C,T}^{I,\mathrm{hom}}$, where $C=\operatorname{ct}(S)$ for $S\in\operatorname{SYT}(\lambda)$, are also denoted by $F_{T,I}^{S}$ and $F_{T,I}^{S,\mathrm{hom}}$ respectively, also for multi-sets, and can be defined directly in terms of the multi-set complement $\operatorname{Asi}^{c}_{I}(S)=I\setminus\operatorname{Dsi}^{c}(S)$. If $T$ is also standard, then there is a unique element $w \in S_{n}$ such that $T=P(w)$ from the RSK algorithm, and $S$ is the image $\tilde{Q}(w)$ of $Q(w)$ from that algorithm under the Sch\"{u}tzenberger evacuation process (see, e.g., Section 3.9 of \cite{[Sa]} for its definition). Then these polynomials carry the respective additional notations $F_{w,I}$ and $F_{w,I}^{\mathrm{hom}}$, and the direct definition involves the multi-set complement $\operatorname{Asl}_{I}(w)$ inside $I$ of the set $\operatorname{Dsl}(w)$ of descending locations, or descents, of $w$ (the latter is denoted by $\operatorname{Des}(w)$ in \cite{[RW]} and others). However, we will stick with the notation from Definition \ref{Spechtdef} itself in this paper. \label{notation}
\end{rmk}

\medskip

We will now present the properties of the generalized higher Specht polynomials from Definition \ref{Spechtdef}, extending Lemma 2.5 and Proposition 2.6 of \cite{[Z1]}. The proofs applied a certain ordering on monomials, which we now define as we shall use it here in several places below.
\begin{defn}
Let $T$ be a fixed tableau of shape $\lambda \vdash n$ and content $\mathbb{N}_{n}$. Then, for every monomial $y\in\mathbb{Z}[\mathbf{x}_{n}]$ we set $\deg_{T}y$ to be the vector of non-negative integers, of length $\lambda_{1}$ (the first, maximal entry of $\lambda$), such that the $j$th entry $\deg_{T}^{j}y$ of that vector is the sum of the exponents to which the variables $x_{i}$ with $C_{T}(i)=j$ show up in $y$. For another monomial $z$, we write $y>_{T}z$ to say that $\deg_{T}y$ is (strictly) larger than $\deg_{T}z$ in the reverse lexicographic order. \label{orderT}
\end{defn}
It is clear that if $\sigma \in C(T)$ then the vector from Definition \ref{orderT} satisfies $\deg_{T}\sigma y=\deg_{T}y$ for every monomial $y$.

We now describe the properties of the generalized higher Specht polynomials, as was carried out in \cite{[Z1]} and elsewhere.
\begin{lem}
If $T$ and $M$ are as in Definition \ref{Spechtdef}, then $\sum_{\tau \in R(T)}\tau p_{M,T}/s_{M,T}$ is a sum of distinct, monic monomials, all of which have the same content as $M$ and thus degree $\Sigma(M)$, with the original monomial $p_{M,T}$ being the largest in the order from Definition \ref{orderT}. \label{operRT}
\end{lem}
Lemma \ref{operRT} is proved just like Lemma 2.5 of \cite{[Z1]}, since the only property of $C=\operatorname{ct}(S)$ that was used there was it being semi-standard. Based on that lemma, we follow the proof of Proposition 2.6 of that reference, and get the following description.
\begin{prop}
The generalized higher Specht polynomial $F_{M,T}$ from Definition \ref{Spechtdef} lies in $\mathbb{Z}[\mathbf{x}_{n}]_{d}$ for $d=\Sigma(M)$, and it equals $\sum_{\sigma \in C(T)}\operatorname{sgn}(\sigma)\sigma p_{M,T}$ plus a combination of monomials $y$ that satisfy $p_{M,T}>_{T}y$. The action of every element $\sigma\in\tilde{C}(T)$ multiplies it by the sign $\widetilde{\operatorname{sgn}}(\sigma)$. \label{Spechtpols}
\end{prop}
In fact, we have $F_{M,T}\in\mathbb{Z}[\mathbf{x}_{n}]_{\mu}\subseteq\mathbb{Z}[\mathbf{x}_{n}]_{d}$ when $\mu$ is the content of $M$ in Proposition \ref{Spechtpols}.

Recalling from Definition 2.8 of \cite{[Z1]} that to the shape $\lambda \vdash n$ of $M$ and $T$ there is a minimal cocharge tableau $C^{0}\in\operatorname{CCT}(\lambda)$ (with the $i$th row filled with the number $i-1$ throughout), yielding the Specht polynomial $F_{C^{0},T}$, we get from Proposition \ref{Spechtpols} the following extension of Corollary 2.9 there.
\begin{cor}
Given $T$ and $M$ as above, the quotient $Q_{M,T}:=F_{M,T}/F_{C^{0},T}$ is a homogeneous polynomial in $\mathbb{Z}[\mathbf{x}_{n}]$, and the action of $\tilde{C}(T)$ preserves it. \label{quotSpecht}
\end{cor}
We might call the quotients from Corollary 2.9 of \cite{[Z1]} \emph{Specht quotients}, and the ones from Corollary \ref{quotSpecht} \emph{generalized Specht quotients}.

\begin{ex}
Take $T$ to be the tableau $S$ from Example \ref{ctJex}. The the polynomial from Lemma \ref{operRT} that is associated with $T$ and $\operatorname{ct}_{J}(S)$ from that example is $x_{2}^{2}(x_{1}^{2}x_{3}^{2}+x_{1}^{2}x_{4}^{2}+x_{3}^{2}x_{4}^{2})$, and the corresponding generalized higher Specht polynomial from Definition \ref{Spechtdef} and Proposition \ref{Spechtpols} is $(x_{2}^{2}-x_{1}^{2})x_{3}^{2}x_{4}^{2}$, with the quotient from Corollary \ref{quotSpecht} being $(x_{1}+x_{2})x_{3}^{2}x_{4}^{2}$. If $M$ is the second tableau showing up there, then the combination of monomials, the generalized higher Specht polynomial, and the quotient are \[x_{2}^{4}(x_{1}^{2}x_{3}^{4}x_{4}^{5}+x_{1}^{2}x_{3}^{5}x_{4}^{4}+x_{1}^{4}x_{3}^{2}x_{4}^{5}+x_{1}^{4}x_{3}^{5}x_{4}^{2}+x_{1}^{5}x_{3}^{2}x_{4}^{4}+x_{1}^{5}x_{3}^{4}x_{4}^{2}),\] \[(x_{2}^{4}x_{1}^{2}-x_{2}^{2}x_{1}^{4})(x_{3}^{4}x_{4}^{5}+x_{3}^{5}x_{4}^{4})-(x_{2}^{5}x_{1}^{4}-x_{2}^{4}x_{1}^{5})(x_{3}^{2}x_{4}^{4}+x_{3}^{4}x_{4}^{2}),\] and \[(x_{2}^{2}x_{1}^{3}+x_{2}^{3}x_{1}^{2})(x_{3}^{4}x_{4}^{5}+x_{3}^{5}x_{4}^{4})-x_{1}^{4}x_{2}^{4}(x_{3}^{2}x_{4}^{4}+x_{3}^{4}x_{4}^{2})\] respectively, all of which are divisible by $x_{1}^{2}x_{2}^{2}x_{3}^{2}x_{4}^{2}$ because all the entries of that tableau are at least 2. \label{exSpecht}
\end{ex}

\medskip

We recall the notions from Definition 3.1 of \cite{[Z1]} and generalize it as follows.
\begin{defn}
Given $M\in\operatorname{SSYT}(\lambda)$ for some $\lambda \vdash n$, we define $V_{M}$ to the subspace of $\mathbb{Q}[\mathbf{x}_{n}]$ that is generated by $F_{M,T}$, where $T$ runs over the tableaux of shape $\lambda$ and content $\mathbb{N}_{n}$. When $M=C\in\operatorname{CCT}(\lambda)$ and $\vec{h}:=\{h_{r}\}_{r=1}^{n}$ is a vector of non-negative integers, we write $V_{C}^{\vec{h}}$ for the space obtained by multiplying, inside $\mathbb{Q}[\mathbf{x}_{n}]$, all the elements of $V_{C}$ by the symmetric polynomial $\prod_{r=1}^{n}e_{r}^{h_{r}}$ inside $\mathbb{Q}[\mathbf{x}_{n}]$. \label{defSpecht}
\end{defn}

\begin{rmk}
We will, when convenient, identify the vector $\vec{h}$ from Definition \ref{defSpecht} with the vector obtained by adding vanishing entries at some or all of the indices $r>n$. As in Remark \ref{notation}, we recall that if $C\in\operatorname{CCT}(\lambda)$ then the representation $V_{C}^{\vec{h}}$ from that definition was also denoted by $V^{S}_{\vec{h}}$ in \cite{[Z1]}, where $S:=\operatorname{ct}^{-1}(C)\in\operatorname{SYT}(\lambda)$. \label{VSVC}
\end{rmk}

We now state a generalization of Theorem 3.2 of \cite{[Z1]} and similar results.
\begin{thm}
Given an element $M\in\operatorname{SSYT}(\lambda)$ for some $\lambda \vdash n$, the polynomials $\{F_{M,T}\;|\;T\in\operatorname{SYT}(\lambda)\}$ form a basis for $V_{M}$ Definition \ref{Spechtdef}, and the latter is an irreducible representation of $S_{n}$ on a subspace of $\mathbb{Q}[\mathbf{x}_{n}]$, which is homogeneous of degree $d:=\Sigma(M)$ and is isomorphic to $\mathcal{S}^{\lambda}$. \label{VMreps}
\end{thm}
The fact that $V_{C}^{\vec{h}}$ also has similar properties, with the homogeneity degree being $\Sigma(C)+\sum_{r=1}^{n}rh_{r}$, already shows up in Theorem 3.2 of \cite{[Z1]}. For the proof, we recall from Section 3 of \cite{[GR]} that the proof from \cite{[Pe]} only considered the operators from Definition \ref{Spechtdef} (see also \cite{[M]}), and not the monomial on which they were applied, and thus the argument proving Corollary 3.17 of that reference, which applies in our setting as well, yields the desired result.

Combining the notation from Definition \ref{Spechtdef} with Theorem \ref{VMreps}, we obtain the following extension of part $(iv)$ of Lemma 3.11 and of Lemma 3.24 of \cite{[Z1]}.
\begin{lem}
Given $\lambda \vdash n$, $C\in\operatorname{CCT}(\lambda)$, and a multi-set $I$ containing $\operatorname{Dsp}^{c}(C)$ in the usual sense, the polynomials $\big\{F_{C,T}^{I}\;|\;T\in\operatorname{SYT}(\lambda)\big\}$ form a basis for the representation $V_{C}^{\vec{h}_{C}^{I}}$, while $V_{C}^{\vec{h}(C,I)}$, with the two vectors $\vec{h}_{C}^{I}$ and $\vec{h}(C,I)$ from Definition \ref{Spechtdef}, admits $\big\{F_{C,T}^{I,\mathrm{hom}}\;|\;T\in\operatorname{SYT}(\lambda)\big\}$ as a basis. \label{spanreps}
\end{lem}
Lemma \ref{spanreps} follows directly from Theorem \ref{VMreps} (or more precisely Theorem 3.2 of \cite{[Z1]}), since the multipliers $\prod_{i\in\operatorname{Asp}^{c}_{I}(C)}e_{r_{i}}$ and $\prod_{i\in\operatorname{Asp}^{c}_{I}(C)}e_{i}$ from Definition \ref{Spechtdef} are associated with the vectors $\vec{h}_{C}^{I}$ and $\vec{h}(C,I)$ respectively.

\medskip

Theorem 1.7 of \cite{[GR]} was cited in our terminology as Theorem 3.4 of \cite{[Z1]}, and after we change the notation $V^{S}_{\vec{h}}$ from Remark \ref{VSVC} to $V_{C}^{\vec{h}}$ used in the current paper (with the equality $\operatorname{Dsp}^{c}(C)=\operatorname{Dsi}^{c}(S)$ when $S:=\operatorname{ct}^{-1}(C)$), it generalizes to the following result, showing up in the proof from \cite{[GR]}.
\begin{thm}
For $\lambda \vdash n$, integers $k$ and $s$ with $0 \leq s\leq\min\{n,k\}$, and $C\in\operatorname{CCT}(\lambda)$, we define $H_{C}^{k,s}$ to be the set of vectors $\vec{h}=\{h_{r}\}_{r=1}^{n}$ of non-negative integers, for which $h_{r}=0$ wherever $r>n-s$ and the inequality $\sum_{r=1}^{n-s}h_{r}<k-|\operatorname{Dsp}^{c}(C)|$ holds. Using this notation, we get a direct sum $\bigoplus_{\lambda \vdash n}\bigoplus_{C\in\operatorname{CCT}(\lambda)}\bigoplus_{\vec{h} \in H_{C}^{k,s}}V_{C}^{\vec{h}}$, which lifts the quotient $R_{n,k,s}$, as a representation of $S_{n}$, into $\mathbb{Q}[\mathbf{x}_{n}]$. \label{decomRnk}
\end{thm}
It is clear that if $k \leq n$ then the case $s=k$ in Theorem \ref{decomRnk} becomes Theorem 3.4 of \cite{[Z1]} (in the notation using cocharge tableaux), with the special case $s=k=n$ becoming the result of \cite{[ATY]}.

Definition \ref{defSpecht} and Theorem \ref{decomRnk} suggest that given $\lambda \vdash n$ and $C\in\operatorname{CCT}(\lambda)$ we should consider the map taking a multi-set $I$ for which $\operatorname{Dsp}^{c}(C) \subseteq I$ via Definition \ref{multisets} to the vector $\vec{h}_{C}^{I}$, that was constructed via the complement $\operatorname{Asp}^{c}_{I}(C)$ from Definition \ref{sets}. In the case considered in \cite{[Z1]}, where $I$ was a subset of $\mathbb{N}_{n-1}$, analyzing this map was based on the fact that if $i \mapsto r_{i}$ was non-decreasing on $I$, or on $\operatorname{Asp}^{c}_{I}(C)$. Here it is no longer the case.
\begin{lem}
Let $I$ be a multi-set of size $k-1$, let $\hat{I}$, $\hat{k}$, and $I\setminus\hat{I}$ be as in Definition \ref{multisets}, and assume that $D$ is a subset of $\hat{I}$, of size $d-1$.
\begin{enumerate}[$(i)$]
\item Order the set $\hat{I} \setminus D$ in an increasing order, and then continue this sequence by putting $I\setminus\hat{I}$ in a non-increasing order afterwards. Then the map taking $i \in I$ to $r_{i}$ from Definition \ref{defSpecht} sends the resulting sequence to a non-decreasing sequence of length $k-d$.
\item When $\hat{k}>d$, the element in the $(\hat{k}-d)$th location of the resulting sequence is at most $n-\hat{k}$. If $\hat{k}<k$, then the element at the location number $\hat{k}-d+1$ there equals at least $n-\hat{k}$.
\item Given any non-decreasing sequence of integers between 0 and $n$, whose length is $k-d$ and that is not constantly $n$, there exists a unique $d\leq\hat{k} \leq k$ for which the condition from part $(ii)$ is satisfied.
\item The non-increasing sequences of that sort are in one-to-one correspondence with vectors $\vec{h}=\{h_{r}\}_{r=1}^{n}$ with $\sum_{r=0}^{n}h_{r} \leq k-d$. The correspondence is based on $h_{r}$ counting the number of times $r$ shows up in the sequence, including the additional value of $h_{0}\geq0$ for which $\sum_{r=0}^{n}h_{r}=k-d$.
\item Given such a sequence, set $\hat{k}$ to be the parameter from part $(iii)$, and let $\vec{h}$ be the associated vector via part $(iv)$. Then the first $\hat{k}-d$ entries, which arise from $\hat{I}$ in case our sequence was obtained from $I$ via part $(i)$, consists of $h_{r}$ instances of $r$ for every $0 \leq r<n-\hat{k}$, and $\hat{k}-d-\sum_{r=0}^{n-\hat{k}-1}h_{r}$ instances of $n-\hat{k}$.
\end{enumerate} \label{incri}
\end{lem}

\begin{proof}
Note that $r_{i}$ from Definition \ref{Spechtdef} counts elements that are smaller than $i$ in case $i\in\hat{I}$, but counts (up to a constant) elements that are larger than $i$ when $i \in I\setminus\hat{I}$. Hence this index increases with $i$ on $\hat{I}$, but decreases on $I\setminus\hat{I}$.

Moreover, in the first case the set from which these elements are taken is $\mathbb{N}_{n-1}\setminus\hat{I}$, of size $n-\hat{k}$, hence $r_{i}$ is bounded by that number, while for $i \in I\setminus\hat{I}$ the value is at least this number by definition (when such elements exist). Hence the resulting sequence (whose length is clearly $n-d$). This proves both parts $(i)$ and $(ii)$.

We now take a non-decreasing sequence of integers between 0 and $n$, and compare it with the decreasing sequence having $n-d+1-j$ in the $j$th location. In case the first element of the sequence is already at least $n-d$, we obtain the condition for $\hat{k}=d$. If $k \leq n$ and all the sequence is bounded by $n-k$, then the condition is satisfied with $\hat{k}=k$.

In any other situation, we have a non-decreasing sequence and a strictly decreasing sequence, both of length $k-d$, that do not bound one another. There thus must an index $d<\hat{k}<k$ such that at the location $\hat{k}-d>0$ the non-decreasing sequence does not yet reach the value in the decreasing one, while at the next one $\hat{k}-d+1 \leq k-d$ it already does so. Hence the two inequalities from part $(ii)$ hold with this value of $\hat{k}$, and it is clear that in all cases, the value of $\hat{k}$ is unique. Part $(iii)$ is thus also established.

Next, since the order is non-decreasing, counting the number of instances of every $0 \leq r \leq n$ in a sequence yields the vector $\vec{h}=\{h_{r}\}_{r=1}^{n}$ with the parameter $h_{0}\geq0$ such that the total sum $\sum_{r=0}^{n}h_{r}$ is $k-d$. Conversely, given such a vector, with the value of $h_{0}$, we get the unique associated non-decreasing sequence by putting $h_{0}$ instances of 0, followed by $h_{1}$ instances of 1, up to $h_{n}$ instances of $n$, with the length being $\sum_{r=0}^{n}h_{r}=k-d$ as desired, and part $(iv)$ follows.

Finally, in the conditions from part $(v)$, it is clear from part $(iii)$ that the first $\hat{k}-d$ entries of the sequence are bounded from above by $n-\hat{k}$, while this number bounds the rest of the sequence from below. It thus contains all the instances of every $r<n-\hat{k}$ (and part $(iv)$ shows that there are $h_{r}$ of those), and the remaining entries, the number of which is thus $\hat{k}-d-\sum_{r=0}^{n-\hat{k}-1}h_{r}$ since the length of the part of the sequence under consideration is $\hat{k}-d$, all equal $n-\hat{k}$ as desired. This proves the lemma.
\end{proof}

\begin{ex}
Take $n=8$, $k=8$ as well, $d=4$, and the subset $D\subseteq\mathbb{N}_{7}$, of size $d-1=3$, to be $\{2,5,6\}$. We consider the multi-set $I:=\{2,3,5,5,6,7,7\}$, of size $k-1=7$. Then $\hat{I}=\{2,3,5,6,7\}$, of size $\hat{k}-1$ for $\hat{k}=6$, and we have $\hat{I} \setminus D=\{3,7\}$ and $I\setminus\hat{I}=\{5,7\}$. Thus the ordering of $I \setminus D$ via part $(i)$ of Lemma \ref{incri} is 3775, replacing each $i$ there by $r_{i}$ produces the sequence 1235 (indeed of length $n-d=4$, and the entry 2 at the location $\hat{k}-d=2$ indeed does not exceed $n-\hat{k}=2$, while that of the next location equals at least this value. when we compare the latter sequence with the decreasing sequence of $n-d+1-j$, namely 4321, we indeed see that the location where the orders between the sequences is between 2 and 3, which is indeed $\hat{k}-d$ and $\hat{k}-d+1$ for $d=4$ and $\hat{k}=6$, as part $(i)$ of that lemma predicts. The vector $\vec{h}$ associated with our sequence via part $(iv)$ of that lemma is with $h_{r}=0$ for $r\in\{1,2,3,5\}$ and $h_{r}=0$ otherwise, and the first $\hat{k}-d=2$ parts of the sequence are 12, with the instances of $h_{r}$ for $r=1<2=\hat{k}-d$ (and $h_{0}=0$), completed by $\hat{k}-d-\sum_{r=0}^{n-\hat{k}-1}h_{r}=2-0-1$ instances of $n-\hat{k}=2$, via part $(v)$ there. \label{hatkmid}
\end{ex}

\begin{ex}
With $n$, $k$, $d$, and $D$ as in Example \ref{hatkmid}, we take $I_{s}$ to be $\{1,2,3,4,5,6,7\}$. Then $\hat{I}_{s}=I_{s}$, with $\hat{k}_{s}=k=8$, the order on $I_{s} \setminus D$ is the increasing one 1347, all the $r_{i}$'s vanish, and we get the sequence 0000. Even the last, maximal entry does not reach the smallest one from the decreasing one from that example, which indeed represents the case $\hat{k}_{s}=k$ in Lemma \ref{incri}. Here we have $h_{0}=4$ and $h_{r}=0$ for $r>0$, the part of the sequence considered in part $(v)$ of that lemma is the entire sequence, which indeed includes all the instances of $h_{0}$ since $0<\hat{k}-d=8-4=4$. In the other extreme, set $I_{b}:=\{0,2,5,5,5,6,8\}$, where $\hat{I}_{b}=\{2,5,6\}=D$, $\hat{k}_{b}=d=4$, and $I_{b} \setminus D$ is ordered non-increasingly as 8550. When we evaluate the $r_{i}$'s, we get the sequence 4668, with already the first location reaches the value in the decreasing sequence, in correspondence with the value $\hat{k}_{b}=d$. The vector $\vec{h}$ from part $(iv)$ of Lemma \ref{incri} is with $h_{4}=h_{8}=1$ and $h_{6}=2$ (and the other entries vanishing), and part $(v)$ there considers the empty sequence, as indeed $h_{r}=0$ for all $r<n-\hat{k}=4$, and the length is $\hat{k}-d=0$. \label{hatkend}
\end{ex}

\medskip

We now extend Lemma 3.17 of \cite{[Z1]} and prove a homogeneous analogue.
\begin{lem}
Fix $n$, $k$, $\lambda \vdash n$, and $C\in\operatorname{CCT}(\lambda)$.
\begin{enumerate}[$(i)$]
\item The map sending a multi-set $I$ of size $k-1$ that contains $\operatorname{Dsp}^{c}(C)$ to $\vec{h}(C,I)$ is a bijection between the collection of these multi-sets and the set $H_{C}^{k,0}$ from Theorem \ref{decomRnk}.
\item The set $H_{C}^{k,0}$ contains $H_{C}^{k-1,0}$, and $\vec{h}(C,I)$ is in this subset for a multi-set $I$ if and only if $0 \in I$.
\item The map taking $I$ as in part $(i)$ to $\vec{h}_{C}^{I}$ is another bijection between these two sets.
\item Given $0 \leq s\leq\min\{n,k\}$ and a multi-set $I$ as in parts $(i)$ and $(iii)$, we have $\mathcal{OP}_{n,I}\subseteq\mathcal{OP}_{n,k,s}$ in Definition \ref{OPnksI} if and only if $\vec{h}_{C}^{I} \in H_{C}^{k,s}$.
\item The set $H_{C}^{k,s+1}$ is contained in $H_{C}^{k,s}$.
\item The map adding 1 to the entry $h_{n-s}$ yields an injection from $H_{C}^{k-1,s}$ into $H_{C}^{k,s}$, and its image there is the complement of the subset from part $(iv)$.
\end{enumerate} \label{bijvecs}
\end{lem}

\begin{proof}
We follow the proof of Lemma 3.17 of \cite{[Z1]}. The sums of the coordinates in both $\vec{h}_{C}^{I}$ and $\vec{h}(C,I)$ collect the number of elements of the multi-set $\operatorname{Asp}^{c}_{I}(C)$ from Definition \ref{sets}, with its multiplicities, which give non-zero contributions. They are thus bounded by the size of this multi-set, which is $k-1-|\operatorname{Dsp}^{c}(C)|$, and since we have $h_{r}=0$ for all $r>n$ in both of them, the maps in parts $(i)$ and $(iii)$ are indeed into $H_{C}^{k,0}$.

For the converses of both maps, take $\vec{h}=\{h_{r}\}_{r=1}^{n} \in H_{C}^{k,0}$, and as its coordinate sum is smaller than $k-|\operatorname{Dsp}^{c}(C)|$, we may set $h_{0}\geq0$ to be such that $\sum_{r=0}^{n}h_{r}=k-1-|\operatorname{Dsp}^{c}(C)|$ (we call this addition of $h_{0}$ to $\vec{h}$ the \emph{completion} of $\vec{h}$). For part $(i)$ we need to show that there exists a unique multi-set $I$ for which $\vec{h}(C,I)=\vec{h}$, while part $(iii)$ requires the existence of a unique multi-set $I$ with $\vec{h}_{C}^{I}=\vec{h}$.

Considering part $(i)$, it is clear from Definitions \ref{Spechtdef} and \ref{sets} that $\vec{h}(C,I)=\vec{h}$ if and only if the completion of $\vec{h}$ describes the multiplicities with which every element (including 0) shows up in $\operatorname{Asp}^{c}_{I}(C)$. But this determines the latter multi-set uniquely, hence yields a unique multi-set $I$ as the multi-set addition of $\operatorname{Dsp}^{c}(C)$ and $\operatorname{Asp}^{c}_{I}(C)$. Since $|\operatorname{Asp}^{c}_{I}(C)|=k-1-|\operatorname{Dsp}^{c}(C)|$, this multi-set $I$ is of size $k-1$ and clearly contains $\operatorname{Dsp}^{c}(C)$, and thus $\vec{h}(C,I)=\vec{h}$ for this unique $I$. This yields part $(i)$.

The fact that when $k$ is decreased to $k-1$ only strengthens the inequality in Theorem \ref{decomRnk} yields the inclusion in part $(ii)$ (and in fact gives the inclusion $H_{C}^{k-1,s} \subseteq H_{C}^{k,s}$ for every $s$, when the usual inequalities hold). As $\vec{h}(C,I)$ is, via Definition \ref{Spechtdef},  the characteristic function of $\operatorname{Asp}^{c}_{I}(C)$ with the appropriate multiplicity of 0 added, we deduce that if $\vec{h}(C,I)$ is in $H_{C}^{k-1,0}$ then $I$ is the union of 0 with the multi-set of size $k-2$ associated with it via part $(i)$ for $k-1$, and thus $0 \in I$. Conversely, from every multi-set with $0 \in I$ we can subtract that 0 and use the same ideas, and part $(ii)$ follows.

Next, set $D:=\operatorname{Dsp}^{c}(C)$ and $d:=|D|+1$, and then part $(iv)$ of Lemma \ref{incri} identifies the completion of our vector $\vec{h} \in H_{C}^{k,0}$ with a non-decreasing sequence of length $k-d$. We will use this sequence in order to determine the unique multi-set $I$ for which $\vec{h}_{C}^{I}=\vec{h}$.

That lemma shows that this sequence can be obtained from the $r_{i}$'s of a multi-set $I$ only when the corresponding parameter $\hat{k}$ is as determined by part $(iii)$ there. So we assume that $\hat{k}$ is given, and then the first $\hat{k}-d$ entries of the sequence, which are described in part $(v)$ there, would be obtained from $\hat{I} \setminus D$. In fact, this part of the sequence determines a unique subset $A$ of size $\hat{k}-d$ inside the set $\mathbb{N}_{n-1} \setminus D$, of size $n-d$, and we take $\hat{I}:=A \cup D$.

The proof of Lemma 3.17 of \cite{[Z1]} shows how to determine this set $A$. Indeed, Lemma \ref{incri} shows that this part contains $h_{0}$ instances of 0, followed by $h_{1}$ instances of 1, ending with $\hat{k}-d-\sum_{r=0}^{n-\hat{k}-1}h_{r}$ instance of $n-\hat{k}$. We thus order the elements of $\mathbb{N}_{n-1} \setminus D$ increasingly, take the first $h_{0}$ elements into $A$ (or none if $h_{0}=0$), then skip one, add the next $h_{1}$ elements (which is nothing in case $h_{1}=0$), and go on, until we complete $A$ with the maximal $\hat{k}-d-\sum_{r=0}^{n-\hat{k}-1}h_{r}$ elements there (which again may be zero). The set $\hat{I}:=A \cup D$ is then the unique set for which ordering $\hat{I} \setminus D=A$ and sending $i$ there to $r_{i}$ yields this sequence, of length $\hat{k}-d$.

Now that we have determined $\hat{I}$ (and $\hat{k}$), we know that $I\setminus\hat{I}$ consists of the elements of $\hat{I}\cup\{0,n\}$, and they give the entries of the remaining $k-\hat{k}$ entries of the sequence, all of which equal at least $n-\hat{k}$. Since the number of elements in $\hat{I}\cup\{n\}$ that are larger than an element of that set, as in the definition of $r_{i}$ for $i \in I\setminus\hat{I}$ via Definition \ref{Spechtdef}, determine it (in particular, any $i$ there with $r_{i}=n-\hat{k}$ equals $n$, while if $r_{i}=n$ then $i=0$), this constructs the remainder of $I$, as desired for the bijectivity in part $(iii)$.

Next, Definition \ref{OPnksI} implies that $\mathcal{OP}_{n,I}\subseteq\mathcal{OP}_{n,k,s}$ for some $s>0$ if and only if $0 \not\in I$ and the first $s-1$ elements of $I$ show up only once. This is equivalent to these elements of $\hat{I}$ and 0 not showing up in $I\setminus\hat{I}$, and as these are precisely the elements of this multi-set for which $r_{i}>n-s$, part $(iv)$ follows. Part $(v)$ is an immediate consequence of the fact that the condition for being in $H_{C}^{k,s+1}$ is stronger than that for being in $H_{C}^{k,s}$.

More precisely, $H_{C}^{k,s+1}$ is the set of elements of $H_{C}^{k,s}$ for which $h_{n-s}=0$. The complement consists of those vectors $\vec{h} \in H_{C}^{k,s}$ that satisfy $h_{n-s}>0$. As subtracting 1 from $h_{n-s}$ yields an arbitrary element of $H_{C}^{k-1,s}$ by definition, and adding 1 to that entry of any vector in that set produces an element of $H_{C}^{k,s} \setminus H_{C}^{k,s+1}$ in a bijective manner, part $(iv)$ is also established. This completes the proof of the lemma.
\end{proof}

\begin{ex}
Consider $n=k=8$ as in Examples \ref{hatkmid} and \ref{hatkend}, and as the only dependence of the vectors $\vec{h}(C,I)$ and $\vec{h}_{C}^{I}$ on $C$ (and $\lambda$) are through the set $\operatorname{Dsp}^{c}(C)$, we assume that they are such for which this set is $D=\{2,5,6\}$ from those examples, of size $d-1=3$. Then for $I=\{2,3,5,5,6,7,7\}$ from the former example, the vector $\vec{h}(C,I) \in H_{C}^{8,0}$ is with $h_{3}=h_{5}=1$ and $h_{7}=2$ (and the other coordinates vanishing), of entry such $4=k-d$. With the multi-set $I_{b}:=\{0,2,5,5,5,6,8\}$ from the latter example, the vector is with non-zero entries $h_{5}=2$ and $h_{8}=1$, it also belongs to the subset $H_{C}^{7,0}$ associated with $k-1=7$, as arising from the multi-set $I_{b}\setminus\{0\}$. \label{exforhom}
\end{ex}

\begin{ex}
The set $\mathbb{N}_{n-1} \setminus D$, for $n=8$ and $D$ as in Examples \ref{hatkmid}, \ref{hatkend}, and \ref{exforhom}, is $\{1,3,4,7\}$. From the sequence 1235 in the former example, where we saw that $\hat{k}=6$, we need to find the subset $A$ of size $\hat{k}-d=2$ associated with the two initial entries 12 of the sequence. We skip 1 (since $h_{0}=0$), take the next element 3 (for $h_{1}=1$), skip 4, and complete with 7, and indeed we had $\hat{I}=\{3,7\}$ in the sequence in question. Then $n-\hat{k}=2$, so the 3 in the sequence comes from the element of $\hat{I}\cup\{0,n\}=\{0,2,3,5,6,7,8\}$ that is one before last (namely 7), and for 5 we go two elements back and pick up 5, yielding indeed the multi-set $\{5,7\}$ for $I\setminus\hat{I}$. This vector is in $H_{C}^{k,s}$ for all $s\leq3$, and indeed $\mathcal{OP}_{n,I}\subseteq\mathcal{OP}_{8,8,3}$. For 0000 and $\hat{k}=8$, we take the first four elements from $\mathbb{N}_{n-1} \setminus D$ for $A$, which indeed reproduces $I_{s}$, with $\mathcal{OP}_{n,I}\subseteq\mathcal{OP}_{8,8,8}$. Using 4668 and $\hat{k}=4=d$, with $\hat{I}_{b}\cup\{0,n\}=D\cup\{0,8\}=\{0,2,5,6,8\}$ and $n-\hat{k}=4$, the 4 comes from the maximal element, each 6 comes from the element 5 (with two entries from that set that are larger than it), and 0 is the element producing $n=8$, and indeed this sequence was the image of $I_{b}$ (here we can only take $s=0$). \label{hCItoI}
\end{ex}
The second case in Example \ref{hatkend}, as well as the last one in Example \ref{hCItoI}, also exemplify the fact that $r_{n}$ always equals $n-\hat{k}$ (in this case $8-4=4$), while we always have $r_{0}=n$ (here it is 8).

\medskip

We now define representations associated with multi-sets.
\begin{thm}
For $I$ be a multi-set as in Definition \ref{OPnksI}, let $R_{n,I}$ be the space generated by the polynomials $F_{C,T}^{I}$ from Definition \ref{defSpecht}, where $C$ runs over all the elements of $\bigcup_{\lambda \vdash n}\operatorname{CCT}(\lambda)$ for which the set $\operatorname{Dsp}^{c}(C)$ from Definition \ref{sets} is contained in $I$ in the sense of Definition \ref{multisets}, and $T\in\operatorname{SYT}(\lambda)$. Similarly, the space spanned by the polynomials $F_{C,T}^{I,\mathrm{hom}}$ for those $C$ and $T$ is denoted by $R_{n,I}^{\mathrm{hom}}$.
\begin{enumerate}[$(i)$]
\item Both $R_{n,I}$ and $R_{n,I}^{\mathrm{hom}}$ are representations of $S_{n}$. and they decompose as direct sums of the representations from Definition \ref{defSpecht}. Explicitly we have $R_{n,I}=\bigoplus_{\lambda \vdash n}\bigoplus_{C\in\operatorname{CCT}(\lambda),\ \operatorname{Dsp}^{c}(C) \subseteq I}V_{C}^{\vec{h}_{C}^{I}}$, while the representation $R_{n,I}^{\mathrm{hom}}$ equals $\bigoplus_{\lambda \vdash n}\bigoplus_{C\in\operatorname{CCT}(\lambda),\ \operatorname{Dsp}^{c}(C) \subseteq I}V_{C}^{\vec{h}(C,I)}$.
\item Both of these representations of $S_{n}$ are isomorphic to the representation $\mathbb{Q}[\mathcal{OP}_{n,I}]$ that is based on the set from Definition \ref{OPnksI}, with $R_{n,I}^{\mathrm{hom}}$ being homogeneous of degree $\sum_{i \in I}i$.
\item If we write every multi-set $I$ of size $k-1$ as $\{i_{h}\}_{h=1}^{k-1}$ and extend with $i_{0}=0$ and $i_{k}=n$, then the sum of $R_{n,I}$ over all those multi-sets $I$ for which $\{i_{h}\}_{h=0}^{s}$ is strictly increasing is a direct sum inside $\mathbb{Q}[\mathbf{x}_{n}]$ that projects bijectively onto $R_{n,k,s}$. For $s=0$ the same assertion holds for the $R_{n,I}^{\mathrm{hom}}$'s.
\item The direct sum yielding $R_{n,k,0}$ contains the one producing $R_{n,k-1,0}$. The complement is the direct sum of the representations $R_{n,I}^{\mathrm{hom}}$ for $I$ of size $k-1$ that do not contain 0.
\item If $s<k$ then the lift of $R_{n,k,s}$ is equal, in this decomposition, to the direct sum of the lift $R_{n,k,s+1}$ and $e_{n-s}$ times the lift of $R_{n,k-1,s}$.
\end{enumerate} \label{Rnksdecom}
\end{thm}
As Remark \ref{notation} allows us to view $R_{n,I}$ and $R_{n,I}^{\mathrm{hom}}$ as spanned by the polynomials $F_{w,I}$ and $F_{w,I}^{\mathrm{hom}}$ respectively, as $w$ runs over all the elements of $S_{n}$ whose associated set $\operatorname{Dsl}(w)$ is contained in $I$, Theorem \ref{Rnksdecom} extends Theorem 3.20 and Proposition 3.25 of \cite{[Z1]}. Part $(v)$ of Theorem \ref{Rnksdecom} shows that our construction respects the splitting sequence from Lemma \ref{splexseq}. The direct sum assertion for $s=0$ in part $(iii)$ there may be related to geometry via Remark \ref{cohomology} below.

\begin{proof}
For every $C$, the multiplier in the definition of $F_{C,T}^{I}$ and $F_{C,T}^{I,\mathrm{hom}}$ in Definition \ref{defSpecht} depends only on $C$ and $I$ and not on $T$. Hence for a fixed $C\in\operatorname{CCT}(\lambda)$ for some $\lambda \vdash n$, Theorem \ref{VMreps} implies that running over $T\in\operatorname{SYT}(\lambda)$ yields a basis for a representation. Comparing this multiplier with the vectors from that definition shows that $R_{n,I}$ and $R_{n,I}^{\mathrm{hom}}$ are indeed representations of $S_{n}$, and are given by the asserted sums from part $(i)$.

The degree of homogeneity of each $V_{C}^{\vec{h}(C,I)}$ is $\Sigma(C)+\sum_{r=1}^{n}h_{r}$, where part $(ii)$ of Lemma \ref{Dspc} gives $\Sigma(C)=\sum_{i\in\operatorname{Dsp}^{c}(M)}i$, and the second sum for the characteristic vector $\vec{h}(C,I)$ of the multi-set $\operatorname{Asp}^{c}_{I}(M)=I\setminus\operatorname{Dsp}^{c}(M)$, via Definition \ref{Spechtdef}, equals $\sum_{i\in\operatorname{Asp}^{c}_{I}(M)}i$. As the sum of these two expressions is clearly $\sum_{i \in I}i$, independently of $C$, the homogeneity from part $(ii)$ follows.

We now consider the sum from part $(iii)$, for some $0 \leq s\leq\min\{n,k\}$, where for every $\lambda \vdash n$ and $C\in\operatorname{CCT}(\lambda)$, we get a copy of $V_{C}^{\vec{h}_{C}^{I}}$ for every multi-set $I$ satisfying the asserted conditions and containing $\operatorname{Dsp}^{c}(M)$ in the usual sense. But parts $(iii)$ and $(iv)$ of Lemma \ref{bijvecs} imply that for any such $\lambda$ and $C$, the resulting sum is of $V_{C}^{\vec{h}}$ for all $h \in H_{C}^{k,s}$. Moreover, for $s=0$ part $(i)$ of Lemma \ref{bijvecs} implies that this is also the sum of $V_{C}^{\vec{h}(C,I)}$ for the same collection of multi-sets $I$, meaning that the resulting sum of the $R_{n,I}^{\mathrm{hom}}$'s is the same as the sum of the $R_{n,I}$'s.

But Theorem \ref{decomRnk} implies that the sums in question are direct, with the first one from the last paragraph lifting $R_{n,k,s}$. Hence the sums from part $(i)$, as partial to those sums, are also direct, establishing both this part and part $(iii)$. Moreover, as in the proof of Theorem 3.20 of \cite{[Z1]} (and Proposition 3.25 there), the multiplicity of $\mathcal{S}^{\lambda}$ in $R_{n,I}$ or in $R_{n,I}^{\mathrm{hom}}$ is the number of $C\in\operatorname{CCT}(\lambda)$ with $\operatorname{Dsp}^{c}(M) \subseteq I$.

Now, the bijection $\operatorname{ct}_{J}$ from Lemma \ref{ctJmulti}, for $J:=\{n-i\;|\;i \in I\}$, shows that this number is the same as the number of tableaux $S\in\operatorname{SYT}(\lambda)$ with $\operatorname{Dsi}(S) \subseteq J$. If $\alpha$ is the composition $\operatorname{comp}_{n}J\vDash_{w}n$ from Lemma \ref{mswcomp}, then the latter number is the Kostka number $K_{\lambda,\alpha}$.

Noting now that $S_{n}$ acts transitively on the orbit $\mathcal{OP}_{n,I}$, and that if the vector $\vec{m}:=\{m_{h}\}_{h=0}^{k-1}$ is $\operatorname{comp}_{n}I\vDash_{w}n$ (so that $\mathbb{Q}[\mathcal{OP}_{n,I}]$ is $W_{\vec{m}}$ in Definition \ref{OPnksI}), then the stabilizer of some element of that set is $\prod_{h=0}^{k-1}S_{m_{h}}$ (the possible vanishing of some of the entries of $\vec{m}$ does not affect this fact). Then $W_{\vec{m}}$ is isomorphic to $M^{\alpha}$ from \cite{[Sa]}, and Theorem 2.11.2 there implies that it contains the Specht module $\mathcal{S}^{\lambda}$ with the multiplicity $K_{\lambda,\alpha}$ as well. Part $(ii)$ is therefore also established.

When we write $R_{n,k,0}$ as the lift of the direct sum of $R_{n,I}^{\mathrm{hom}}$ over the multi-sets $I$ of size $k-1$, we recall from part $(ii)$ of Lemma \ref{bijvecs} that the vectors $\vec{h}(C,I)$ showing up in the summand associated with $I$ can also arise from a set of size $k-2$ if and only if $0 \in I$. Hence the direct sum of the $R_{n,I}^{\mathrm{hom}}$'s for such $I$ lifts $R_{n,k-1,0}$, and other $I$'s produce the complement, as desired for part $(iv)$.

Finally, the containment from part $(v)$ of Lemma \ref{bijvecs} (which is related to the one from part $(ii)$ of Lemma \ref{propOPnks}), and the complement from part $(vi)$ there, decompose the collection of multi-sets $I$ such that $R_{n,I}$ participates in the lift of $R_{n,k,s}$ (for $s<k$) into those showing up in the lift of $R_{n,k,s+1}$, and those that do not. Moreover, for a multi-set in the complement, every $\vec{h}_{C}^{I}$ is obtained by adding 1 at the location $n-s$, and since part $(v)$ of Lemma \ref{propOPnks} implies that $s\leq\hat{k}$, the element contributing it to the vector is from $I\setminus\hat{I}$.

It follows that the set contributing the corresponding element of $H_{C}^{k-1,s}$ from part $(vi)$ of Lemma \ref{bijvecs} has the same set $\hat{I}$, and we saw that all the vectors from that set are obtained in this way. Summing over all $\lambda$ and $C$, we see that the complement of the lift of $R_{n,k,s+1}$ inside that of $R_{n,k,s}$ is the image of the direct sum of all the $R_{n,I}$'s participating in the lift of $R_{n,k-1,s}$ under this addition to their vectors, which amounts to multiplication by $e_{n-s}$, yielding part $(v)$ as well. This completes the proof of the theorem.
\end{proof}
Theorem \ref{Rnksdecom} establishes, in particular, the direct sum assertion from Proposition 3.25 of \cite{[Z1]}, thus completing its proof.

\begin{rmk}
We can similarly construct the (generally non-homogeneous) sum $R_{n,k,s}^{\mathrm{hom}}$ of $R_{n,I}^{\mathrm{hom}}$ over $I$ with $\mathcal{OP}_{n,I}\subseteq\mathcal{OP}_{n,k,s}$, (generalizing the representation considered in Remark 3.26 of \cite{[Z1]}, which is obtained by taking $s=k$ here). Part $(iii)$ of Theorem \ref{Rnksdecom} implies that this sum is also direct, with part $(iv)$ being expressible as a decomposition of $R_{n,k,0}^{\mathrm{hom}}$ (which is the same as $R_{n,k,0}$ by part $(iii)$) as the ``older parts'' $R_{n,k-1,0}^{\mathrm{hom}}$ coming from smaller values of $k$, and the ``newer part'' $R_{n,k,1}^{\mathrm{hom}}$ that really require the index $k$. The decomposition showing up in Remark \ref{difreps} for $s=k-1$ implies that $R_{n,k,k-1}^{\mathrm{hom}}$ is the direct sum of $R_{n,k,k}^{\mathrm{hom}}$ and $e_{n}R_{n,k-1,k-1}^{\mathrm{hom}}$. However, analogues for part $(v)$ of Theorem \ref{Rnksdecom} for these sums of homogeneous representations do not hold. \label{Rnkshom}
\end{rmk}

\begin{rmk}
While the inequality $s \leq k$ is more natural than $s \leq n$ in Definition \ref{OPnksI}, note that with $s=n+1$ the quotient $R_{n,k,s}$ vanishes, as we divide also by the scalar $e_{n-s+1}=e_{0}$. However, if $k<n$ then for $s>k$ we may still get non-zero quotients, those they no longer have the direct sum decompositions from Theorem \ref{Rnksdecom} as it is stated. \label{bigs}
\end{rmk}
Recall from Examples 3.14 and 3.22 of \cite{[Z1]} that $R_{4,3,3}=R_{4,3}$ is the direct sum of the representations $V_{0000} \oplus V_{\substack{000 \\ 1\hphantom{11}}} \oplus V_{\substack{001 \\ 1\hphantom{12}}} \oplus V_{\substack{00 \\ 11}} \oplus V_{\substack{00 \\ 1\hphantom{1} \\ 2\hphantom{2}}}$ associated with the set $\{1,2\}$, $V_{0000}e_{1} \oplus V_{\substack{000 \\ 1\hphantom{11}}}e_{1} \oplus V_{\substack{011 \\ 1\hphantom{22}}} \oplus V_{\substack{01 \\ 12}} \oplus V_{\substack{01 \\ 1\hphantom{1} \\ 2\hphantom{2}}}$ coming from the set $\{1,3\}$ and $V_{0000}e_{1}^{2} \oplus V_{\substack{001 \\ 1\hphantom{12}}}e_{1} \oplus V_{\substack{011 \\ 1\hphantom{22}}}e_{1} \oplus V_{\substack{00 \\ 11}}e_{1} \oplus V_{\substack{02 \\ 1\hphantom{1} \\ 2\hphantom{2}}}$ arising from $\{2,3\}$. Going to the quotient $R_{4,3,4}$, with $s=4>3=k$, means dividing this representation by $e_{1}$, thus leaving all the part corresponding to $\{1,2\}$, but only 3 representations associated with $\{1,3\}$ and the last one from $\{2,3\}$, exemplifying Remark \ref{bigs}.

\begin{ex}
If $k=1$ then the only multi-set $I$ is the empty one. Hence the only $\lambda \vdash n$ and $C\in\operatorname{CCT}(\lambda)$ that shows up is the one having one line of zeros (for $\operatorname{Dsp}^{c}(C)$ to be empty and $\Sigma(C)=0$). This means that for every $n$, the representations $R_{n,1,1}=R_{n,1}$, $R_{n,1,0}$, $R_{n,1,1}^{\mathrm{hom}}=R_{n,1}^{\mathrm{hom}}$, and $R_{n,1,0}^{\mathrm{hom}}$ are the trivial representation of $S_{n}$ on constants. \label{k1triv}
\end{ex}

\begin{ex}
For $n=4$ and $k=2$, the values of $I$ for which we can take $s=2$ as well are the singletons $\{1\}$, $\{2\}$, and $\{3\}$ contained in $\mathbb{N}_{3}$. The examples from \cite{[Z1]} showed that the corresponding representations $R_{4,I}$ are \[V_{0000} \oplus V_{\substack{000 \\ 1\hphantom{11}}},\quad V_{0000}e_{1} \oplus V_{\substack{001 \\ 1\hphantom{12}}} \oplus V_{\substack{00 \\ 11}},\quad\mathrm{and}\quad V_{0000}e_{2} \oplus V_{\substack{011 \\ 1\hphantom{22}}}\] respectively, with the respective homogeneous counterparts being \[V_{0000}e_{1} \oplus V_{\substack{000 \\ 1\hphantom{11}}},\quad V_{0000}e_{2} \oplus V_{\substack{001 \\ 1\hphantom{12}}} \oplus V_{\substack{00 \\ 11}},\quad\mathrm{and}\quad V_{0000}e_{3} \oplus V_{\substack{011 \\ 1\hphantom{22}}}\] (indeed homogeneous of degrees 1, 2, and 3 respectively). The representation $R_{4,2,2}=R_{4,2}$ is the direct sum of the first three representations, with the direct sum of the latter three giving $R_{4,2,2}^{\mathrm{hom}}=R_{4,2}^{\mathrm{hom}}$. For $s=1$ we get one extra multi-set $\{4\}$ (as indeed, in the completed sequence 0,4,4 obtained by adding $i_{0}=0$ and $i_{2}=i_{k}=n=4$, the part involving $i_{0}=0$ and $i_{1}=4$ is strictly increasing), but then $\hat{I}$ from Definition \ref{multisets} is empty (so that $\hat{k}=1$), and for this $I$ we get $R_{n,I}=V_{0000}e_{3}$ and $R_{n,I}^{\mathrm{hom}}=V_{0000}e_{4}$ (since $r_{4}=3$ and the characteristic vector for $R_{n,I}^{\mathrm{hom}}$ involves an instance of 4). The remaining multi-set for $s=0$ is $I=\{0\}$ (and the completed sequence 0,0,4 has no increasing initial part), for which $R_{n,I}=V_{0000}e_{4}$ (since $r_{0}=n=4$) and $R_{n,I}^{\mathrm{hom}}=V_{0000}$. Both direct sums $R_{4,2,0}$ and $R_{4,2,0}^{\mathrm{hom}}$ thus contain the non-trivial representations showing up here, and one copy of the trivial representation $V_{0000}$ multiplied by $e_{r}$ for each $0 \leq r\leq4$ (with $e_{0}=1$ as usual). \label{n4k2}
\end{ex}

\begin{ex}
Recalling, for $n=4$, the multi-set $I=\{0,3,3\}$ from Example \ref{ordpartex}, we have $\hat{I}=\{3\}$ and $\hat{k}=2$, so that for $3 \in I\setminus\hat{I}$ we have $r_{3}=3$ and for 0 there we get $r_{0}=4$. Combining this with the expressions for this $\hat{I}$ appearing in Example \ref{n4k2}, we get $R_{4,I}=V_{0000}e_{2}e_{3}e_{4} \oplus V_{\substack{011 \\ 1\hphantom{22}}}e_{3}e_{4}$ and $R_{4,I}^{\mathrm{hom}}=V_{0000}e_{3}^{2} \oplus V_{\substack{011 \\ 1\hphantom{22}}}e_{3}$. Because of the multiplier $e_{4}$ in the former, it only shows up in $R_{4,4,0}$ (recall that $k=4$ as well) with $s=0$, while the multi-set $\{3,3\}$, obtained by omitting the 0, will give the same homogeneous counterpart, as in part $(iv)$ of Theorem \ref{Rnksdecom}. In fact, $R_{4,I}$ is obtained by multiplying the representation associated with this smaller multi-set, which is contained in $R_{4,3,0}$, by $e_{n-s}=e_{4}$, in correspondence with the fact that it is not contained in $R_{4,4,1}$ via part $(v)$ of that theorem. \label{exn4}
\end{ex}
Example \ref{exn4} exemplifies the fact, which is an immediate consequence of Theorem \ref{Rnksdecom}, that those multi-sets $I$ for which $R_{n,I}$ is only contained in the maximal representation $R_{n,k,0}$, with $s=0$, are precisely those for which $R_{n,I}^{\mathrm{hom}}$ arises from a multi-set of smaller size (because $0 \in I$).

\begin{ex}
In fact, the picture from Example \ref{n4k2} extends to every $n$ with $k=2$, to show that $R_{n,2,0}=R_{n,2,0}^{\mathrm{hom}}$ is the direct sum of the non-trivial representations associated with tableaux $C$ having $|\operatorname{Dsp}^{c}(C)|=1$, and one copy of the trivial representation multiplied by $e_{r}$ for each $0 \leq r \leq n$. This reproduce, in some sense, several of the results of \cite{[AAB]}. The parts added for $s=1$ and $s=0$ are similar to the case $n=4$ from that example, and in particular the complement of the constant one $R_{n,1,0}=R_{n,1,0}^{\mathrm{hom}}$ from Example \ref{k1triv} is the direct sum of $R_{n,I}^{\mathrm{hom}}$ for $I\neq\{0\}$, as in part $(iv)$ of Theorem \ref{Rnksdecom}. \label{k2gen}
\end{ex}
One can also verify that by taking $s=0$ and $s=1$ in Example \ref{k2gen}, the multiplier of the copy of the trivial representation from Example \ref{k1triv} is as asserted in part $(v)$ of Theorem \ref{Rnksdecom}.

\begin{ex}
With $n=3$, we have the four basic representations $U_{0}:=V_{000}$, $U_{1}:=V_{\substack{00 \\ 1\hphantom{1}}}$, $U_{2}:=V_{\substack{01 \\ 1\hphantom{2}}}$, and $U_{3}:=V_{\substack{0 \\ 1 \\ 2}}$, with the $\operatorname{Dsp}^{c}$-sets of the corresponding cocharge tableaux being empty, $\{1\}$, $\{2\}$, and $\{1,2\}$ respectively (and the index is the entry sum of that tableau). We take $k=4>3=n$, and state the representations showing up for each new value of $s$. The maximal value $s=3$ yields $I=\{1,2,3\}$, where $\hat{I}=\{1,2\}$ and thus \[R_{3,I}=U_{0} \oplus U_{1} \oplus U_{2} \oplus U_{3}\mathrm{\ and\ }R_{3,I}^{\mathrm{hom}}=U_{0}e_{1}e_{2}e_{3} \oplus U_{1}e_{2}e_{3} \oplus U_{2}e_{1}e_{3} \oplus U_{3}e_{3}.\] With $s=2$ we also have $I=\{1,2,2\}$, again with $\hat{I}=\{1,2\}$ and hence \[R_{3,I}=U_{0}e_{1} \oplus U_{1}e_{1} \oplus U_{2}e_{1} \oplus U_{3}e_{1}\mathrm{\ and\ }R_{3,I}^{\mathrm{hom}}=U_{0}e_{1}e_{2}^{2} \oplus U_{1}e_{2}^{2} \oplus U_{2}e_{1}e_{2} \oplus U_{3}e_{2},\] and the other two multi-sets \[I=\{1,3,3\}\mathrm{\ with\ }R_{3,I}=U_{0}e_{1}^{2} \oplus U_{1}e_{1}^{2}\mathrm{\ and\ }R_{3,I}^{\mathrm{hom}}=U_{0}e_{1}e_{3}^{2} \oplus U_{1}e_{3}^{2},\mathrm{\ and}\] \[I=\{2,3,3\}\mathrm{\ with\ }R_{3,I}=U_{0}e_{1}^{3} \oplus U_{2}e_{1}^{2}\mathrm{\ and\ }R_{3,I}^{\mathrm{hom}}=U_{0}e_{2}e_{3}^{2} \oplus U_{2}e_{3}^{2}.\] A new multi-sets that shows up for $s=1$ is $\{1,1,2\}$, yielding \[R_{3,I}=U_{0}e_{2} \oplus U_{1}e_{2} \oplus U_{2}e_{2} \oplus U_{3}e_{2}\mathrm{\ and\ }R_{3,I}^{\mathrm{hom}}=U_{0}e_{1}^{2}e_{2} \oplus U_{1}e_{1}e_{2} \oplus U_{2}e_{1}^{2} \oplus U_{3}e_{1},\] and there are the representations \[U_{0}e_{1}e_{2} \oplus U_{1}e_{1}e_{2},\ U_{0}e_{2}^{2} \oplus U_{1}e_{2}^{2},\ U_{0}e_{1}^{2}e_{2} \oplus U_{2}e_{1}e_{2},\ U_{0}e_{1}e_{2}^{2} \oplus U_{2}e_{2},\mathrm{\ and\ }U_{0}e_{2}^{3}\] associated with the multi-sets $\{1,1,3\}$, $\{1,1,1\}$, $\{2,2,3\}$, $\{2,2,2\}$, and $\{3,3,3\}$, with their respective homogeneous counterparts \[U_{0}e_{1}^{2}e_{3} \oplus U_{1}e_{1}e_{3},\ U_{0}e_{1}^{3} \oplus U_{1}e_{1}^{2},\ U_{0}e_{2}^{2}e_{3} \oplus U_{2}e_{2}e_{3},\ U_{0}e_{2}^{3} \oplus U_{2}e_{2}^{2},\mathrm{\ and\ }U_{0}e_{3}^{3}.\] The new multi-sets appearing for $s=0$ are those obtained by adding 0 to the multi-sets associated with $k=3$, and the latter are $\{1,2\}$, $\{1,3\}$, $\{1,1\}$, $\{2,3\}$, $\{2,2\}$, $\{3,3\}$, $\{0,1\}$, $\{0,2\}$, $\{0,3\}$, and $\{0,0\}$. Using these we get can obtain the expressions for $R_{3,3,0}=R_{3,3,0}^{\mathrm{hom}}$, with the latter completing $R_{3,4,0}^{\mathrm{hom}}$ and the former multiplied by $e_{3}$ yields $R_{3,4,0}$, and we can again compare them as in part $(iii)$ of Theorem \ref{Rnksdecom}. \label{n3k4}
\end{ex}
Example \ref{n3k4} exemplifies our results in the generality of Remark \ref{kngen}.

\medskip

We can now show how this construction produces a decomposition of $\mathbb{Q}[\mathbf{x}_{n}]_{d}$ that lifts the formula for the multiplicities from Proposition \ref{Adlambda}.
\begin{cor}
Given $d\geq0$ and $\lambda \vdash n$, every pair consisting of an element $C\in\operatorname{CCT}(\lambda)$ and a multi-set $I$ of positive integers not exceeding $n$ such that $\operatorname{Dsp}^{c}(C) \subseteq I$ and $\sum_{i \in I}=d$ contributes one copy of an irreducible sub-representation of $\mathbb{Q}[\mathbf{x}_{n}]_{d}$ that is isomorphic to $\mathcal{S}^{\lambda}$. \label{multCI}
\end{cor}

\begin{proof}
For some $k>d$, since the quotient $R_{n,k,0}$ is obtained by division by the powers $x_{i}^{k}$, all of which are homogeneous of degree larger than $d$, we deduce that $\mathbb{Q}[\mathbf{x}_{n}]_{d}$ is isomorphic to the part of that representation that is homogeneous of degree $d$. Part $(iii)$ of Theorem \ref{Rnksdecom} identifies that representation with $R_{n,k,0}^{\mathrm{hom}}$ from Remark \ref{Rnkshom}, and when we decompose it into the direct sum of the representations $R_{n,\tilde{I}}^{\mathrm{hom}}$ for the multi-sets $\tilde{I}$ of size $k-1$ consisting of integers between 0 and $n$, the part that is homogeneous of degree $d$ is the direct sum of the representations associated with those multi-sets $I$ for which $\sum_{i\in\tilde{I}}=d$.

Now, part $(ii)$ of Lemma \ref{bijvecs} (or the definition of the homogeneous representations via Theorem \ref{Rnksdecom} and Definition \ref{Spechtdef}) implies that for every such multi-set $\tilde{I}$ we have $R_{n,\tilde{I}}^{\mathrm{hom}}=R_{n,I}^{\mathrm{hom}}$ for a multi-set $I$ of positive integers with $\sum_{i \in I}=d$ (obtained by simply removing all the zeros from $\tilde{I}$). Conversely, since $d<k$, every multi-set of positive integers that sum to $d$ has size at most $k-1$, and can thus be completed to a multi-set of size $k-1$, by adding the appropriate number of zeros. Hence the direct sum obtained from Theorem \ref{Rnksdecom} is the same as the asserted one. This proves the corollary.
\end{proof}

\begin{ex}
Take $n=4$ and $d=2$. Then there are two multi-sets of entry sum $d=2$, which are $\{1,1\}$ and $\{2\}$. With the former multi-set $I$ we have $\hat{I}=\{1\}$, and then $R_{4,I}^{\mathrm{hom}}$ is the same as $R_{4,\hat{I}}^{\mathrm{hom}}$ which appears explicitly in Example \ref{n4k2} but multiplied by $e_{1}$ arising from the element of $I\setminus\hat{I}$. The homogeneous representation associated with the latter multi-set (or just set) is also given in that example. Altogether, Corollary \ref{multCI} yields the decomposition
\[\mathbb{Q}[\mathbf{x}_{4}]_{2}=V_{0000}e_{1}^{2} \oplus V_{\substack{000 \\ 1\hphantom{11}}}e_{1} \oplus V_{0000}e_{2} \oplus V_{\substack{001 \\ 1\hphantom{12}}} \oplus V_{\substack{00 \\ 11}},\] corresponding to the elements of the set from Proposition \ref{Adlambda} in this case. \label{n4d2}
\end{ex}

\begin{ex}
We now consider the case where $n=2$ and $d=4$. Here there are three multi-sets of positive integers that sum to $d=4$ and cannot exceed $n=2$, which are $\{1,1,1,1\}$, $\{1,1,2\}$, and $\{2,2\}$. Now, there are only two partitions of 2, each one with a single cocharge tableau, one yielding the trivial representation $U_{+}$ and has an empty $\operatorname{Dsp}^{c}$-set, and the other one produces the sign representation $U_{-}$ and has the $\operatorname{Dsp}^{c}$-set $\{1\}$. We deduce that the last multi-set only yields a single representation $U_{+}e_{2}^{2}$, while the first multi-set gives the representation $U_{+}e_{1}^{4} \oplus U_{-}e_{1}^{3}$, and the remaining one produces $U_{+}e_{1}^{2}e_{2} \oplus U_{-}e_{1}e_{2}$. Altogether we get \[\mathbb{Q}[\mathbf{x}_{2}]_{4}=U_{+}e_{1}^{4} \oplus U_{+}e_{1}^{2}e_{2} \oplus U_{+}e_{2}^{2} \oplus U_{-}e_{1}^{3} \oplus U_{-}e_{1}e_{2},\] which respects the decomposition of that space into the symmetric and anti-symmetric polynomials, and is also related to the five elements showing up in Proposition \ref{Adlambda} for these parameters. \label{n2d4}
\end{ex}
Examples \ref{n4d2} and \ref{n2d4} hint at the different behavior when $n>d$ and when $n<d$, about which we expound in Examples \ref{n5d2} and \ref{n3d4} below.

\begin{rmk}
The proof of the decomposition of $R_{n,n}$ in \cite{[ATY]} used a pairing into the highest degree $\frac{n(n-1)}{2}$, which amounts to Poincar\'{e} duality in the cohomology ring of the flag variety from \cite{[B]} (most likely pairing $V^{S}$ with $V^{S^{t}}$ in our decomposition, though I did not verify this in general). Indeed, this variety is complete of dimension $\frac{n(n-1)}{2}$, and the top degree cohomology group is a 1-dimensional space on which $S_{n}$ acts via the sign. No such duality result exists for the varieties $X_{n,k}$ from \cite{[PR]}, since they are non-complete, and indeed their dimension is $n(k-1)$ (as they are open subvarieties of $(\mathbb{P}^{k-1})^{n}$), and the maximal degree in $R_{n,k}$ is smaller (note that for $k=n$ this dimension is twice that of the flag variety). For $s>0$ the varieties $X_{n,k,s}$ are also non-complete of the same dimension, but when $s=0$ it is the full complete variety $(\mathbb{P}^{k-1})^{n}$, and the isomorphism between the cohomology ring of that variety (extended to $\mathbb{Q}$) with $R_{n,k,0}$ is clear via the K\"{u}nneth formula. As $S_{n}$ acts on $X_{n,k,0}=(\mathbb{P}^{k-1})^{n}$, it also operates on the cohomology ring $R_{n,k,0}$, and the top degree cohomology group becomes, under this isomorphism, the one spanned by $e_{n}^{k-1}$. Poincar\'{e} duality there produces a pairing on $R_{n,k,0}$, but it is not orthogonal on the different $V_{C}^{\vec{h}}$'s in general, in the sense that most components pair non-trivially with more than one component. It might be interesting to find a decomposition of $R_{n,k,0}$ into sub-representations that are orthogonal with respect to this pairing. \label{cohomology}
\end{rmk}

\medskip

We now use the representations from Definition \ref{defSpecht} to get a good decomposition of $\mathbb{Q}[\mathbf{x}_{n}]_{d}$ into representations of $S_{n}$, lifting the formula from Proposition \ref{Kostka}. The argument parallels Section 2.10 of \cite{[Sa]}, which works in a more abstract setting, but is carried out in the setting of the action on polynomials.
\begin{lem}
The sum of the representations $V_{M}$ from Definition \ref{defSpecht} over the elements $M\in\operatorname{SSYT}_{d}(\lambda)$ is a direct sum. \label{FMTindep}
\end{lem}

\begin{proof}
We begin by taking some $\lambda \vdash n$ and some tableau $T$ of shape $\lambda$ and content $\mathbb{N}_{n}$, and we claim that $\{F_{M,T}\;|\;M\in\operatorname{SSYT}_{d}(\lambda)\}$ are linearly independent. To see this, we recall the vector $\deg_{T}y$ associated with every monomial $y$ in Definition \ref{orderT}, as well as the partial order $>_{T}$ defined there. Proposition \ref{Spechtpols} expresses $F_{M,T}$, for any $M\in\operatorname{SSYT}_{d}(\lambda)$, as $\sum_{\sigma \in C(T)}\operatorname{sgn}(\sigma)\sigma p_{M,T}$ plus a sum of monomials $y$, all of which satisfy $p_{M,T}>_{T}y$.

Consider thus a linear combination $\sum_{M\in\operatorname{SSYT}_{d}(\lambda)}a_{M}F_{M,T}$, with not all the coefficients being 0, and let $\vec{d}$ be the maximal vector, in the reverse lexicographic order, that is obtained as $\deg_{T}p_{M,T}$ for some $M\in\operatorname{SSYT}_{d}(\lambda)$ with $a_{M}\neq0$. Since the order on vectors is a total order, our linear combination produces $\sum_{M\in\operatorname{SSYT}_{d}(\lambda),\ \deg_{T}p_{M,T}=\vec{d}}\sum_{\sigma \in C(T)}a_{M}\operatorname{sgn}(\sigma)\sigma p_{M,T}$ plus a combination of monomials all of whose $\deg_{T}$-vectors are smaller than $\vec{d}$.

But if a monomial $y$ is obtained as $\sigma p_{M,T}$ for some $M\in\operatorname{SSYT}_{d}(\lambda)$ and $\sigma \in C(T)$, then for every $1 \leq j\leq\lambda_{1}$ the exponents of the variables $x_{i}$ with $C_{T}(i)=j$ in $y$ are all distinct, and their ordering determines the action of $\sigma$ on them. By gathering the information from all $j$, we deduce which $\sigma$ produced our monomial, and once we know $\sigma$ we easily determine $p_{M,T}$ and from it the tableau $M$. It follows that the sums $\sum_{\sigma \in C(T)}\operatorname{sgn}(\sigma)\sigma p_{M,T}$ over various $M\in\operatorname{SSYT}_{d}(\lambda)$ are supported on distinct sets of monomials, and therefore the last sum from the previous paragraph cannot vanish, as we assume that at least one $a_{M}$ with $M\in\operatorname{SSYT}_{d}(\lambda)$ that satisfies $\deg_{T}p_{M,T}=\vec{d}$ does not vanish. Hence $\sum_{M\in\operatorname{SSYT}_{d}(\lambda)}a_{M}F_{M,T}\neq0$, which proves the claim.

Recalling from Theorem \ref{VMreps} that $\{F_{M,T}\;|\;T\in\operatorname{SYT}(\lambda)\}$ is a basis for $V_{M}$, a linear dependence between the $V_{M}$'s with $M\in\operatorname{SSYT}_{d}(\lambda)$ would be an equality of the sort $\sum_{M\in\operatorname{SSYT}_{d}(\lambda)}\sum_{T\in\operatorname{SYT}(\lambda)}c_{M,T}F_{M,T}=0$. Letting elements of $S_{n}$ act, and taking any linear combination of the resulting expressions, produces the equality $\sum_{M\in\operatorname{SSYT}_{d}(\lambda)}\sum_{T\in\operatorname{SYT}(\lambda)}c_{M,T}\alpha F_{M,T}=0$ for every element $\alpha$ in the group ring $\mathbb{Q}[S_{n}]$.

But the latter group ring decomposes as $\bigoplus_{\nu \vdash n}\operatorname{End}_{\mathbb{Q}}(\mathcal{S}^{\nu})$, and for each $T\in\operatorname{SYT}(\lambda)$, there is an element $\alpha_{T}\in\operatorname{End}_{\mathbb{Q}}(\mathcal{S}^{\lambda})\subseteq\mathbb{Q}[S_{n}]$ which takes the basis element associated with $T$ to itself and the other basis elements to 0 (for some shapes $\lambda$ this element $\alpha_{T}$ is the idempotent multiple of $\varepsilon_{T}$, but this is not the case in general---see \cite{[Ste]} for more details). Taking $\alpha=\alpha_{T}$ in our formula thus leaves $F_{M,T}$ invariant for every $M$ and annihilates the other basis vectors of $V_{M}$, and thus produces the equality $\sum_{M\in\operatorname{SSYT}_{d}(\lambda)}c_{M,T}F_{M,T}=0$ with our $T$. But our claim then implies that $c_{M,T}=0$ for all $M\in\operatorname{SSYT}_{d}(\lambda)$, meaning that in total our linear combination was trivial. This proves the lemma.
\end{proof}

\medskip

Note that given a tableau $T$ of some shape $\lambda \vdash n$ and content $\mathbb{N}$, and any monic monomial in $\mathbb{Q}[\mathbf{x}_{n}]$, we can construct a tableau $H$ of shape $\lambda$, with no additional properties, by putting in the box $v_{T}(i)$ of $\lambda$ the exponent with which $x_{i}$ shows up in the monomial in question. Thus, once $T\in\operatorname{SYT}(\lambda)$ is fixed, we can write any monomial as $p_{H,T}$, by extending the notation from Definition \ref{Spechtdef}, where $H$ is an arbitrary tableau of shape $\lambda$.

Using this terminology, we prove the following auxiliary result.
\begin{lem}
Take some monomial $z\in\mathbb{Q}[\mathbf{x}_{n}]$.
\begin{enumerate}[$(i)$]
\item There exists a tableau $H$ of shape $\lambda$ whose entries are non-decreasing along each row such that $\varepsilon_{T}z$ equals $\varepsilon_{T}p_{H,T}$.
\item For $H$ as in part $(i)$, set $\vec{d}:=\deg_{T}p_{H,T}$ and $d$ to be the entry sum of $H$, or equivalently of $\vec{d}$. Then, if the entries in each column of $H$ are distinct, then there exists some $M\in\operatorname{SSYT}_{d}(\lambda)$, of the same content as $H$, such that $M$ and $H$ are related by an element $\rho \in C(T)$.
\item If $H$, $\vec{d}$, and $d$ are as in part $(ii)$, then there is an element $M\in\operatorname{SSYT}_{d}(\lambda)$ and a scalar $c$ such that $\varepsilon_{T}p_{H,T}-c\varepsilon_{T}p_{M,T}$ is supported only on monomials $y$ that satisfy $p_{H,T}>_{T}y$.
\end{enumerate} \label{FHTFMT}
\end{lem}

\begin{proof}
As we saw before the lemma, there is a tableau $\tilde{H}$ such that $z=p_{\tilde{H},T}$. Recall that $\varepsilon_{T}$ is the product of two operators, with the first one to act being $\sum_{\tau \in R(T)}\tau$ as in Lemma \ref{operRT}. As this operator yields the same image on $p_{\tilde{H},T}$ and on its $\tau$-image for any $\tau \in R(T)$, we take $\tau$ to be an element ordering the entries in each row of $\tilde{H}$ in a non-decreasing order, and then if $H=\tau\tilde{H}$ then $H$ has the desired property and our operator takes $z=p_{\tilde{H},T}$ and $p_{H,T}$ to the same polynomial. By applying the second operator $\sum_{\sigma \in C(T)}\operatorname{sgn}(\sigma)\sigma$ (as in Proposition \ref{Spechtpols}), part $(i)$ follows.

Next, consider a column of $H$ as in part $(ii)$, and since we assume that its entries are distinct, we can define the number of inversions in that column to be the number of pairs of entries there such that the upper box has a larger value than the lower one (this is the same as the usual inversion number of a permutation in one-line notation, by identifying the entries of that column with the numbers between 1 and the length of that column in an order-preserving manner). We define the \emph{column inversion number} of $H$ to be the sum of the inversions in the column of $H$.

We prove part $(ii)$ by induction on the column inversion number of $H$, with the base case being the observation that if this number is 0 then $H$ is already semi-standard itself, and we can take $\rho$ to be trivial. Assume now that this number is positive, and it then suffices to prove that there is an element of $C(T)$ taking $H$ to a tableau with non-decreasing rows and a smaller column inversion number. Indeed, the induction hypothesis implies that some element of $C(T)$ takes the latter tableau to an element of $\operatorname{SSYT}_{d}(\lambda)$, and by composing the two elements of $C(T)$, the result for $H$ follows.

So let $c$ be the leftmost column of $H$ with a positive contribution to its column inversion number, and we write the entry of $H$ in the box of row $i$ and column $j$ as $h_{i,j}$. Since $c$ is not ordered increasingly, there is some $r\geq1$ such that $h_{r+1,c}<h_{r,c}$, and note that if $c>1$ then $h_{r+1,c-1}>h_{r,c-1}$ by the minimality of $c$. We define the index $l\geq1$ to be such that $h_{r+1,c+j}<h_{r,c+j}$ for every $0 \leq j<l$, but either $h_{r+1,c+l}>h_{r,c+l}$, or the box $(r+1,c+l)$ is not in $\lambda$.

The element that we take interchanges the boxes $(r,c+j)$ and $(r+1,c+j)$ for every $0 \leq j<l$, and leaves the remaining boxes invariant (this resembles the action of a Garnir element---see, e.g., Section 2.6 of \cite{[Sa]}). It is clearly in $C(T)$, and we denote the resulting tableau by $\tilde{H}$, with entries $\tilde{h}_{i,j}$. Since our operation removed one adjacent inversion in each column it affected, the column inversion number is reduced by this operation, namely that of $\tilde{H}$ is smaller than that of $H$ (by exactly $l$). We only have to show that the rows of $\tilde{H}$ are also non-decreasing.

All the rows except for $r$ and $r+1$ are the same as those of $H$ hence remain non-decreasing, so that take any $k>1$, and we have to show that $\tilde{h}_{r,k-1}\leq\tilde{h}_{r,k}$ and $\tilde{h}_{r+1,k-1}\leq\tilde{h}_{r+1,k}$, provided that the boxes in each pair are in $\lambda$. If $c>2$ and $k<c$, or if $k>c+l$, then these inequalities are just $h_{r,k-1} \leq h_{r,k}$ and $h_{r+1,k-1} \leq h_{r+1,k}$, which hold by our assumption on $H$. If $l\geq2$ and $c<k \leq c+l$ then the first inequality is $h_{r+1,k-1} \leq h_{r+1,k}$ and the second one is $h_{r,k-1} \leq h_{r,k}$, and again the assumption on $H$ yields the result. It remains to consider the cases $k=c$ when $c>1$, as well as $k=c+l$.

But if $c>1$ the assumptions on $H$ and $c$ yield $h_{r,c-1}<h_{r+1,c-1} \leq h_{r+1,c}$ and $h_{r+1,c-1} \leq h_{r+1,c}<h_{r,c}$ which give $\tilde{h}_{r,c-1}<\tilde{h}_{r,c}$ and $\tilde{h}_{r+1,c-1}\leq\tilde{h}_{r+1,c}$ respectively. Similarly, when the box $(r,c+l)$ is in $\lambda$ we get the inequalities $\tilde{h}_{r,c+l-1}=h_{r+1,c+l-1}<h_{r+1,c+l-1} \leq h_{r,c+l}=\tilde{h}_{r,c+l}$, and if $(r+1,c+l)$ is also in $\lambda$ then $\tilde{h}_{r+1,c+l-1}=h_{r,c+l-1} \leq h_{r,c+l}<h_{r+1,c+l}=\tilde{h}_{r+1,c+l}$ by the assumptions on $H$ and $c$. Hence $\tilde{H}$ has all the desired properties, and the inductive argument establishes part $(ii)$.

Next, Lemma \ref{operRT} only used the row property of semi-standard tableaux, so Proposition \ref{Spechtpols} implies that $\varepsilon_{T}p_{H,T}$ is a scalar times $\sum_{\sigma \in C(T)}\operatorname{sgn}(\sigma)\sigma p_{H,T}$ plus a linear combination of monomials that have $\deg_{T}$-vectors smaller than $\vec{d}$ (the scalar multiple would be $s_{H,T}$, defined analogously to $s_{M,T}$ from Definition \ref{Spechtdef}, as here we did not divide by it). If some column in $H$ contains the same entry twice, then the former sum vanishes, and we can take any $M$ with the scalar $c=0$, yielding the desired result in part $(iii)$ for this situation.

Finally, when each column of $H$ contains distinct entries, let $M\in\operatorname{SSYT}_{d}(\lambda)$ be the element obtained from part $(ii)$, and since the action of $C(T)$ does not affect the $\deg_{T}$-vectors of monomials, we deduce that $\deg_{T}p_{M,T}=\vec{d}$ as well. Proposition \ref{Spechtpols} then expresses $\varepsilon_{T}p_{M,T}$ as $s_{M,T}\sum_{\sigma \in C(T)}\operatorname{sgn}(\sigma)\sigma p_{M,T}$ plus a combination of monomials that are smaller than $p_{M,T}$ and $p_{H,T}$ in our order, and the fact that $M$ and $H$ are related by $\rho \in C(T)$ implies that $\sum_{\sigma \in C(T)}\operatorname{sgn}(\sigma)\sigma p_{H,T}$ equals $\sum_{\sigma \in C(T)}\operatorname{sgn}(\sigma)\sigma p_{M,T}$ times the sign $\operatorname{sgn}(\rho)$.

But then when we consider the combination $\varepsilon_{T}p_{H,T}-c\varepsilon_{T}p_{M,T}$, where $c$ is $\operatorname{sgn}(\rho)s_{H,T}/s_{M,T}$, then we saw that both terms involve monomials with $\deg_{T}$-vector $\vec{d}$ and others that are strictly smaller, and when we compare the parts with $\deg_{T}$-value $\vec{d}$ in both summands, we deduce from the previous paragraph that they cancel out. Hence the combination has the asserted property from part $(iii)$ also in this case. This completes the proof of the lemma.
\end{proof}
Note that in general $\varepsilon_{T}p_{H,T}$ will not be a multiple of $\varepsilon_{T}p_{M,T}$ in the proof of Lemma \ref{FMTindep}, as indeed $\varepsilon_{T}\rho$ need not the same as $\varepsilon_{T}$ for $\rho \in C(T)$ (this is also visible in the resulting stabilizers possibly being of different sizes), as we see in Example \ref{CTrelex} below. The proof of that lemma also shows what happens with the construction from Definitions \ref{defSpecht} and \ref{Spechtdef} and Theorem \ref{VMreps} in case we do not assume $M$ to be semi-standard.

\begin{ex}
We take $\lambda=31\vdash4$, $d=5$, and the tableaux \[T:=\begin{ytableau} 1 & 3 & 4 \\ 2 \end{ytableau},\ \tilde{H}:=\begin{ytableau} 1 & 3 & 1 \\ 0 \end{ytableau},\ H:=\begin{ytableau} 1 & 1 & 3 \\ 0 \end{ytableau},\mathrm{\ \ and\ \ }M:=\begin{ytableau} 0 & 1 & 3 \\ 1 \end{ytableau}.\] Then $H$ is the element in the $R(T)$-orbit of $\tilde{H}$ that has non-decreasing rows, and we have \[\varepsilon_{T}p_{\tilde{H},T}=\varepsilon_{T}p_{H,T}=2(x_{1}x_{3}x_{4}^{3}-x_{2}x_{3}x_{4}^{3}+x_{1}x_{3}^{3}x_{4}-x_{2}x_{3}^{3}x_{4}+x_{1}^{3}x_{3}x_{4}-x_{2}^{3}x_{3}x_{4}),\] with the first term being the multiple of $p_{H,T}$, and the third one involving $p_{\tilde{H},T}$. The $\deg_{T}$-vectors of the monomials involved are the maximal one $\vec{d}=113$ for the first two, followed by 131 twice, and then 311 in the last two summands, and the external scalar 2 is $s_{H,T}$. Applying the unique non-trivial element $\rho \in C(T)$ takes $H$ to $M\in\operatorname{SSYT}_{5}(\lambda)$, where $s_{M,T}=1$ and we get \[\varepsilon_{T}p_{M,T}\!=\!x_{2}x_{3}x_{4}^{3}-x_{1}x_{3}x_{4}^{3}+x_{2}x_{3}^{3}x_{4}-x_{1}x_{3}^{3}x_{4}+x_{1}^{3}x_{2}x_{4}-x_{1}x_{2}^{3}x_{4}+x_{1}^{3}x_{2}x_{3}-x_{1}x_{2}^{3}x_{3},\] which is also the value of $F_{M,T}$. The first four terms in $\varepsilon_{T}p_{H,T}$ are $-2$ times those in $\varepsilon_{T}p_{M,T}$ (including the two elements yielding $\vec{d}$ as their $\deg_{T}$-images), and the remaining terms of the latter have $\deg_{T}$-vectors 401 and 410. We have \[\varepsilon_{T}p_{H,T}+2\varepsilon_{T}p_{M,T}=2(x_{1}^{3}x_{3}x_{4}-x_{2}^{3}x_{3}x_{4}+x_{1}^{3}x_{2}x_{4}-x_{1}x_{2}^{3}x_{4}+x_{1}^{3}x_{2}x_{3}-x_{1}x_{2}^{3}x_{3}),\] indeed supported on monomials $y$ such that $p_{H,T}>_{T}y$ and $p_{M,T}>_{T}y$. \label{CTrelex}
\end{ex}

\medskip

We can now establish the decomposition corresponding to Proposition \ref{Kostka}, as well as the one associated with the subspaces showing up in its proof.
\begin{thm}
The space $\mathbb{Q}[\mathbf{x}_{n}]_{d}$ equals $\bigoplus_{\lambda \vdash n}\bigoplus_{M\in\operatorname{SSYT}_{d}(\lambda)}V_{M}$. Moreover, for each content $\mu$ of sum $d$, the sub-representation $\mathbb{Q}[\mathbf{x}_{n}]_{\mu}$ of $\mathbb{Q}[\mathbf{x}_{n}]_{d}$ is the partial sum $\bigoplus_{\lambda \vdash n}\bigoplus_{M\in\operatorname{SSYT}_{\mu}(\lambda)}V_{M}$. \label{FMTdecom}
\end{thm}

\begin{proof}
Since every $V_{M}$ for $M\in\operatorname{SSYT}_{d}(\lambda)$ is isomorphic to $\mathcal{S}^{\lambda}$ by Theorem \ref{VMreps}, the direct sum from Lemma \ref{FMTindep} is contained in the $\mathcal{S}^{\lambda}$-isotypical part of $\mathbb{Q}[\mathbf{x}_{n}]_{d}$ as a representation of $S_{n}$. Since this space is the direct sum of its $\mathcal{S}^{\lambda}$-isotypical part over $\lambda \vdash n$ (by complete reducibility of representations of finite groups in characteristic 0), the sum over $\lambda$ of the expressions from Lemma \ref{FMTindep} is indeed a direct sum, as asserted.

It therefore suffices to show that for each $\lambda \vdash n$, the associated direct sum produces the full $\mathcal{S}^{\lambda}$-isotypical part of the space in question. Moreover, we may fix some $T\in\operatorname{SYT}(\lambda)$, and note that every representation in that part contains a vector with non-trivial $\alpha_{T}$-image, where $\alpha_{T}$ is the idempotent from the proof of Lemma \ref{FMTindep} (by the basis property of $\mathcal{S}^{\lambda}$), and that vector generates a copy of $\mathcal{S}^{\lambda}$ over $S_{n}$ inside this space. Moreover, it is clear from the construction that any $\alpha_{T}$-image is also an $\varepsilon_{T}$-image, so we are reduced to verifying that any $\varepsilon_{T}$-image of a polynomial from $\mathbb{Q}[\mathbf{x}_{n}]_{d}$ lies in the desired direct sum.

We thus consider the $\varepsilon_{T}$-image of a monomial there, which we can take, by part $(i)$ of Lemma \ref{FHTFMT}, to be $p_{H,T}$ for some tableau $H$ of shape $\lambda$ with non-decreasing rows. We argue by induction on $\vec{d}:=\deg_{T}p_{H,T}$, where we assume that all the $\varepsilon_{T}$-images that are supported on polynomials with $\deg_{T}$-vectors that are smaller than $\vec{d}$ are in the direct sum in question (the base case is, of course, that this direct sum contains 0).

But part $(iii)$ of that lemma also produces some $M\in\operatorname{SSYT}_{d}(\lambda)$ such that $\varepsilon_{T}p_{H,T}$ equals some multiple of $\varepsilon_{T}p_{M,T}$, or equivalently of its multiple $F_{M,T}$ from Definition \ref{Spechtdef}, plus a linear combination of monomials $y$ for which $\deg_{T}y$ is smaller than $\vec{d}$ in our order. Since the latter combination is an $\varepsilon_{T}$-image (as a linear combination of such), it lies in our direct sum by the induction hypothesis, and as $F_{M,T} \in V_{M}$ by definition, $\varepsilon_{T}p_{H,T}$ lies in that direct sum as well, as desired.

This establishes the first assertion, and for the second one we note that if $M\in\operatorname{SSYT}_{\mu}(\lambda)$ then $p_{M,T}\in\mathbb{Q}[\mathbf{x}_{n}]_{\mu}$ for every $T\in\operatorname{SYT}(\lambda)$ by Definition \ref{Qxnd}, and as the latter is a sub-representation of $\mathbb{Q}[\mathbf{x}_{n}]_{d}$, we deduce that it contains the image $F_{M,T}$ of $p_{M,T}$ under an element of $\mathbb{Q}[S_{n}]$. This shows that $V_{M}\subseteq\mathbb{Q}[\mathbf{x}_{n}]_{\mu}$ in this case, so that the second assertion follows from the first. This completes the proof of the theorem.
\end{proof}
The decomposition of $\mathbb{Q}[\mathbf{x}_{n}]_{d}$ as the direct sum of $\mathbb{Q}[\mathbf{x}_{n}]_{\mu}$ over all contents $\mu$ of sum $d$, as was seen in the proof of Proposition \ref{Kostka}, combines with Theorem \ref{FMTdecom} to indeed show how the multiplicity of $\mathcal{S}^{\lambda}$ in each component is the corresponding Kostka number, as was used in the proof of that proposition, and of Theorem \ref{Rnksdecom} here, as well as Theorem 3.20 and Proposition 3.25 of \cite{[Z1]}. Note that in Example \ref{CTrelex}, the expression for $\varepsilon_{T}p_{H,T}+2\varepsilon_{T}p_{M,T}$ equals minus twice the higher Specht polynomial having the same second index $T$ and the first index $\begin{ytableau} 0 & 1 & 1 \\ 3 \end{ytableau}$, which has the same content as $M$ and $H$, exemplifying Theorem \ref{FMTdecom} and the decomposition into contents there.

\begin{ex}
With the parameters $n=4$ and $d=2$ from Example \ref{n4d2}, the decomposition from Theorem \ref{FMTdecom} becomes \[\mathbb{Q}[\mathbf{x}_{4}]_{2}=V_{0002} \oplus V_{\substack{000 \\ 2\hphantom{22}}} \oplus V_{0011} \oplus V_{\substack{001 \\ 1\hphantom{12}}} \oplus V_{\substack{00 \\ 11}},\] with the first two summands generating the part associated with one content, and the other two summands yielding the part with the other content. In case $n=2$ and $d=4$, as in Example \ref{n2d4}, we get \[\mathbb{Q}[\mathbf{x}_{2}]_{4}=V_{04} \oplus V_{\substack{0 \\ 4}} \oplus V_{13} \oplus V_{\substack{1 \\ 3}} \oplus V_{22},\] again with the separations into contents visible, and where the action of $S_{2}$ being trivial when the tableau is horizontal, and via the sign in case it is vertical. \label{nd42VM}
\end{ex}
The decompositions in Examples \ref{n4d2} and \ref{n2d4} are not the same as those of the respective space in Example \ref{nd42VM}. Indeed, $e_{1}^{2}$ in the first one and $e_{1}^{4}$ in the second one involve monomials of different contents, which are thus not supported on a single representation via Theorem \ref{FMTdecom}.

\section{Operations Between Representations \label{OpersReps}}

We now establish relations between the constructions from the previous section for $n$ and for $n+1$. In order to do this, we recall Definitions 2.11, 4.6, and 4.11 of \cite{[Z1]}, and generalize some of them as follows.
\begin{defn}
Take a partition $\lambda \vdash n$ for some positive integer $n$.
\begin{enumerate}[$(i)$]
\item The set $\operatorname{EC}(\lambda)$ of \emph{external corners} of $\lambda$ consists of the boxes that do not lie in the Ferrers diagram of $\lambda$, but adding them to it yields the Ferrers diagram of a partition.
\item For any $v\in\operatorname{EC}(\lambda)$, we denote by $\lambda+v \vdash n+1$ the partition whose Ferrers diagram is the union of $\lambda$ and $v$. If $v$ is the element $(1,\lambda_{1}+1)\in\operatorname{EC}(\lambda)$, then we write $\lambda_{+}$ for this $\lambda+v$.
\item Given $T$ is of shape $\lambda$ and content $\mathbb{N}_{n}$ and $v\in\operatorname{EC}(\lambda)$, adding $v$ to $\lambda$ and putting $n+1$ in it yields the tableau that we denote by $T+v$, and for $v$ in the first row we write it as $\iota T$. It lies in $\operatorname{SYT}(\lambda+v)$ (or $\operatorname{SYT}(\lambda_{+})$) in case $T\in\operatorname{SYT}(\lambda)$.
\item For $S\in\operatorname{SYT}(\lambda)$ and $v\in\operatorname{EC}(\lambda)$ we set $S\tilde{+}v:=\operatorname{ev}(\operatorname{ev}S+v)$, which lies in $\operatorname{SYT}(\lambda+v)$, and in particular $\tilde{\iota}S:=\operatorname{ev}(\iota\operatorname{ev}S)\in\operatorname{SYT}(\lambda_{+})$.
\item For $M\in\operatorname{SSYT}(\lambda)$, write it as $\operatorname{ct}_{J}(S)$ as in Lemma \ref{ctJmulti}, and then for $v\in\operatorname{EC}(\lambda)$ we set $\delta_{M,v}$ to be 0 if the row of $v$ is strictly below the row $R_{\operatorname{ev}S}(n)$ containing $n$ in $\operatorname{ev}S$ (a situation that we describe as $v$ \emph{lying below $n$ in $\operatorname{ev}S$}), and define it to be 1 otherwise.
\item Given a multi-set $J$, we write $J_{+}$ for the multi-set $\{j+1\;|\;i \in J\}$. Similarly $M_{+}$ for $M\in\operatorname{SSYT}(\lambda)$ is the element obtained by adding 1 to each of the entries of $M$. It is also in $\operatorname{SSYT}(\lambda)$.
\item We define $M\hat{+}v$, for $M=\operatorname{ct}_{J}(S)\in\operatorname{SSYT}(\lambda)$, and $v\in\operatorname{EC}(\lambda)$, to be $\operatorname{ct}_{J_{+}}(S\tilde{+}v)$ in case $\delta_{M,v}=1$, and $\operatorname{ct}_{J_{+}\cup\{1\}}(S\tilde{+}v)$ (the union being a multi-set sum by adding a singleton) when $\delta_{M,v}=0$. In particular we set $\hat{\iota}M$ to be $\operatorname{ct}_{J_{+}}(\tilde{\iota}S)$.
\item Given $M\in\operatorname{SSYT}(\lambda)$, the \emph{extension} $\operatorname{Ext}_{S_{n}}^{S_{n+1}}V_{M}$ is $\bigoplus_{v\in\operatorname{EC}(\lambda)}\delta_{M,v}V_{M\hat{+}v}$, and the \emph{lower enlargement} $\operatorname{LE}_{S_{n}}^{S_{n+1}}V_{M}$ is $\bigoplus_{v\in\operatorname{EC}(\lambda)}(1-\delta_{M,v})V_{M\hat{+}v}$. In case $M=C\in\operatorname{CCT}(\lambda)$ we also define, for every $0 \leq t \leq n+1$, the \emph{induction} $\operatorname{Ind}_{t,S_{n}}^{S_{n+1}}V_{C}:=e_{t}\operatorname{Ext}_{S_{n}}^{S_{n+1}}V_{C}\oplus\operatorname{LE}_{S_{n}}^{S_{n+1}}V_{C}$.
\item For applying $\operatorname{Ext}_{S_{n}}^{S_{n+1}}$, $\operatorname{LE}_{S_{n}}^{S_{n+1}}$, or $\operatorname{Ind}_{t,S_{n}}^{S_{n+1}}$ to a representation of the sort $V_{C}^{\vec{h}}$ for a vector $\vec{h}:=\{h_{r}\}_{r=1}^{n}$, we divide by $\prod_{r=1}^{n}e_{r}^{h_{r}}\in\mathbb{Q}[\mathbf{x}_{n}]$, apply the operator to $V_{C}$, and multiply the result by the same symmetric function, but now from $\mathbb{Q}[\mathbf{x}_{n}]$. Moreover, applying any of these operators to a representation that is given as a direct sum of components of the sort $V_{M}$ or $V_{C}^{\vec{h}}$ yields the direct sum of the images of the components under the operator in question.
\end{enumerate} \label{plusv}
\end{defn}
As we saw in Lemma 4.8 of \cite{[Z1]} and its proof that the set $\operatorname{Dsi}(S\tilde{+}v)$ equals $\operatorname{Dsi}(S)_{+}\cup\{1\}$ when $v$ lies below $n$ in $\operatorname{ev}S$ and just $\operatorname{Dsi}(S)_{+}$ otherwise (in the notation from Definition \ref{plusv}), the condition $\operatorname{Dsi}(S) \subseteq J$ in the sense of Definition \ref{multisets} implies that $\operatorname{Dsi}(S\tilde{+}v)$ in indeed contained in $J_{+}\cup\{1\}$ in the former case and in $J_{+}$ in the latter, making $M\hat{+}v$ from Definition \ref{plusv}, and with it $\hat{\iota}M$ (with whose $v$ we always have $\delta_{M,v}=1$), well-defined. We can write it as $\operatorname{ct}_{J_{+}\cup\{1\}}\big(\operatorname{ct}_{J}^{-1}(M)\tilde{+}v\big)$ and $\operatorname{ct}_{J_{+}}\big(\operatorname{ct}_{J}^{-1}(M)\tilde{+}v\big)$ respectively in these cases. Lemma 4.13 of \cite{[Z1]} explains how $\operatorname{Ind}_{t,S_{n}}^{S_{n+1}}$ lifts the usual branching rule for induction of irreducible representations of $S_{n}$ to $S_{n+1}$, and the division involved when acting on $V_{C}^{\vec{h}}$ is by the maximal symmetric divisor from Remark 3.3 there.

Similar considerations, and the proofs of Lemmas 2.13 and 4.8 of \cite{[Z1]}, yield the following properties, generalizing parts of these lemmas.
\begin{lem}
Fix some $\lambda \vdash n$ and an element $M\in\operatorname{SSYT}(\lambda)$.
\begin{enumerate}[$(i)$]
\item The tableau $\hat{\iota}M$ is obtained by adding to $M$ the box yielding $\lambda_{+}$ from $\lambda$, pushing all the entries in the first row of $M$ one box to the right, and filling the freed upper left corner with a 0.
\item The data part $(iii)$ of Lemma \ref{Dspc} for $\hat{\iota}M$ is the same as that for $M$, except for replacing $n$ by $n+1$.
\item The content of $M_{+}$ is obtained from that of $M$ by adding 1 to each entry, and we have $F_{M_{+},T}=e_{n}F_{M,T}$ for every tableau $T$ of shape $\lambda$ and content $\mathbb{N}_{n}$, as well as $V_{M_{+}}=e_{n}V_{M}$.
\item The content of $M\hat{+}v$ when $\delta_{M,v}=1$, and in particular the content of $\hat{\iota}M$, is the content of $M$ plus one more instance of 0, and that of $M\hat{+}v$ in case $\delta_{M,v}=0$ is the content of $M_{+}$ together with one instance of 0.
\item The multi-set $\operatorname{Dsp}^{c}(M\hat{+}v)$ for some $v\in\operatorname{EC}(\lambda)$ coincides with $\operatorname{Dsp}^{c}(M)$ in case $\delta_{M,v}=1$, so that in particular we have $\operatorname{Dsp}^{c}(\hat{\iota}M)=\operatorname{Dsp}^{c}(M)$. If $\delta_{M,v}=0$ then $\operatorname{Dsp}^{c}(M\hat{+}v)$ is the multi-set addition of $\operatorname{Dsp}^{c}(M)$ and the singleton $\{n\}$.
\item The multi-set $\operatorname{Dsp}^{c}(M_{+})$ is $\operatorname{Dsp}^{c}(M)$ plus one instance of $n$.
\item The number $\Sigma(M\hat{+}v)$ is $\Sigma(M)$ when $\delta_{M,v}=1$ and $\Sigma(M)+n$ in case $\delta_{M,v}=0$. We also have $\Sigma(\hat{\iota}M)=\Sigma(M)$ and $\Sigma(M_{+})=\Sigma(M)+n$.
\end{enumerate} \label{SSYTct}
\end{lem}

\begin{proof}
Set $S:=\operatorname{ct}_{J}^{-1}(M)$ for the multi-set $J$ for which the content of $M$ is represented by $\operatorname{comp}_{n}J\vDash_{w}n$. We recall from Lemma 2.13 of \cite{[Z1]} that $\tilde{\iota}S$ is obtained by considering $S_{+}$, adding to $\lambda$ the box for $\lambda_{+}$, moving the entries of the first row of $S_{+}$ to the right, and completing with 1 in the upper left box. Since $\operatorname{Dsi}(S) \subseteq J$ (in the sense of Definition \ref{multisets}, as always), it is clear from this construction that $\operatorname{Dsi}(\tilde{\iota}S) \subseteq J_{+}$, and by applying $\operatorname{ct}_{J_{+}}$, in order to get $\hat{\iota}M$ via Definition \ref{plusv}, we get the asserted tableau. This proves part $(i)$, of which part $(ii)$ is an immediate consequence.

All the assertions in part $(iii)$ are immediate from Definition \ref{plusv}, since $p_{M_{+},T}$ is then the product of $P_{M,T}$ and the product $e_{n}$ of all the variables involved, which is symmetric. We now recall the relation between the multi-set showing up in the index of $\operatorname{ct}$ to get a semi-standard Young tableau and its $\operatorname{Dsp}^{c}$-set from Definition \ref{sets}, and that $\lambda+v \vdash n+1$. Observing that $\{n+1-i\;|\;i \in J_{+}\}$ is the same as $\{n-i\;|\;i \in J\}$ (and we omit the zeros in both) and $(n+1)-1=n$ implies part $(iv)$. Part $(v)$ then follows directly from part $(i)$ of Lemma \ref{Dspc}, as does part $(vi)$, and part $(vii)$ is a consequence of the previous two parts by part $(ii)$ of Lemma \ref{Dspc}. This proves the lemma.
\end{proof}

\begin{ex}
We take $\lambda=431\vdash8$, $J=\{1,3,4,4,6,6,7\}$, and set \[T:=\begin{ytableau} 1 & 2 & 4 & 7 \\ 3 & 6 & 8 \\ 5 \end{ytableau},S:=\begin{ytableau} 1 & 3 & 4 & 8 \\ 2 & 5 & 6 \\ 7 \end{ytableau},\mathrm{\ and\ }M:=\begin{ytableau} 0 & 1 & 2 & 7 \\ 1 & 4 & 4 \\ 6 \end{ytableau},\] so that $S=\operatorname{ct}_{J}^{-1}(M)$, $T=\operatorname{ev}S$, and $\operatorname{Dsp}^{c}(M)=\{1,2,2,4,4,5,7\}$, with $\Sigma(M)=25$ (with the relations between $\operatorname{Dsp}^{c}(M)$, $\Sigma(M)$, and the content of $M$ being in correspondence with Lemma \ref{Dspc}). The set $\operatorname{EC}(\lambda)$ has four boxes, and when $v$ is the one in the first row the tableaux $T+v$, $S\tilde{+}$, and $M\hat{+}v$ are \[\iota T=\begin{ytableau} 1 & 2 & 4 & 7 & 9 \\ 3 & 6 & 8 \\ 5 \end{ytableau},\ \tilde{\iota}S=\begin{ytableau} 1 & 2 & 4 & 5 & 9 \\ 3 & 6 & 7 \\ 8 \end{ytableau},\mathrm{\ and\ }\hat{\iota}M=\begin{ytableau} 0 & 0 & 1 & 2 & 7 \\ 1 & 4 & 4 \\ 6 \end{ytableau}\] respectively, in accordance with Lemma \ref{SSYTct} and the value $\delta_{M,v}=1$. When $v$ in the second row, we still have $\delta_{M,v}=1$, and we get \[T+v=\begin{ytableau} 1 & 2 & 4 & 7 \\ 3 & 6 & 8 & 9 \\ 5 \end{ytableau},\ S\tilde{+}v=\begin{ytableau} 1 & 2 & 4 & 5 \\ 3 & 6 & 7 & 9 \\ 8 \end{ytableau}\mathrm{\ and\ }M\hat{+}v=\begin{ytableau} 0 & 0 & 1 & 2 \\ 1 & 4 & 4 & 7 \\ 6 \end{ytableau}.\] The next element $v$, in the third row, now yields $\delta_{M,v}=0$, and the tableaux are \[T+v=\begin{ytableau} 1 & 2 & 4 & 7 \\ 3 & 6 & 8 \\ 5 & 9 \end{ytableau},\ S\tilde{+}v=\begin{ytableau} 1 & 4 & 5 & 9 \\ 2 & 6 & 7 \\ 3 & 8 \end{ytableau}\mathrm{\ and\ }M\hat{+}v=\begin{ytableau} 0 & 2 & 3 & 8 \\ 1 & 5 & 5 \\ 2 & 7 \end{ytableau},\] again in correspondence with Lemma \ref{SSYTct}. The last box adds another row, with $\delta_{M,v}=0$ as usual, and the tableaux are now \[T+v=\begin{ytableau} 1 & 2 & 4 & 7 \\ 3 & 6 & 8 \\ 5 \\ 9 \end{ytableau},\ S\tilde{+}v=\begin{ytableau} 1 & 4 & 5 & 9 \\ 2 & 6 & 7 \\ 3 \\ 8 \end{ytableau}\mathrm{\ and\ }M\hat{+}v=\begin{ytableau} 0 & 2 & 3 & 8 \\ 1 & 5 & 5 \\ 2 \\ 7 \end{ytableau}.\] Finally, we have $M_{+}=\begin{ytableau} 1 & 2 & 3 & 8 \\ 2 & 5 & 5 \\ 7 \end{ytableau}=\operatorname{ct}_{J\cup\{0\}}(S)$, with the content of the two tableaux $M\hat{+}v$ with $\delta_{M,v}=0$ being related that of $M_{+}$ as in Lemma \ref{SSYTct}. \label{hatplus}
\end{ex}

\medskip

We will use Definition \ref{plusv} and Lemma \ref{SSYTct} for investigating the relations between representations with different parameters. We begin with those from Theorem \ref{Rnksdecom}, extending the results from Section 4 of \cite{[Z1]} to our more general case. By extending Definition \ref{multisets}, we say that one multi-set $I$ is contained in another multi-set $\tilde{I}$ if the multiplicity of every number in $\tilde{I}$ is at least its multiplicity in $I$, and the difference is the multi-set in which the multiplicity of every element is the difference between its multiplicities in $I$ and in $\tilde{I}$. In order to see how $R_{n,I}$ and $R_{n,I}^{\mathrm{hom}}$ are related to $R_{n,\tilde{I}}$ and $R_{n,\tilde{I}}^{\mathrm{hom}}$ respectively in this case, we consider the case where the difference consists of a single element, with multiplicity 1, yielding the following extension of Propositions 4.1 of \cite{[Z1]}.
\begin{prop}
Assume that the multi-set $\tilde{I}$ is the multi-set addition of $I$ and a single element $\ell$.
\begin{enumerate}[$(i)$]
\item If $\ell$ already lies in $\hat{I}$ or equals 0 or $n$, then we have $R_{n,\tilde{I}}=e_{r_{\ell}}R_{n,I}$ and $R_{n,\tilde{I}}^{\mathrm{hom}}=e_{\ell}R_{n,I}^{\mathrm{hom}}$.
\item When $\ell\in\mathbb{N}_{n-1}\setminus\hat{I}$, the representation $R_{n,\tilde{I}}^{\mathrm{hom}}$ is $e_{\ell}R_{n,I}^{\mathrm{hom}}$ plus the direct sum of $V_{C}^{\vec{h}(C,\tilde{I})}$ over those $C$ for which $\ell\in\operatorname{Dsp}^{c}(C)\subseteq\tilde{I}$.
\item For $\ell$ as in part $(ii)$, in order to get $R_{n,\tilde{I}}$ we take the the direct sum of $V_{C}^{\vec{h}_{C}^{I}}$ over the same set of cocharge tableaux $C$, plus the image of $R_{n,I}$ under multiplication by $e_{r_{\ell}}$ and changing every $e_{r_{i}}$ for $\ell<i \in I$ to $e_{r_{i}-1}$.
\end{enumerate} \label{incI}
\end{prop}

\begin{proof}
We follow the proof of Proposition 4.1 of \cite{[Z1]}. When $\ell\in\hat{I}\cup\{0,n\}$, its addition to $\tilde{I}$ leaves $\hat{I}$ invariant and increases the multi-set difference $\tilde{I}\setminus\hat{I}$, so that the condition on $C$ in part $(i)$ of Theorem \ref{Rnksdecom} remains invariant, and we only add 1 to $h_{r_{\ell}}$ for $\vec{h}_{C}^{I}$ and to $h_{\ell}$ in $\vec{h}(C,\tilde{I})$, yielding part $(i)$. Otherwise $\ell$ is as in part $(ii)$, the new subset of $\mathbb{N}_{n-1}$ is now $\hat{I}\cup\{\ell\}$ (with $I\setminus\hat{I}$ remaining invariant), so that the cocharge tableaux $C$ with $\ell\in\operatorname{Dsp}^{c}(C)\subseteq\tilde{I}$ are now also allowed, and as the effect on the vectors $\vec{h}_{C}^{I}$ for $C$ with $\operatorname{Dsp}^{c}(C) \subseteq I$ is increasing $h_{\ell}$ by 1 once again, part $(ii)$ follows.

For part $(iii)$, the new tableaux are as in part $(ii)$, and for the older one we need to consider the effect on $r_{i}$ from Definition \ref{Spechtdef} for each $i \in I$ (and then $i\neq\ell$, and the factor $e_{r_{\ell}}$ is accounted for). When $i\in\hat{I}$ we get, as in Proposition 4.1 of \cite{[Z1]}, that if $i<\ell$ it remains the same, and when $i>\ell$ it decreases by 1. If $i \in I\setminus\hat{I}$ (a multi-set that remains unaffected, as we saw), then we note that $\hat{k}$ increases by 1, while the size of the set of $j\in\hat{I}\cup\{n\}$ that are larger than $i$ also increases by 1 when $i<\ell$ (due to the new element $\ell$) but remains unchanged for $i>\ell$. This shows that $r_{i}$ decreases by 1 for all $i>\ell$, but remains invariant for any $i<\ell$, thus establishing part $(iii)$. This proves the proposition.
\end{proof}

\begin{ex}
We take $n=3$ and $k=3$, set $I:=\{2,2\}$, and use the notation from Example \ref{n3k4}, so that $R_{3,I}=U_{0}e_{1}e_{2} \oplus U_{2}e_{2}$ and $R_{3,I}^{\mathrm{hom}}=U_{0}e_{2}^{2} \oplus U_{2}e_{2}$. When $\tilde{I}$ is $\{2,2,3\}$ the formulae for $R_{3,\tilde{I}}$ and $R_{3,\tilde{I}}^{\mathrm{hom}}$ are given in that example and give $e_{1}R_{3,I}$ and $e_{3}R_{3,I}^{\mathrm{hom}}$, and for $\tilde{I}=\{2,2,2\}$ we get from that example that $R_{3,\tilde{I}}=e_{2}R_{3,I}$ and $R_{3,\tilde{I}}^{\mathrm{hom}}=e_{2}R_{3,I}^{\mathrm{hom}}$, as part $(i)$ of Proposition \ref{incI} predicts (the formulae for $\tilde{I}=\{0,2,2\}$ would be $e_{2}R_{3,I}$ and $R_{3,I}^{\mathrm{hom}}$ respectively). When we take $\ell=1$ and get to $\tilde{I}=\{1,2,2\}$, the representation $R_{3,\tilde{I}}^{\mathrm{hom}}$ is $e_{1}R_{3,I}^{\mathrm{hom}}$ plus the two new summands $U_{1}e_{2}^{2}$ and $U_{3}e_{2}$, in correspondence with part $(ii)$ of that proposition. Finally, the new terms in $R_{3,\tilde{I}}$ are $U_{1}e_{1}$ and $U_{3}e_{1}$, based on the same representations, but the older terms $U_{0}e_{1}$ and $U_{2}$ are obtained by both the index 1 corresponding to $i=2\in\hat{I}$ and the index 2 associated with $i=2 \in I\setminus\hat{I}$ (note that $\hat{I}=\{2\}$ here) reduced by 1, and the multiplier associated with $r_{\ell}=0$ is also trivial, as part $(iii)$ of that proposition asserts. A similar behavior is seen when we start with the multi-set $\{1,1\}$, yielding the non-homogeneous representation $U_{0}e_{2} \oplus U_{1}e_{2}$ and the homogeneous counterpart $U_{0}e_{1}^{2} \oplus U_{1}e_{1}$---adding 3, 1, or 0 gives multiples (the formulae for the multi-sets resulting from the adding the first two expressions show up in Example \ref{n3k4} as well), and the homogeneous representation associated with the increased multi-set $\{1,1,2\}$ is $e_{2}$ times the older one plus $U_{2}e_{1}^{2} \oplus U_{3}e_{1}$. In the non-homogeneous setting, the new representations yield $U_{2}e_{2} \oplus U_{3}e_{2}$, the multipliers associated with both instances of $1 \in I$ now remain the same (compare the cases $i<\ell$ and $i>\ell$ in part $(iii)$ of Proposition \ref{incI}), and the extra multiplier associated with $r_{2}=0$ is again trivial. \label{addtoIex}
\end{ex}

\medskip

We now see how Propositions 4.16 and 4.17 of \cite{[Z1]} extend to this setting, and add some other constructions that now exist.
\begin{prop}
Let $I$ be as above, with the set $\hat{I}$ from Definition \ref{multisets}.
\begin{enumerate}[$(i)$]
\item The representation $R_{n+1,I\cup\{n\}}^{\mathrm{hom}}$ is $\operatorname{Ind}_{n,S_{n}}^{S_{n+1}}R_{n,I}^{\mathrm{hom}}$.
\item The representation $R_{n+1,I}^{\mathrm{hom}}$ equals $\operatorname{Ext}_{S_{n}}^{S_{n+1}}R_{n,I}^{\mathrm{hom}}$ when $n \not\in I$, and if $n \in I$, with multi-set difference $I\setminus\{n\}$, this representation is  $\operatorname{Ind}_{n,S_{n}}^{S_{n+1}}R_{n,I\setminus\{n\}}^{\mathrm{hom}}$, which is the direct sum of $\operatorname{Ext}_{S_{n}}^{S_{n+1}}R_{n,I}^{\mathrm{hom}}$ and $\operatorname{LE}_{S_{n}}^{S_{n+1}}R_{n,I\setminus\{n\}}^{\mathrm{hom}}$.
\end{enumerate} \label{mapsmulti}
\end{prop}

\begin{proof}
For part $(i)$ we follow the proof of Proposition 4.16 of \cite{[Z1]}. Part $(v)$ of Lemma \ref{SSYTct} implies that if $\operatorname{Dsp}^{c}(C) \subseteq I$ for some $C\in\operatorname{CCT}(\lambda)$ and $\lambda \vdash n$ then $\operatorname{Dsp}^{c}(C\hat{+}v) \subseteq I\cup\{n\}$ for every $v\in\operatorname{EC}(\lambda)$, and since $\operatorname{Dsp}^{c}(C)$ for such $C$ is a set (rather than just a multi-set) and is contained in $\mathbb{N}_{n-1}$, we deduce that it equals $\operatorname{Dsp}^{c}(C\hat{+}v)\cap\mathbb{N}_{n-1}$ for any such $v$. The fact that the complementary multi-set $I\cup\{n\}\setminus\operatorname{Dsp}^{c}(C\hat{+}v)$ coincides with $I\setminus\operatorname{Dsp}^{c}(C)$ when $\delta_{C,v}=0$ and is the union of the latter with $\{n\}$ in case $\delta_{C,v}=1$ (by the same result) combines with Definition \ref{Spechtdef} to yield part $(i)$.

Next, we compare $R_{n+1,I}^{\mathrm{hom}}$ with $R_{n+1,I\cup\{n\}}^{\mathrm{hom}}$ via Proposition 4.17 of \cite{[Z1]}. If $n \in I$ then part $(i)$ of Proposition \ref{incI} (with $\ell=n$) implies that there is only a factor of $e_{n}$ between them, and the second assertion in part $(ii)$ follows (it can also be viewed as part $(i)$ applied to $I\setminus\{n\}$). The decomposition there follows from the definition of the operators involved via Definitions \ref{Spechtdef} and \ref{plusv}.

We thus assume that $n \not\in I$ where part $(ii)$ there (again with $\ell=n$) expresses $R_{n+1,I\cup\{n\}}^{\mathrm{hom}}$, as in the proof of Proposition 4.17 of \cite{[Z1]}, as $e_{n}R_{n+1,I}^{\mathrm{hom}}$ plus some terms. Using the same part of Lemma \ref{SSYTct} we see that the added terms inside the expression $\operatorname{Ind}_{n,S_{n}}^{S_{n+1}}R_{n,I}^{\mathrm{hom}}$ are those associated with $C\hat{+}v$ for which $\delta_{C,v}=0$, namely the added part is $\operatorname{LE}_{S_{n}}^{S_{n+1}}R_{n,I}^{\mathrm{hom}}$. Via Definition \ref{Spechtdef} we thus compare $e_{n}R_{n+1,I}^{\mathrm{hom}}$ with $e_{n}\operatorname{Ext}_{S_{n}}^{S_{n+1}}R_{n,I}^{\mathrm{hom}}$, and the first expression from part $(ii)$ follows as well.
This proves the proposition.
\end{proof}

\begin{rmk}
The differences in part $(iii)$ of Proposition \ref{incI} allows us to state the results analogous to Proposition \ref{mapsmulti} also for the non-homogeneous representations. Explicitly, $R_{n+1,I\cup\{n\}}$ is obtained from $\operatorname{Ind}_{n-\hat{k},S_{n}}^{S_{n+1}}R_{n,I}$ by replacing each multiplier $e_{r_{i}}$ for $i \in I\setminus\hat{I}$ by $e_{r_{i}+1}$, and $R_{n+1,I}$ is $\operatorname{Ext}_{S_{n}}^{S_{n+1}}R_{n,I}$ after we replace any $e_{r_{i}}$ for $i \in I\setminus\hat{I}$, except for one instance of $n$ if there is one, by $e_{r_{i}+1}$ (note the effect of increasing $n$ to $n+1$ on the parameters from Definition \ref{Spechtdef} to see this). Thus the neater results appear only for the homogeneous representations from Proposition \ref{mapsmulti} and for the non-homogeneous ones from Propositions 4.16 and 4.17 of \cite{[Z1]} in case $I$ is a set (and not a general multi-set), with the small extra compatibility that exists when $I$ is a set that contains $n$. Indeed, the infinite case from \cite{[Z2]} requires either the limits of all the homogeneous representations or only those of the non-homogeneous ones associated with sets (the expressions mentioned here mean that limits of non-homogeneous representations having actual multi-sets as indices do not exist). \label{onlysetsRnI}
\end{rmk}

\begin{rmk}
The constructions from Propositions 4.16 and 4.17 of \cite{[Z1]}, which are the two cases of Proposition \ref{mapsmulti} in case $I$ is a subset of $\mathbb{N}_{n-1}$, were called, as in \cite{[HRS]}, star and bar insertions. In that case the bar insertion always created a new set of size 1, and the star insertion always increased the last set, to be larger than a singleton. Here, some sets in the partition corresponding to $I$ (namely the sizes in $\vec{m}=\operatorname{comp}_{n}I\vDash_{w}n$) are allowed to be empty, but the bar insertion, corresponding to adding $n$ to $I$ when going from $n$ to $n+1$, still adds a new singleton at the end, and indeed case $(i)$ of Proposition \ref{mapsmulti} is the direct generalization of Proposition 4.16 of \cite{[Z1]} for the homogeneous representations (though not for the non-homogeneous ones, as stated in Remark \ref{onlysetsRnI}). The star insertion, increasing $n$ by 1 and leaving $I$ invariant, increases the last set, but the latter can either be empty and become non-empty (corresponding to a bar insertion on another multi-set), or be non-empty and increase further, explaining the two cases in part $(ii)$ of Proposition \ref{mapsmulti}. \label{nostarbar}
\end{rmk}

\begin{ex}
If we take $I=\{0,3,3\}$ as in Examples \ref{ordpartex} and \ref{exn4}, then for $n=3$ we get $R_{3,I}=V_{000}e_{2}^{2}e_{3}$ and $R_{3,I}^{\mathrm{hom}}=V_{000}e_{3}^{2}$. Similarly, the set $\{0,3\}$ produces the non-homogeneous representation $V_{000}e_{2}e_{3}$, and the homogeneous one which is $V_{000}e_{3}$. Applying $\operatorname{Ind}_{3,S_{3}}^{S_{4}}$ to the latter representation produces the representation $R_{4,I}^{\mathrm{hom}}$ from the latter example (as both parts of Proposition \ref{mapsmulti} suggest), and for getting $R_{4,I}$ appearing there, we have $\hat{k}=1$ and all the indices arise from $I\setminus\hat{I}$ (since $\hat{I}$ is empty), so that we apply $\operatorname{Ind}_{2,S_{3}}^{S_{4}}$ and increase both indices 2 and 3 to 3 and 4 respectively, as in Remark \ref{onlysetsRnI}. \label{justtriv}
\end{ex}
Example \ref{justtriv} shows how the fact that $n=3$ was in the multi-set affects the construction. We now turn to an example where this is not the case.
\begin{ex}
Recall the two multi-sets $\{1,1\}$ and $\{2\}$ from Example \ref{n4d2}, which are those consisting of positive integers that sum to 2. The same considerations showing up there produce the equality \[\mathbb{Q}[\mathbf{x}_{5}]_{2}=V_{00000}e_{1}^{2} \oplus V_{\substack{0000 \\ 1\hphantom{111}}}e_{1} \oplus V_{00000}e_{2} \oplus V_{\substack{0001 \\ 1\hphantom{112}}} \oplus V_{\substack{000 \\ 11\hphantom{1}}},\] which is the direct sum of $R_{5,I}^{\mathrm{hom}}$ for the two multi-sets $I$ (the first two terms with $\{1,1\}$, the last three for $\{2\}$). Comparing it with the expression for that example, we see that ours is the $\operatorname{Ext}_{S_{4}}^{S_{5}}$-image of the one from there, with $\operatorname{Ext}_{S_{4}}^{S_{5}}$ taking each irreducible component associated with $C$ and $I$ to the one corresponding to $\hat{\iota}C$ and $I$. \label{n5d2}
\end{ex}

\begin{ex}
Comparing the notation from Example \ref{n2d4} with that from Examples \ref{n3k4} and \ref{addtoIex}, we get $\operatorname{Ext}_{S_{2}}^{S_{3}}U_{+}=U_{0}$ and $\operatorname{Ext}_{S_{2}}^{S_{3}}U_{-}=U_{1}$, as well as $\operatorname{LE}_{S_{2}}^{S_{3}}U_{+}=U_{2}$ and $\operatorname{LE}_{S_{2}}^{S_{3}}U_{-}=U_{3}$. Thus to $R_{2,I}^{\mathrm{hom}}=U_{+}e_{1}^{4} \oplus U_{-}e_{1}^{3}$ we have to apply $\operatorname{Ext}_{S_{2}}^{S_{3}}$ when we move from $n=2$ to $n+1=3$, while for the other two multi-sets there we remove $n=2$ and let $\operatorname{Ind}_{2,S_{2}}^{S_{3}}$ act on the representations associated with $\{1,1\}$ and $\{2\}$, which are $U_{+}e_{1}^{2} \oplus U_{-}e_{1}$ and $U_{+}e_{2}$ respectively. This produces the sub-representation \[U_{0}e_{1}^{4} \oplus U_{1}e_{1}^{3} \oplus U_{0}e_{1}^{2}e_{2} \oplus U_{2}e_{1}^{2} \oplus U_{1}e_{1}e_{2} \oplus U_{3}e_{1} \oplus U_{0}e_{2}^{2} \oplus U_{2}e_{2}^{2}\subseteq\mathbb{Q}[\mathbf{x}_{3}]_{4}.\] The space $\mathbb{Q}[\mathbf{x}_{3}]_{4}$ is completed by adding in the representation $U_{0}e_{1}e_{3} \oplus U_{1}e_{3}$, corresponding to the other multi-set $\{1,3\}$ of entry sum 4, since now the value $n+1=3$ is also allowed to appear in it. \label{n3d4}
\end{ex}

\medskip

We deduce the following generalization of Corollary 4.19 of \cite{[Z1]}.
\begin{prop}
Take $n\geq1$ and $k\geq0$.
\begin{enumerate}[$(i)$]
\item If $0 \leq s\leq\min\{n+1,k\}$, then $R_{n+1,k+1,s}^{\mathrm{hom}}$ is the direct sum of three representations. One is $\operatorname{Ext}_{S_{n}}^{S_{n+1}}R_{n,k+1,s}^{\mathrm{hom}}$, showing up only when $s \leq n$, the second one is $\operatorname{LE}_{S_{n}}^{S_{n+1}}R_{n,k,s}^{\mathrm{hom}}$ appearing for $k\geq1$ and $s \leq n$, and finally $e_{n+1}R_{n+1,k,s}^{\mathrm{hom}}$, which has a contribution in case $k\geq1$.
  \item For $s=k+1 \leq n+1$ we get that $R_{n+1,k+1,k+1}^{\mathrm{hom}}$ equals the direct sum of $\operatorname{Ext}_{S_{n}}^{S_{n+1}}R_{n,k+1,k}^{\mathrm{hom}}$ and $\operatorname{LE}_{S_{n}}^{S_{n+1}}R_{n,k,k}^{\mathrm{hom}}$, with the latter showing up only when $k\geq1$, or equivalently of $\operatorname{Ext}_{S_{n}}^{S_{n+1}}R_{n,k+1,k+1}^{\mathrm{hom}}$ (when $k<n$) and $\operatorname{Ind}_{n,S_{n}}^{S_{n+1}}R_{n,k,k}^{\mathrm{hom}}$ (if $k\geq1$).
\end{enumerate} \label{homdecom}
\end{prop}

\begin{proof}
Remark \ref{Rnkshom}, Definition \ref{OPnksI}, and part $(iii)$ of Theorem \ref{Rnksdecom} decompose $R_{n+1,k+1,s}^{\mathrm{hom}}$ as the direct sum of $R_{n+1,\tilde{I}}^{\mathrm{hom}}$ over multi-sets $\tilde{I}$ of size $k$, containing elements between 0 and $n+1$, and such that the union of 0 with its $s$ smallest elements form a strictly increasing sequence. We decompose the collection of these multi-sets into those that contain $n+1$, and those that do not. Recall that the condition involving $s$ means that adding $i_{0}=0$ to $\tilde{I}$ yields an initial increasing sequence in the first $s$ steps, which for $s \geq k$ means that $\tilde{I}$ is a subset of $\mathbb{N}_{n+1}$, and when $s=k+1$ the condition involving $i_{k+1}=n+1$ implies that $I$ must be contained in $\mathbb{N}_{n}$.

Hence we may have $n+1\in\tilde{I}$ only if $s \leq k$, with the restriction $k\geq1$ showing up since if $k=0$ then $\tilde{I}$ is empty. We write such a multi-set as the union of $\{n+1\}$ and a multi-set $I$ of size $k-1$, and $I$ satisfies the condition for $s$. Indeed, when $s<k$ this is the same condition for $I$ and for $\tilde{I}$, and when $s=k$ it means that adding $i_{k}=n+1$ to $I$ still yields an increasing sequence, which is again the condition on $\tilde{I}=I\cup\{n+1\}$. The fact that part $(i)$ of Proposition \ref{incI} gives $R_{n+1,\tilde{I}}^{\mathrm{hom}}=e_{n+1}R_{n+1,I}^{\mathrm{hom}}$, and $I$ runs over all the multi-sets appearing in $R_{n+1,k,s}^{\mathrm{hom}}$, shows that these sets give the third asserted term in part $(i)$.

We now turn to multi-sets $\tilde{I}$ not containing $n+1$, and note that if $s \leq k$ then the condition for $R_{n,I}^{\mathrm{hom}}$ to appear in $R_{n,k+1,s}^{\mathrm{hom}}$ is the same as the one determining whether $R_{n+1,I}^{\mathrm{hom}}$ shows up in $R_{n+1,k+1,s}^{\mathrm{hom}}$. However, in the first paragraph we saw that for $s=k+1$ we get $\tilde{I}\subseteq\mathbb{N}_{n}$, while for $n$ the same condition means $\tilde{I}\subseteq\mathbb{N}_{n-1}$, and the condition for general $\tilde{I}\subseteq\mathbb{N}_{n}$ is associated with $s=k$. We also note that if $s=n+1$ then $\tilde{I}$ must contain $n+1$ in contradiction to our assumption, so that we henceforth assume $s \leq n$.

Now, Proposition \ref{mapsmulti} shows that for every such $\tilde{I}$ we get a contribution of $\operatorname{Ext}_{S_{n}}^{S_{n+1}}R_{n,I}^{\mathrm{hom}}$, yielding the first term in both parts $(i)$ and $(ii)$ here because of the condition on $\tilde{I}$ for $n$ from the previous paragraph (and we saw that $s \leq n$). Moreover, if $n\in\tilde{I}$ (so that $k\geq1$ again), so that $\tilde{I}=I\cup\{n\}$ for a multi-set $I$ of size $k-1$, then the same considerations from the case with $n+1\in\tilde{I}$ show that $I$ satisfies the condition for $s$ and $n$ in case $s \leq k$ because $\tilde{I}$ does (also when $s=k$, with $i_{k}=n$ completing $I$ to $\tilde{I}$), and when $s=k+1$ it satisfies the condition for $s=k$. Gathering the $\operatorname{LE}_{S_{n}}^{S_{n+1}}$-images from part $(ii)$ of Proposition \ref{mapsmulti} yields the remaining term in both parts $(i)$ and $(ii)$, with the second expression in part $(ii)$ being a consequence of Definition \ref{plusv}, as in the proof of part $(ii)$ of Proposition \ref{mapsmulti}. This completes the proof of the proposition.
\end{proof}

\begin{ex}
When $s=n+1<k+1$, the only possible value of the multi-set $\tilde{I}$ is $\mathbb{N}_{n}$ plus $k-n>0$ copies of $n+1$ (see Example \ref{n3k4} for the case $n=2$ and $k=3$ of this situation), showing that $R_{n+1,k+1,n+1}^{\mathrm{hom}}$ is a homogeneous regular representation multiplied by some positive power of $e_{n+1}$. As $n<k$, the same description applies for $R_{n+1,k,n+1}^{\mathrm{hom}}$ (with one power of $e_{n+1}$ less), yielding the form of Proposition \ref{homdecom} in this case. The case $s=k+1 \leq n+1$ from part $(ii)$ there is the homogeneous version of Corollary 4.19 of \cite{[Z1]} (since $R_{n+1,k+1,k+1}^{\mathrm{hom}}$ and $R_{n,k+1,k+1}^{\mathrm{hom}}$, and $R_{n,k,k}^{\mathrm{hom}}$ from Remark \ref{Rnkshom} are $R_{n+1,k+1}^{\mathrm{hom}}$, $R_{n,k+1}^{\mathrm{hom}}$, and $R_{n,k}^{\mathrm{hom}}$ from Remark 3.26 of \cite{[Z1]} respectively). When $k=0$, part $(i)$ of this proposition reduces to the assertion that applying $\operatorname{Ext}_{S_{n}}^{S_{n+1}}$ to the trivial representation $R_{n,1,0}^{\mathrm{hom}}$ of $S_{n}$ on constants yields the trivial representation $R_{n+1,1,0}^{\mathrm{hom}}$ of $S_{n+1}$ on constants, while the two expressions in part $(ii)$ there are for the same trivial representation $R_{n+1,1,1}^{\mathrm{hom}}$ of $S_{n+1}$, either as $\operatorname{Ext}_{S_{n}}^{S_{n+1}}R_{n,1,0}^{\mathrm{hom}}$ or as $\operatorname{Ext}_{S_{n}}^{S_{n+1}}R_{n,1,1}^{\mathrm{hom}}$. \label{decomext}
\end{ex}

After the case $k=0$ from Remark \ref{decomext} we present a case with $k=1$.
\begin{ex}
We recall from Example \ref{n4k2} that \[R_{4,2,2}^{\mathrm{hom}}=V_{\substack{000 \\ 1\hphantom{11}}} \oplus V_{\substack{001 \\ 1\hphantom{12}}}  \oplus V_{\substack{00 \\ 11}} \oplus V_{\substack{011 \\ 1\hphantom{22}}} \oplus V_{0000}e_{1} \oplus V_{0000}e_{2} \oplus V_{0000}e_{3},\] and that we add $V_{0000}e_{4}$ in order to get to $R_{4,2,1}^{\mathrm{hom}}$ and the additional term $V_{0000}$ to yield $R_{4,2,0}^{\mathrm{hom}}$. Example \ref{k2gen} generalizes this statement to any $n$, so that in particular $R_{5,2,2}^{\mathrm{hom}}$ equals \[V_{\substack{0000 \\ 1\hphantom{111}}} \oplus V_{\substack{0001 \\ 1\hphantom{112}}} \oplus V_{\substack{000 \\ 11\hphantom{1}}}  \oplus V_{\substack{001 \\ 11\hphantom{1}}} \oplus V_{\substack{0011 \\ 1\hphantom{122}}} \oplus V_{\substack{0111 \\ 1\hphantom{222}}} \oplus V_{00000}e_{1} \oplus V_{00000}e_{2} \oplus V_{00000}e_{3} \oplus V_{00000}e_{4},\] and the extra summands $V_{00000}e_{5}$ for $R_{5,2,1}^{\mathrm{hom}}$ and $V_{00000}$ to get to $R_{5,2,0}^{\mathrm{hom}}$. The case $n=4$ and $k=1$ in part $(ii)$ of Proposition \ref{homdecom} indeed compares $R_{5,2,2}^{\mathrm{hom}}$ with the $\operatorname{Ext}_{S_{4}}^{S_{5}}$-image of $R_{4,2,1}^{\mathrm{hom}}$ (with this operator mostly taking $V_{C}$ to $V_{\hat{\iota}C}$, except for $V_{\substack{011 \\ 1\hphantom{22}}}$, which also yields $V_{\substack{001 \\ 11\hphantom{1}}}$) plus $\operatorname{LE}_{S_{4}}^{S_{5}}V_{0000}=V_{\substack{0111 \\ 1\hphantom{222}}}$ (in the alternative expression, $V_{00000}e_{4}$ comes from $\operatorname{Ind}_{4,S_{4}}^{S_{5}}V_{0000}$, rather than the $\operatorname{Ext}_{S_{4}}^{S_{5}}$-image of the representation $V_{0000}e_{4}$ separating $R_{4,2,1}^{\mathrm{hom}}$ from $R_{4,2,2}^{\mathrm{hom}}$). From part $(i)$ there we see that $R_{5,2,1}^{\mathrm{hom}}$ is the same representation plus $e_{5}R_{5,1,1}^{\mathrm{hom}}=V_{00000}e_{5}$, and the term to add in order to get $R_{5,2,1}^{\mathrm{hom}}$ is the $\operatorname{Ext}_{S_{4}}^{S_{5}}$-image $V_{00000}$ of the summand $V_{0000}$ separating $R_{4,2,0}^{\mathrm{hom}}$ from $R_{4,2,1}^{\mathrm{hom}}$. \label{k1inc}
\end{ex}
Example \ref{k2gen} implies that the case $k=1$ resembles Example \ref{k1inc} for any $n$.

\begin{ex}
In the notation of Example \ref{n2d4}, we can write the representation $R_{2,4,2}^{\mathrm{hom}}$ as $U_{+}e_{1}e_{2}^{2} \oplus U_{-}e_{2}^{2}$ from the multi-set $\{1,2,2\}$, with $R_{2,4,1}^{\mathrm{hom}}$ obtained by adding $U_{+}e_{1}^{2}e_{2} \oplus U_{-}e_{1}e_{2}$, $U_{+}e_{1}^{3} \oplus U_{-}e_{1}^{2}$, and $U_{+}e_{2}^{3}$ associated with the multi-sets $\{1,1,2\}$, $\{1,1,1\}$, and $\{2,2,2\}$ respectively. Similarly, for $k=3$ we get $U_{+}e_{1}e_{2} \oplus U_{-}e_{2}$ (using the set $\{1,2\}$) as $R_{2,3,2}^{\mathrm{hom}}$, and the expressions to add in order to get $R_{2,3,1}^{\mathrm{hom}}$, based on the multi-sets $\{1,1\}$, and $\{2,2\}$ respectively, are $U_{+}e_{1}^{2} \oplus U_{-}e_{1}$ and $U_{+}e_{2}^{2}$. Applying $\operatorname{Ext}_{S_{2}}^{S_{3}}$ to the representations with $k=4$ and $\operatorname{LE}_{S_{2}}^{S_{3}}$ to those having $k=3$, using the formulae from Example \ref{n3d4}, we get \[U_{0}e_{1}e_{2}^{2} \oplus U_{1}e_{2}^{2} \oplus U_{0}e_{1}^{2}e_{2} \oplus U_{1}e_{1}e_{2} \oplus U_{0}e_{1}^{3} \oplus U_{1}e_{1}^{2} \oplus U_{2}e_{1}e_{2} \oplus U_{3}e_{2} \oplus U_{2}e_{1}^{2} \oplus U_{3}e_{1},\] plus the two representations $U_{0}e_{2}^{3}$ and $U_{2}e_{2}^{2}$, in the notation from Example \ref{n3k4}. These combine to the parts of the representations $R_{3,4,2}^{\mathrm{hom}}$ and $R_{3,4,1}^{\mathrm{hom}}$ arising from the representations in that example that are based on multi-sets not containing $n+1=3$, as the two first terms in part $(i)$ of Proposition \ref{homdecom} predict. \label{n23k3}
\end{ex}
The terms completing those from Example \ref{n23k3} into the full representations $R_{3,4,2}^{\mathrm{hom}}$ and $R_{3,4,1}^{\mathrm{hom}}$ are those that are based on the multi-sets from Example \ref{n3k4} that do contain 3, a condition that is also satisfied by the unique multi-set $I=\{1,2,3\}$ for which $R_{3,4,3}^{\mathrm{hom}}=R_{3,I}^{\mathrm{hom}}$.

\medskip

The proof of Proposition \ref{homdecom} and the direct sums from Remark \ref{Rnkshom} imply that pairs of cocharge tableaux with shapes of size $n+1$ and multi-sets containing their $\operatorname{Dsp}^{c}$-sets are obtained in a unique manner via the operations from Definition \ref{plusv}. In order to consider similar properties of the representations showing up in Theorem \ref{FMTdecom}, we first set some notation.
\begin{defn}
We consider the following objects:
\begin{enumerate}[$(i)$]
\item An \emph{internal corner} of a partition $\nu \vdash n+1$, or of its Ferrers diagram, is a box $v$ inside that diagram such that removing it also produces a Ferrers diagram of a partition (now of $n$).
\item The set of internal corners of $\nu$ will be denoted by $\operatorname{IC}(\nu)$.
\item For every $v\in\operatorname{IC}(\nu)$, the partition whose Ferrers diagram is the one obtained by removing $v$ from that of $\nu$ will be denoted by $\nu-v$.
\end{enumerate} \label{ICdef}
\end{defn}

\begin{rmk}
It is clear from Definitions \ref{plusv} and \ref{ICdef} that if $v\in\operatorname{IC}(\nu)$ and $\lambda=\nu-v$ then $v\in\operatorname{EC}(\lambda)$ and $\lambda+v=\nu$, and conversely when $\lambda \vdash n$ and $v\in\operatorname{EC}(\lambda)$, by setting $\nu:=\lambda+v$ we get $v\in\operatorname{IC}(\nu)$ and $\nu-v=\lambda$. \label{ECIC}
\end{rmk}

Here is a converse, in some sense, to Lemma \ref{SSYTct}.
\begin{prop}
Consider an element $N\in\operatorname{SSYT}(\nu)$ for a partition $\nu \vdash n+1$.
\begin{enumerate}[$(i)$]
\item If no entry of $N$ vanishes, then $N=K_{+}$ for a unique $K\in\operatorname{SSYT}(\nu)$.
\item When $N$ contains a vanishing entry, there is a unique pair of $v\in\operatorname{IC}(\nu)$ and $M\in\operatorname{SSYT}(\lambda)$, for $\lambda=\nu-v \vdash n$, such that $M\hat{+}v=N$.
\item The set $\operatorname{SSYT}(\nu)$ is a disjoint union of $\{K_{+}\;|\;K\in\operatorname{SSYT}(\nu)\}$ and of the union $\bigcup_{v\in\operatorname{IC}(\nu)}\{M\hat{+}v\;|\;\operatorname{SSYT}(\nu-v)\}$, the latter being disjoint as well.
\item If $N$ contains more than one vanishing entry, then $\delta_{M,v}=1$ for $M$ and $v$ from part $(ii)$.
\item If the maximal element of $\operatorname{Dsp}^{c}(N)$ is smaller than $n+1$, and in particular when $d:=\Sigma(N)<n+1$, the case from part $(ii)$ must hold. When this element is smaller than $n$, and in particular in case $d<n$, we must also have $\delta_{M,v}=1$.
\end{enumerate} \label{imaddv}
\end{prop}
By Remark \ref{ECIC}, when $\lambda \vdash n$ we indeed have $v\in\operatorname{EC}(\lambda)$ and $\lambda+v=\nu$, so that Definition \ref{plusv} indeed produces the tableau $M\hat{+}v$ and its shape is $\nu=\operatorname{sh}(N)$, and we can also consider $\delta_{M,v}$.

\begin{proof}
When all the entries of $N$ are positive, we define $K$ by simply subtracting 1 from every entry of $N$, and then we get an element of $\operatorname{SSYT}(\nu)$ because the shape and the semi-standard conditions are preserved and all the entries are non-negative by assumption. Since it is clear that $N=K_{+}$ in this case, and this is the only choice of shape and tableau $K$ that would produce this equality, part $(i)$ follows.

Next, part $(vii)$ of Lemma \ref{SSYTct} implies that if $N=M\hat{+}v$ then $\operatorname{Dsp}^{c}(N)$ does not contain $n+1$, while from part $(v)$ there we deduce that in case $N=K_{+}$ as in part $(i)$, this multi-set does contain this number. Hence the first union in part $(iii)$ is disjoint, and we recall from part $(i)$ of Lemma \ref{Dspc} that the maximal element $i_{k-1}$ in that multi-set is the number of positive entries of $N$. Since this is $n+1$ precisely when $N$ contains no vanishing entries, we deduce that these are indeed the elements of the first set in part $(iii)$.

We thus need to show that if $\operatorname{Dsp}^{c}(M)$ only contains elements between 1 and $n$, then $N=M\hat{+}v$ for a unique pair of $v\in\operatorname{IC}(\nu)$ and $M\in\operatorname{SSYT}(\nu-v)$. This will establish part $(ii)$ (as we saw that $\operatorname{Dsp}^{c}(N)$ does not contain $n+1$ if and only if $N$ has a vanishing entry), as well as part $(iii)$ since the uniqueness implies that the second union there is also disjoint.

For this we obtain conditions that are equivalent to the equality $N=M\hat{+}v$, where $N$ is as usual and $M\in\operatorname{SSYT}(\lambda)$ for some $\lambda \vdash n$. Lemma \ref{ctJmulti} expresses $M$ uniquely as $\operatorname{ct}_{J}(S)$ for $S\in\operatorname{SYT}(\lambda)$ and the appropriate multi-set $J$, and similarly there are a unique multi-set $\tilde{J}$ and a unique $T\in\operatorname{SYT}(\nu)$ such that $N=\operatorname{ct}_{\tilde{J}}(T)$. Then Definition \ref{plusv} shows that $N=M\hat{+}v$ if and only if $T=S\tilde{+}v$ and there is a relation between $J$ and $\tilde{J}$, and the former condition is equivalent to $\operatorname{ev}T=\operatorname{ev}S+v$.

Thus the only possible $v$ is $v_{\operatorname{ev}T}(n+1)$, and by setting $\lambda:=\nu-v \vdash n$ we obtain an element of $\operatorname{SYT}(\lambda)$ by removing $v$ and $n+1$ from $\operatorname{ev}T$, and we must take $S$ to be the $\operatorname{ev}$-image of the latter tableau. It is clear that $n$ lies in $\operatorname{Dsi}(\operatorname{ev}T)=\operatorname{Dsi}^{c}(T)$ if and only if the row $R_{\operatorname{ev}T}(n+1)$ containing $v$ is larger than $R_{\operatorname{ev}S}(n)$, namely to the condition that $v$ lies below $n$ in $\operatorname{ev}S$. This also determines the value of $\delta_{M,v}$ when $M=\operatorname{ct}_{J}(S)$, and we get, as in Lemma 4.8 of \cite{[Z1]}, that $\operatorname{Dsi}^{c}(T)$ equals $\operatorname{Dsi}^{c}(S)$ in case $\delta_{M,v}=1$, and it is $\operatorname{Dsi}^{c}(S)\cup\{n\}$ when $\delta_{M,v}=0$.

We now recall that the tableau $S\in\operatorname{SYT}(\lambda)$ such that $M=\operatorname{ct}_{J}(S)$ and the set $\operatorname{Dsp}^{c}(M)$ determine $M$ uniquely (as indeed knowing the latter is equivalent to knowing $J$ via Definition \ref{sets} or part $(i)$ of Lemma \ref{Dspc}). Moreover, part $(v)$ of Lemma \ref{SSYTct} implies that given our $S$, $T$, $\lambda$, $\nu$, and $v$, and with them $\delta_{M,v}$, if the latter parameter is 1 then $\operatorname{Dsp}^{c}(N)=\operatorname{Dsp}^{c}(M)$, and otherwise $\operatorname{Dsp}^{c}(N)$ is the union of $\operatorname{Dsp}^{c}(M)$ and $\{n\}$.

Thus, if $\delta_{M,v}=1$, then the fact that $\operatorname{Dsp}^{c}(N)$ consists only of elements between 1 and $n$ and contains the set $\operatorname{Dsi}^{c}(T)=\operatorname{Dsi}^{c}(S)$ implies that there is $M=\operatorname{ct}_{J}(S)\in\operatorname{SSYT}(\lambda)$ with $\operatorname{Dsp}^{c}(M)=\operatorname{Dsp}^{c}(N)$. When $\delta_{M,v}=0$ we know that $\operatorname{Dsp}^{c}(N)$ contains $\operatorname{Dsi}^{c}(T)=\operatorname{Dsi}^{c}(S)\cup\{n\}$ and again consists of elements that are bounded by $n$. Hence by subtracting one copy of $n$ we get a multi-set of integers, with the same bound, that contains $\operatorname{Dsi}^{c}(S)$, so that there is $M=\operatorname{ct}_{J}(S)\in\operatorname{SSYT}(\lambda)$ with this multi-set as $\operatorname{Dsp}^{c}(M)$, and thus $\operatorname{Dsp}^{c}(N)=\operatorname{Dsp}^{c}(M)\cup\{n\}$ (multi-set addition).

It follows that for $N$ with $n+1\not\in\operatorname{Dsp}^{c}(N)$ we found that the option to get $v\in\operatorname{IC}(\nu)$ and $M\in\operatorname{SSYT}(\nu-v)$ for which $N=M\hat{+}v$ exists and is unique, establishing parts $(ii)$ and $(iii)$. Recall again that when $\delta_{M,v}=0$ we have $n\in\operatorname{Dsp}^{c}(N)$, so that there are at least $n$ positive entries in $N$. As this contradicts, for $N\in\operatorname{SSYT}(\nu)$ with $\nu \vdash n+1$, the assumption from part $(vi)$ that $N$ contains more than one vanishing entry, this part follows as well.

Finally, recall from parts $(vi)$ and $(v)$ of Lemma \ref{SSYTct} that if $N=K_{+}$ for $K\in\operatorname{SSYT}(\nu)$ then $n+1\in\operatorname{Dsp}^{c}(N)$, while when $N=M\hat{+}v$ with $\delta_{M,v}=0$ this multi-set contains $n$, both contradicting the respective assumptions in part $(v)$. For the assertions involving the sum, either note, via part $(ii)$ of Lemma \ref{Dspc}, that the sum $\Sigma(N)$ is at least the maximal element in question, or, via part $(vii)$ of Lemma \ref{SSYTct}, that $\Sigma(N)$ equals $\Sigma(K)+n+1$ in case $N=K_{+}$ for $K\in\operatorname{SSYT}(\nu)$ and it is $\Sigma(M)+n$ when $N=M\hat{+}v$ for $M$ and $v$ as above with and $\delta_{M,v}=0$ with $\Sigma(K)$ and $\Sigma(M)$ being non-negative. This completes the proof of the proposition.
\end{proof}

We draw from Proposition \ref{imaddv} the following consequence, which is closely related to Lemma 4.20 and Corollary 4.22 of \cite{[Z1]}.
\begin{cor}
Take $N\in\operatorname{SSYT}_{d}(\nu)$ for $\nu \vdash n+1$, and let $i_{\max}$ be the maximal element of $\operatorname{Dsp}^{c}(N)$. If $n>2i_{\max}$ then $N=\hat{\iota}M$ for appropriate $\lambda \vdash n$ and $M\in\operatorname{SSYT}_{d}(\lambda)$. This is the case for all $N\in\operatorname{SSYT}_{d}(\nu)$ in case $n>2d$. \label{n2diota}
\end{cor}

\begin{proof}
Part $(v)$ of Proposition \ref{imaddv} shows that $N$ must be of the form $M\hat{+}v$ for some $M\in\operatorname{SSYT}(\lambda)$ and $v\in\operatorname{EC}(\lambda)$ with $\delta_{M,v}=1$ (so that $\lambda+v=\nu$), and then $\Sigma(M)=\Sigma(N)=d$. As part $(ii)$ of Lemma \ref{Dspc} implies that $d \geq i_{\max}$, the second assertion follows from the first.

Now, part $(i)$ of that lemma implies that $i_{\max}$ is the number of positive entries of $M$. But $M\in\operatorname{SSYT}(\lambda)$, so that all the entries except the first row must be positive. In particular we get $\lambda_{2} \leq i_{\max}$. But the number of zeros in $M$ is then at least $n-i_{\max}>i_{\max}\geq\lambda_{2}$, and all of them are in the first row. Hence in the tableau $S:=\operatorname{ct}_{J}^{-1}(M)\in\operatorname{SYT}(\lambda)$, the first $\lambda_{2}+1$ entries in the first row are the integers between 1 and $\lambda_{2}+1$ (in the appropriate order).

But this implies that the first path in evaluating $\operatorname{ev}S$ runs over the first row until it has no choice, so that it ends at the end of the first row and we get $R_{\operatorname{ev}S}(n)=1$. Therefore the only $v\in\operatorname{EC}(\lambda)$ for which $\delta_{M,v}=1$ (namely it does not lie below $n$ in $\operatorname{ev}S$) is the one in the first row, for which $\lambda+v=\lambda_{+}$ and $N=M\hat{+}v=\hat{\iota}M$, as desired. This completes the proof of the corollary.
\end{proof}
Corollary \ref{n2diota} implies the following immediate consequence.
\begin{cor}
If $I$ is a multi-set with maximal element $i_{\max}$, and $n>2i_{\max}$, then $R_{n+1,I}^{\mathrm{hom}}$ is obtained by decomposing $R_{n,I}^{\mathrm{hom}}$ as in part $(i)$ of Theorem \ref{Rnksdecom} and replacing each $C$ there with $\hat\iota{C}$. \label{RnIhomstab}
\end{cor}
Corollary \ref{RnIhomstab} follows either from applying Corollary \ref{n2diota} to the decomposition from Theorem \ref{Rnksdecom}, or by noticing that the proof of Lemma 4.20 and Corollary 4.22 of \cite{[Z1]} extends, in the homogeneous setting, to multi-set indices. In the setting of Corollary \ref{n2diota}, the associated generalized higher Specht polynomials can be described using Proposition \ref{forstab} below.

As with Remark 4.21 of \cite{[Z1]}, a weak inequality $n\geq2i_{\max}$ (or $n\geq2d$) is sufficient for obtaining the consequence of Corollaries \ref{n2diota} and \ref{RnIhomstab}, but we will not need this stronger result. The proof of that corollary shows, in relation with Remark \ref{notation}, that when $\nu \vdash n+1$, every element of $\operatorname{SYT}(\nu)$ whose $\operatorname{Dsi}^{c}$-set has a maximal element $i_{\max}$ for which $n>2i_{\max}$ (and even $n\geq2i_{\max}$), and especially one for which the sum $d$ of the elements of that set satisfies $n>2d$ (and even $n\geq2d$), must be a $\tilde{\iota}$-image, while an element of $\operatorname{SYT}(\nu)$ for which the maximal element $i_{\max}$ of $\operatorname{Dsi}$-set satisfies this inequality (and in particular when the inequality involving the sum of this set holds) is a $\iota$-image.
\begin{ex}
If $N$ is the tableau $M_{+}$ from Example \ref{hatplus}, then it contains no zeros, and subtracting 1 from each entry yields $K=M$ for the tableau for which $N=K_{+}$. The two tableaux $M\hat{+}v$ with $\delta_{M,v}=1$ there contain two zeros, and are indeed of the form asserted in Proposition \ref{imaddv}. The other two tableaux $M\hat{+}v$, for which $\delta_{M,v}=1$, indeed contain a single 0. Note, however, that $\hat{\iota}(M_{+})=\begin{ytableau} 0 & 1 & 2 & 3 & 8 \\ 2 & 5 & 5 \\ 7 \end{ytableau}$ (parentheses required, since this is not the same tableau as $(\hat{\iota}M)_{+}$) also contains a single zero, but it is obtained from a pair with $\delta$-parameter 1 (as a $\hat{\iota}$-image). \label{imhatplus}
\end{ex}
Following the proof of Proposition \ref{imaddv}, when we consider all the tableaux $M\hat{+}v$ from Examples \ref{hatplus} and \ref{imhatplus}, we first apply $\operatorname{ct}_{\tilde{J}}^{-1}$ to get the respective tableau $S\tilde{+}v$ (here $\tilde{J}$ is $J_{+}$ or $J_{+}\cup\{1\}$, the latter valid also for $\hat{\iota}(M_{+})$), and then $\operatorname{ev}$ on it, and see that $n+1=9$ shows up in $v$ and removing it brings us back to $T$, and thus to $S$ via $\operatorname{ev}$ and to $M$ through $\operatorname{ct}_{J}$.

\begin{ex}
There are five semi-standard tableaux of shape of size 5 and entry sum 2, and using them we get, as in Example \ref{nd42VM}, the decomposition \[\mathbb{Q}[\mathbf{x}_{5}]_{2}=V_{00002} \oplus V_{\substack{0000 \\ 2\hphantom{222}}} \oplus V_{00011} \oplus V_{\substack{0001 \\ 1\hphantom{112}}} \oplus V_{\substack{000 \\ 11\hphantom{1}}}.\] Each of these tableaux is a $\hat{\iota}$-image of one of the tableaux showing up in the decomposition of $\mathbb{Q}[\mathbf{x}_{4}]_{2}$ from the latter example, as predicted by Corollary \ref{n2diota} (with the weak inequality $n\geq2d$). Increasing $n$ from the other case in that example by 1, we get
\[\mathbb{Q}[\mathbf{x}_{3}]_{4}=V_{004} \oplus V_{\substack{00 \\ 4\hphantom{4}}} \oplus V_{013} \oplus V_{\substack{01 \\ 3\hphantom{3}}} \oplus V_{022} \oplus V_{\substack{03 \\ 1\hphantom{4}}} \oplus V_{\substack{0 \\ 1 \\ 3}} \oplus V_{\substack{02 \\ 1\hphantom{3}}} \oplus V_{112} \oplus V_{\substack{11 \\ 2\hphantom{2}}},\] with the first five summands being $\hat{\iota}$-images, the next three terms are obtained by $\mathbb{Q}[\mathbf{x}_{2}]_{2}=V_{02} \oplus V_{\substack{0 \\ 2}} \oplus V_{11}$ by the construction $M\hat{+}v$ with $v$ in the new row (and hence $\delta_{M,v}=0$), and the last two representations have indices $K_{+}$ for $K$ showing up in $\mathbb{Q}[\mathbf{x}_{3}]_{1}=V_{001} \oplus V_{\substack{00 \\ 1\hphantom{1}}}$. \label{decomex}
\end{ex}
We can compare the expressions from Example \ref{decomex} with those from Examples \ref{n5d2} and \ref{n3d4}, and the construction of the former from the ones in Example \ref{nd42VM} with the way the latter were obtained from Examples \ref{n4d2} and \ref{n2d4}. This comparison exemplifies the similarities and the differences of the operations from Propositions \ref{mapsmulti} and \ref{homdecom} with those appearing in Theorem \ref{opersVM} below.

\medskip

Up until now we examined the how the representations $V_{C}$ and $V_{C}^{\vec{h}}$ (for some $\vec{h}$) from Definition \ref{defSpecht} behave under the operations from Definition \ref{plusv} when $C$ is restricted to be a cocharge tableau (as in \cite{[Z1]}, but with $\vec{h}$ arising from more general multi-sets). We now turn to $V_{M}$ for a general semi-standard Young tableau $M$ (with no symmetric multipliers hence no vector $\vec{h}$), and we begin by generalizing Proposition 2.15 of \cite{[Z1]}.
\begin{prop}
Take $\lambda \vdash n$, $M\in\operatorname{SSYT}(\lambda)$, and $T$ of shape $\lambda$ and content $\mathbb{N}_{n}$. Then the generalized higher Specht polynomial $F_{\hat{\iota}M,\iota T}$ also equals $\varepsilon_{\iota T}p_{M,T}/s_{\hat{\iota}M,\iota T}$, and substituting $x_{n+1}=0$ in it yields $F_{M,T}$. \label{forstab}
\end{prop}

\begin{proof}
We follow the proof of Proposition 2.15 of \cite{[Z1]}, and we denote by $v$ the box separating $\lambda_{+}$ from $\lambda$. The expression for $\hat{\iota}M$ in part $(i)$ of Lemma \ref{SSYTct} implies that an appropriate element of $R(\iota T)$, acting only on the first row, takes this tableau to the one obtained from $M$ by adding $v$ and putting 0 inside it. Recalling the form of $\iota T$, the image of $p_{\hat{\iota}M,\iota T}$ under that element of $R(T)$ is $p_{M,T}$, so that the first assertion follows from Proposition \ref{Spechtpols}, or Lemma \ref{operRT}.

Next, we note that a tableau in the orbit of $\hat{\iota}M$ under the action of $R(\iota T)$ either contains 0 in $v$, or it does not. In the former case, this tableau is obtained by adding $v$ and 0 to some tableau that is in the $R(T)$-orbit of $M$, while in the latter case the corresponding monomial is divisible by $x_{n+1}$. This means that the expression $\sum_{\tau \in R(\iota T)}\tau p_{\hat{\iota}M,\iota T}/s_{\hat{\iota}M,\iota T}$ from Lemma \ref{operRT}, which was seen to be the same as $\sum_{\tau \in R(\iota T)}\tau p_{M,T}/s_{\hat{\iota}M,\iota T}$, equals $\sum_{\tau \in R(T)}\tau p_{M,T}/s_{M,T}$ plus a multiple of $x_{n+1}$.

But we now recall that in $\iota T$, the number $n+1$ lies in $v$, which is alone in its column. This means that $C(\iota T)$ is the image of $C(T)$ under the embedding of $S_{n}$ into $S_{n+1}$ as the stabilizer of $n+1$. Hence letting this group act on the expression obtained from Lemma \ref{operRT}, the first part yields $F_{M,T}$ by Definition \ref{Spechtdef}, while the second is taken to a multiple of $x_{n+1}$, which easily implies the second assertion. This proves the proposition.
\end{proof}

\begin{ex}
We saw in Example \ref{exSpecht} the generalized higher Specht polynomial $F_{M,T}$, where $T$ is the tableau denoted by $S$ in Example \ref{ctJex}, and $M$ is the last tableau there. Since $\iota T=\begin{ytableau} 1 & 3 & 4 & 5 \\ 2 \end{ytableau}$ and $\hat{\iota}M=\begin{ytableau} 0 & 2 & 4 & 5 \\ 4 \end{ytableau}$ via part $(i)$ of Lemma \ref{SSYTct}, it is clear that $p_{\hat{\iota}M,\iota T}=x_{2}^{4}x_{3}^{2}x_{4}^{4}x_{5}^{5}$ is in the $R(\iota T)$-orbit of $p_{M,T}=x_{1}^{2}x_{2}^{4}x_{3}^{4}x_{4}^{5}$ as Proposition \ref{forstab} predicts, and evaluating $F_{\hat{\iota}M,\iota T}$ yields
\[(x_{2}^{4}-x_{1}^{4})(x_{3}^{2}x_{4}^{4}x_{5}^{5}+x_{3}^{2}x_{4}^{5}x_{5}^{4}+x_{3}^{4}x_{2}^{4}x_{5}^{5}+x_{3}^{4}x_{4}^{5}x_{5}^{2}+x_{3}^{5}x_{4}^{2}x_{5}^{4}+x_{3}^{5}x_{4}^{4}x_{5}^{2})+\]
\[(x_{2}^{4}x_{1}^{2}-x_{2}^{2}x_{1}^{4})(x_{3}^{4}x_{4}^{5}+x_{3}^{5}x_{4}^{4}+x_{3}^{4}x_{5}^{5}+x_{3}^{5}x_{5}^{4}+x_{4}^{4}x_{5}^{5}+x_{4}^{5}x_{5}^{4})+\] \[-(x_{2}^{5}x_{1}^{4}-x_{2}^{4}x_{1}^{5})(x_{3}^{2}x_{4}^{4}+x_{3}^{4}x_{4}^{2}+x_{3}^{2}x_{5}^{4}+x_{3}^{4}x_{5}^{2}+x_{4}^{2}x_{5}^{4}+x_{4}^{4}x_{5}^{2}).\] Substituting $x_{5}=0$ annihilates the first row and the last four terms in each of the other rows, thus indeed yielding $F_{M,T}$ from Example \ref{exSpecht}. \label{iotaex}
\end{ex}

\begin{rmk}
Proposition \ref{forstab} and the invariance under the group $\tilde{C}(T)$ allows us to view each generalized higher Specht polynomial as obtained by substituting all but finitely many variables to 0 inside a power series that is symmetric under permuting almost all the indices, in a way that generalizes Theorem 2.17 of \cite{[Z1]}. These power series, which are naturally called \emph{stable generalized higher Specht polynomials}, and produce \emph{stable generalized higher Specht quotients} as in Remark 2.18 of that reference, will be further studied, together with the representations that they span, in the sequel \cite{[Z2]}. \label{stabSpecht}
\end{rmk}
The stable generalized higher Specht polynomial arising via Remark \ref{stabSpecht} from the tableaux from Example \ref{iotaex} is obtained by replacing the longer multipliers in the parentheses in the second and third lines there by the monomial symmetric functions $\sum_{m=3}^{\infty}\sum_{l=3,\ l \neq m}^{\infty}x_{m}^{5}x_{4}$ and $\sum_{m=3}^{\infty}\sum_{l=3,\ l \neq m}^{\infty}x_{m}^{4}x_{2}$ in the variables $\{x_{i}\}_{i=3}^{\infty}$ respectively, with the one in the first line becoming the symmetric function with index 542 in these variables.

\medskip

The result for the representations from Theorem \ref{FMTdecom} is as follows.
\begin{thm}
Take $d\geq0$, $n\geq1$, and a content $\eta$ of length $n+1$.
\begin{enumerate}[$(i)$]
\item The space $\mathbb{Q}[\mathbf{x}_{n+1}]_{d}$ is the direct sum of the three spaces $\operatorname{Ext}_{S_{n}}^{S_{n+1}}\mathbb{Q}[\mathbf{x}_{n}]_{d}$, $\operatorname{LE}_{S_{n}}^{S_{n+1}}\mathbb{Q}[\mathbf{x}_{n}]_{d-n}$, and $e_{n+1}\mathbb{Q}[\mathbf{x}_{n+1}]_{d-n-1}$.
\item If $0\in\eta$, with complement $\eta\setminus\{0\}$, then the intersection of $\mathbb{Q}[\mathbf{x}_{n+1}]_{\eta}$ with first summand from part $(i)$ is $\operatorname{Ext}_{S_{n}}^{S_{n+1}}\mathbb{Q}[\mathbf{x}_{n}]_{\eta\setminus\{0\}}$. This produces all of $\mathbb{Q}[\mathbf{x}_{n+1}]_{\eta}$ in case $\eta$ contains 0 more than once.
\item When $0\in\eta$ and the entries of $\eta\setminus\{0\}$ are all positive, then set $\mu$ to be the content of length $n$ obtained by subtracting 1 from each element of $\eta\setminus\{0\}$, and then the second summand from part $(i)$ intersects $\mathbb{Q}[\mathbf{x}_{n+1}]_{\eta}$ in $\operatorname{LE}_{S_{n}}^{S_{n+1}}\mathbb{Q}[\mathbf{x}_{n}]_{\mu}$.
\item If all the elements in $\eta$ are positive then subtracting 1 from all of its entries yields a content $\mu$, now of length $n+1$, and intersecting the third summand from part $(i)$ with $\mathbb{Q}[\mathbf{x}_{n+1}]_{\eta}$ yields $e_{n+1}\mathbb{Q}[\mathbf{x}_{n+1}]_{\mu}$.
\end{enumerate} \label{opersVM}
\end{thm}

\begin{proof}
Theorem \ref{FMTdecom} expresses the four spaces $\mathbb{Q}[\mathbf{x}_{n+1}]_{d}$, $\mathbb{Q}[\mathbf{x}_{n}]_{d}$, $\mathbb{Q}[\mathbf{x}_{n}]_{d-n}$, and $\mathbb{Q}[\mathbf{x}_{n+1}]_{d-n-1}$ as direct sums of representations from Definition \ref{defSpecht}. Moreover, parts $(iv)$ and $(vii)$ of Lemma \ref{SSYTct} show that if $M\in\operatorname{SSYT}_{d}(\lambda)$ for some $\lambda \vdash n$ then $\operatorname{Ext}_{S_{n}}^{S_{n+1}}V_{M}$ from Definition \ref{plusv} is the direct sum of distinct representations participating in the former space, while if $M\in\operatorname{SSYT}_{d-n}(\lambda)$ then $\operatorname{LE}_{S_{n}}^{S_{n+1}}V_{M}$ is another direct sum of this sort. All these representations are of the form $V_{M\hat{+}v}$ for appropriate $v\in\operatorname{EC}(\lambda)$, while parts $(iii)$ and $(vii)$ of that lemma imply that multiplying $V_{K}$ for $\nu \vdash n+1$ and $K\in\operatorname{SSYT}_{d-n-1}(\nu)$ by $e_{n+1}$ produces the representation $V_{K_{+}}$, which also appears in $\mathbb{Q}[\mathbf{x}_{n+1}]_{d}$.

Recall from part $(iii)$ of Proposition \ref{imaddv} that given some $\nu \vdash n+1$ and $N\in\operatorname{SSYT}_{d}(\nu)$, it is either $K_{+}$ for a unique $K\in\operatorname{SSYT}_{d-n-1}(\nu)$ (with the same shape $\nu$), or there is a unique triple of $\lambda \vdash n$, $M\in\operatorname{SSYT}_{d}(\lambda)$, and $v\in\operatorname{EC}(\lambda)$ with $N=M\hat{+}v$ and $\delta_{M,v}=1$, or there exists a unique triple of $\lambda \vdash n$, $M\in\operatorname{SSYT}_{d-n}(\lambda)$, and $v\in\operatorname{EC}(\lambda)$ for which $N=M\hat{+}v$ and $\delta_{M,v}=0$, and exactly one of these cases occurs. Decomposing $\mathbb{Q}[\mathbf{x}_{n+1}]_{d}$ via Theorem \ref{FMTdecom} and comparing with the decompositions of the other three representations involved (using the same theorem), we see that the sum over tableaux $N$ of the first kind gives the third asserted expression (via part $(iii)$ of Lemma \ref{SSYTct}), those of the second kind merge to the second desired term, and the tableaux of the first kind produce the third term. Part $(i)$ is thus established.

Recall from Definition \ref{Qxnd} that each of the representations showing up in part $(i)$ decomposes into the summands associated with the appropriate contents, and that Theorem \ref{FMTdecom} gives the decomposition of each of these summands as well. Part $(ii)$ then follows from part $(iv)$ of Lemma \ref{SSYTct} and part $(ii)$ of Proposition \ref{imaddv}, with the second assertion there being a consequence of part $(iv)$ of that proposition. For part $(iii)$ we apply part $(ii)$ of Proposition \ref{imaddv} again with the other assertion in part $(iv)$ of Lemma \ref{SSYTct}, and part $(iv)$ is a consequence of part $(i)$ of that proposition, together with part $(iii)$ of Lemma \ref{SSYTct} (with $n+1$ instead of $n$). This proves the theorem.
\end{proof}

Here is a special case of Theorem \ref{opersVM} that will be useful for \cite{[Z2]}.
\begin{cor}
Take some $d\geq0$ and some $n\geq1$.
\begin{enumerate}[$(i)$]
\item If $n>d$ then $\mathbb{Q}[\mathbf{x}_{n+1}]_{d}$ is just $\operatorname{Ext}_{S_{n}}^{S_{n+1}}\mathbb{Q}[\mathbf{x}_{n}]_{d}$.
\item every content $\eta$ for which $\mathbb{Q}[\mathbf{x}_{n+1}]_{\eta}\subseteq\mathbb{Q}[\mathbf{x}_{n+1}]_{d}$ contains 0 when $n>d$, and this sub-representation of the one from part $(i)$ is $\operatorname{Ext}_{S_{n}}^{S_{n+1}}\mathbb{Q}[\mathbf{x}_{n}]_{\eta\setminus\{0\}}$.
\item When $n>2d$, every summand $\operatorname{Ext}_{S_{n}}^{S_{n+1}}V_{M}$ from part $(i)$ is just $V_{\hat{\iota}M}$.
\end{enumerate} \label{ExtVMlim}
\end{cor}

\begin{proof}
Part $(i)$ follows directly from part $(i)$ of Theorem \ref{opersVM} and the fact that the grading on $\mathbb{Q}[\mathbf{x}_{n+1}]$ is non-negative, or equivalently from part $(v)$ of Proposition \ref{imaddv}. We now note that if all the entries of $\eta$ are positive then its entry sum $d$ is at least its length $n+1$, which cannot occur if $n>d$. Part $(ii)$ is then a consequence of part $(ii)$ of that theorem (in fact, $\eta$ must contains 0 more than once in our argument, so the second assertion of that part also applies). Combining these results with Corollary \ref{n2diota} then yields part $(iii)$. This completes the proof of the corollary.
\end{proof}
As before, the inequality $n\geq2d$ suffices for part $(iii)$ of Corollary \ref{ExtVMlim}. Indeed, in the case $n=4$ and $d=2$ showing up in Example \ref{decomex}, the expression for $\mathbb{Q}[\mathbf{x}_{5}]_{2}$ was seen to be the $\operatorname{Ext}_{S_{4}}^{S_{5}}$-image of the formula for $\mathbb{Q}[\mathbf{x}_{4}]_{2}$ in Example \ref{nd42VM}, with each summand in the former being associated with the $\hat{\iota}$-image of a summand from the latter, as Corollary \ref{ExtVMlim} predicts. The expression corresponding to $n+1=3$ and $d=4$ in Example \ref{decomex} was seen to involve the five $\hat{\iota}$-images of those from Example \ref{nd42VM} (which give the full $\operatorname{Ext}_{S_{2}}^{S_{3}}$-image in this case), three elements which give $\operatorname{LE}_{S_{2}}^{S_{3}}\mathbb{Q}[\mathbf{x}_{2}]_{2}$, and the last two element $V_{112}$ and $V_{\substack{11 \\ 2\hphantom{2}}}$ are the $e_{3}$-multiples of the representations forming $\mathbb{Q}[\mathbf{x}_{3}]_{1}$. As another example, we have \[\mathbb{Q}[\mathbf{x}_{4}]_{3}=V_{0003} \oplus V_{\substack{000 \\ 3\hphantom{33}}} \oplus V_{0012} \oplus V_{\substack{001 \\ 2\hphantom{22}}} \oplus V_{\substack{002 \\ 1\hphantom{13}}}  \oplus V_{\substack{00 \\ 12}} \oplus V_{0111} \oplus V_{\substack{011 \\ 1\hphantom{22}}},\] separated into contents, and to exemplify part $(i)$ of Corollary \ref{ExtVMlim} we write
\[\mathbb{Q}[\mathbf{x}_{5}]_{3}=V_{00003} \oplus V_{\substack{0000 \\ 3\hphantom{333}}} \oplus V_{00012} \oplus V_{\substack{0001 \\ 2\hphantom{222}}} \oplus V_{\substack{0002 \\ 1\hphantom{113}}}  \oplus V_{\substack{000 \\ 12\hphantom{2}}} \oplus V_{00111} \oplus V_{\substack{0011 \\ 1\hphantom{122}}} \oplus V_{\substack{011 \\ 11\hphantom{2}}},\] in correspondence also with part $(ii)$ of that corollary, with the last summand producing two representations as its $\operatorname{Ext}_{S_{4}}^{S_{5}}$-image (as we saw in Example \ref{k1inc}).

\begin{rmk}
Consider the map taking $F_{M,T}$ to $F_{\hat{\iota}M,\iota T}$, as well as the one sending $F_{C,T}^{I,\mathrm{hom}}$ to $F_{\hat{\iota}C,\iota T}^{I,\mathrm{hom}}$. They embed $V_{M}$ into $V_{\hat{\iota}M}$ and $V_{C}^{\vec{h}(C,I)}$ into $V_{\hat{\iota}C}^{\vec{h}(\hat{\iota}C,I)}$, and therefore $R_{n,I}^{\mathrm{hom}}$ into $R_{n+1,I}^{\mathrm{hom}}$ and with it also $\mathbb{Q}[\mathbf{x}_{n}]_{d}$ into $\mathbb{Q}[\mathbf{x}_{n+1}]_{d}$ (in two different ways) as well as $\mathbb{Q}[\mathbf{x}_{n}]_{\mu}$ into $\mathbb{Q}[\mathbf{x}_{n+1}]_{\mu}$ for every content $\mu$, in a manner that respects the action of $S_{n}$. Since Corollaries \ref{n2diota}, \ref{RnIhomstab}, and \ref{ExtVMlim} show that for large enough $n$, the image of these embeddings generates the range over $\mathbb{Q}[S_{n+1}]$, we deduce that $\{R_{n,I}\}_{n \geq i_{\max}}$, $\{R_{n,I}^{\mathrm{hom}}\}_{n \geq i_{\max}}$, $\{\mathbb{Q}[\mathbf{x}_{n}]_{d}\}_{n=1}^{\infty}$, and $\{\mathbb{Q}[\mathbf{x}_{n}]_{\mu}\}_{n=1}^{\infty}$ are all stable families of representations as defined in \cite{[CF]}, \cite{[CEF]}, \cite{[Fa]}, and \cite{[SS]}, and they are also centrally stable as considered in \cite{[Pu]}. \label{repstab}
\end{rmk}
In fact, all the representations from Remark \ref{repstab} produce limits as $n\to\infty$ as representations of infinite symmetric groups on certain power series in infinitely many variables, as will be studied in detail in \cite{[Z2]}. In fact, the representations $\{R_{n,I}\}_{n \geq i_{\max}}$ are also stable, but as Remark \ref{onlysetsRnI} explains, they only produce such a limit when $I$ is a set. As with Remark 4.24 of \cite{[Z1]}, we may define $R_{n,I}$ and $R_{n,I}^{\mathrm{hom}}$ also for smaller values of $n$ and make the family $\{R_{n,I}^{\mathrm{hom}}\}_{n=1}^{\infty}$, as well as $\{R_{n,I}\}_{n=1}^{\infty}$, stable.

\noindent\textsc{Einstein Institute of Mathematics, the Hebrew University of Jerusalem, Edmund Safra Campus, Jerusalem 91904, Israel}

\noindent E-mail address: zemels@math.huji.ac.il


\begin{thebibliography}{}{}

\bibitem[ATY]{[ATY]} Ariki, S., Terasoma, T., Yamada, H., \textsc{Higher Specht polynomials}, Hiroshima Math. J., vol 27 no. 1 ,177--188 (1997).
\bibitem[AAB]{[AAB]} Ashraf, S., Azam, H., Berceanu, B., \textsc{Representation Stability of Power Sets and Square Free Polynomials}, Canad. J. Math., vol 67 issue 5, 1024–-1045 (2015).
\bibitem[AS]{[AS]} Assaf, S., Searles, D., \textsc{Schubert Polynomials, Slide Polynomials, Stanley Symmetric Functions, and Quasi-Yamanouchi Pipe Dreams}, Adv. Math., vol 306, 89-–122 (2017).
\bibitem[B]{[B]} Borel, A., \textsc{Sur la Cohomologie des Espaces Fibr\'{e}s Principaux et des Espaces homog\`{e}nes des Groupes de Lie Compacts}, Ann. of Math., vol 57, 115-–207 (1953).
\bibitem[BCDS]{[BCDS]} Brauner, S., Corteel, S., Daugherty, J., Schilling, A., \textsc{Crystal Skeletons: Combinatorics and Axioms}, pre-print, https://arxiv.org/abs/2503.14782 (2025).
\bibitem[CEF]{[CEF]} Church, T., Ellenberg, J., Farb, B., \textsc{FI-Modules and Stability for Representations of Symmetric Groups}, Duke Math. J., vol 164 no. 9, 1833–-1910 (2015).
\bibitem[CF]{[CF]} Church, T., Ellenberg, J., Farb, B., \textsc{Representations Theory and Homological Stability}, Adv. Math., vol 245, 250–-314 (2015).
\bibitem[CZ]{[CZ]} Cohen, A., Zemel, S., \textsc{Polynomial Expressions for the Dimensions of the Representations of Symmetric Groups and Restricted Standard Young Tableaux}, pre-print, https://arxiv.org/abs/2410.16758 (2024).
\bibitem[Fa]{[Fa]} Farb, B., \textsc{Representation Stability}, Proceedings of the ICM, Seoul, vol II, 1173–-1196 (2014).
\bibitem[GR]{[GR]} Gillespie, M., Rhoades, B., \textsc{Higher Specht Bases for Generalizations of the Coinvariant Ring}, Ann. Comb., vol 25, 51-–77 (2021).
\bibitem[HRS]{[HRS]} Haglund, J., Rhoades, B., Shimozono, M., \textsc{Ordered Set Partitions, Generalized Coinvariant Algebras, and the Delta Conjecture}, Adv. Math., vol 329, 851-–915 (2018).
\bibitem[HRW]{[HRW]} Haglund, J., Remmel, J., B. Wilson, A. T., \textsc{The Delta Conjecture}, Trans. Amer. Math. Soc., vol 370, 4029–-4057 (2018).
\bibitem[M]{[M]} Murphy, G. E., \textsc{A New Construction of Young's Seminormal Representation of the Symmetric Group}, J. Algebra, vol 69 issue 2, 287–-297 (1981).
\bibitem[PR]{[PR]} Pawlowski, B., Rhoades, B., \textsc{A Flag Variety for the Delta Conjecture}, Trans. Amer. Math. Soc., vol 372, 8195–-8248 (2019).
\bibitem[Pe]{[Pe]} Peel, H., \textsc{Specht modules and symmetric groups}, J. Algebra, vol 36 issue 1, 88–-97 (1975).
\bibitem[Pu]{[Pu]} Putman, A., \textsc{Stability in the Homology of Congruence Subgroups}, Invent. Math., vol 202, 987–-1027 (2015).
\bibitem[RW]{[RW]} Remmel, J., B. Wilson, A. T., \textsc{An Extension of MacMahon's Equidistribution Theorem to Ordered Set Partitions}, J. Combin. Theory, Ser. A, vol 134, 242–-277 (2015).
\bibitem[Sa]{[Sa]} Sagan, B. E., \textsc{The Symmetric Group: Representations, Combinatorial Algorithms, and Symmetric Functions}, 2nd edition, Graduate Texts in Mathematics 203, Springer, New York, New York, xv+238pp (2000).
\bibitem[SS]{[SS]} Sam, S. V.,, Snowden, A., \textsc{Stability Patterns in Representation Theory}, Forum Math. Sigma, vol 3, e11, 1–-108 (2015).
\bibitem[St1]{[St1]} Stanley, R. P., \textsc{Enumerative Combinatorics, Volume 1}, 2nd edition, Cambridge Studies in Advanced Mathematics 49, Cambridge University Press, xiv+626pp (2012).
\bibitem[St2]{[St2]} Stanley, R. P., \textsc{Enumerative Combinatorics, Volume 2}, 2nd edition, Cambridge Studies in Advanced Mathematics 62, Cambridge University Press, xvi+783pp (2024).
\bibitem[Ste]{[Ste]} Stembridge, J. R., \textsc{Orthogonal Sets of Young Symmetrizers}, Adv. Math, vol 46 issues 1--4, 576–-582 (2011).
\bibitem[vL]{[vL]} van Leeuwen, M. A. A., \textsc{The Littlewood--Richardson Rule, and Related Combinatorics}, in: \textsc{Interaction of Combinatorics and Representation Theory}, eds. Stembridge, J. R., Thibon, J.-Y., van Leeuwen, M. A. A., Mathematical Society of Japan Memoirs 11, The Mathematical Society of Japan, vii+145pp (2001).
\bibitem[Z1]{[Z1]} Zemel, S., \textsc{Compatibility of Higher Specht Polynomials and Decompositions of Representations}, pre-print (2025).
\bibitem[Z2]{[Z2]} Zemel, S., \textsc{Stable Higher Specht Polynomials and Representations of Infinite Symmetric Groups}, pre-print (2025).

\end{thebibliography}
\end{document}